\documentclass[a4paper,12pt]{amsart}

\usepackage[english]{babel}

\usepackage{amssymb,amsfonts,amsthm,amsmath,latexsym}
\usepackage{enumerate, caption, array, esint, bm}
\usepackage{hyperref}\hypersetup{final}
\usepackage{expl3,xparse}

\usepackage{pgfplots}
\pgfplotsset{compat=1.15}
\usepackage{mathrsfs}
\usetikzlibrary{arrows}
\usetikzlibrary{external}
\usepackage{subcaption}

\usepackage[margin=1.1in]{geometry}

\theoremstyle{plain}
\newtheorem{theorem}{Theorem}[section]
\newtheorem{corollary}[theorem]{Corollary}
\newtheorem{definition}[theorem]{Definition}
\newtheorem{lemma}[theorem]{Lemma}
\newtheorem{proposition}[theorem]{Proposition}

\newtheorem{hypothesis}{Hypothesis}

\theoremstyle{remark}
\newtheorem{remark}[theorem]{Remark}
\newtheorem{example}[theorem]{Example}

\newcommand\numberthis{\stepcounter{equation}\tag{\theequation}}

\newcommand{\seq}{\subseteq}
\renewcommand{\setminus}{\smallsetminus}
\newcommand{\ssum}[1]{\sum_{\substack{#1}}}

\newcommand{\e}{{\rm e}}

\newcommand{\df}{\mathop{}\!\mathrm{d}}
\newcommand{\eps}{{\varepsilon}}
\renewcommand{\emptyset}{\varnothing}
\let\C\relax
\let\P\relax
\let\H\relax

\newcommand{\interval}[2]{\draw ({cos(#1)}, {sin(#1)}) arc ({270+#1}:{90+#2}:{tan((#2-#1)/2)})}
\newcommand{\intervalpartial}[3]{\draw[densely dotted] ({cos(#1)}, {sin(#1)}) arc ({270+#1}:{270-180*#3+#3*#2+(1-#3)*#1}:{tan((#2-#1)/2)})}

\newcommand{\angeodesic}[2]{
  \pgfmathsetmacro{\thetaA}{#1}
  \pgfmathsetmacro{\thetaB}{#2}
  \pgfmathsetmacro{\diff}{abs(\thetaA-\thetaB)}
  \pgfmathsetmacro{\halfdiff}{\diff>180 ? (360-\diff)/2 : \diff/2}
  \pgfmathsetmacro{\halfsum}{\diff>180 ? (\thetaA+\thetaB)/2 + 180 : (\thetaA+\thetaB)/2}
  \pgfmathsetmacro{\R}{1/cos(\halfdiff)}
  \pgfmathsetmacro{\r}{abs(tan(\halfdiff))}
  \draw[color=black, line width=0.5pt] (\halfsum:\R) circle (\r);
}

\ExplSyntaxOn
\NewDocumentCommand{\makebb}{m}{
  \tl_map_inline:nn {#1} { \__mk_bb_cmd:n {##1} }
}
\cs_new_protected:Npn \__mk_bb_cmd:n #1
{
  \regex_match:nnTF { \A [A-Z] \z } {#1}
  { \cs_set:cpn { #1 } { \mathbb{#1} } }
  { } 
}
\NewDocumentCommand{\makescr}{m}{
  \tl_map_inline:nn {#1} { \__mk_scr_cmd:n {##1} }
}
\cs_new_protected:Npn \__mk_scr_cmd:n #1
{
  \regex_match:nnTF { \A [A-Z] \z } {#1}
  { \cs_set:cpn { #1#1} { \mathscr{#1} } } { }
}
\NewDocumentCommand{\makecal}{m}{
  \tl_map_inline:nn {#1} { \__mk_cal_cmd:n {##1} }
}
\cs_new_protected:Npn \__mk_cal_cmd:n #1
{
  \regex_match:nnTF { \A [A-Z] \z } {#1}
  { \cs_set:cpn { c#1} { \mathcal{#1} } } { }
}
\ExplSyntaxOff

\makebb{EPHCDKRQUZNT}
\makescr{CKLWNPT}
\makecal{ABCDEFGHIJLOPMRSTUNW}

\newcommand{\1}{{\mathbf 1}}

\newcommand{\DcF}{{\mathcal D}}
\newcommand{\cFAMU}{\cF_{\operatorname{AMU}}}

\newcommand{\vpsi}{{\bm \psi}}
\newcommand{\vchi}{{\bm \chi}}

\newcommand{\cb}{{\mathfrak b}}
\newcommand{\qf}{{\mathfrak f}}
\newcommand{\fS}{{\mathfrak S}}
\newcommand{\fd}{{\mathfrak d}}
\newcommand{\fm}{{\mathfrak m}}

\newcommand{\dD}{\partial\D}

\DeclareMathOperator{\dist}{dist}
\DeclareMathOperator{\im}{Im}
\DeclareMathOperator{\Dom}{Dom}

\DeclareMathOperator{\id}{id}
\DeclareMathOperator{\Id}{Id}
\DeclareMathOperator{\card}{card}
\DeclareMathOperator{\den}{den}
\DeclareMathOperator{\vol}{vol}
\DeclareMathOperator{\supp}{supp}
\DeclareMathOperator{\GL}{GL}
\DeclareMathOperator{\SL}{SL}
\DeclareMathOperator{\SU}{SU}

\DeclareMathOperator{\meas}{meas}
\DeclareMathOperator{\spr}{spr}

\newcommand{\vth}{{\vartheta}}
\renewcommand{\tilde}{\widetilde}
\renewcommand{\bar}{\overline}
\renewcommand{\hat}{\widehat}
\newcommand{\hide}[1]{{}}

\newcommand{\norm}[1]{{\left\| {#1} \right\|}}

\newcommand{\vt}{{\bm t}}

\newcommand{\nvt}{|\vt|}
\newcommand{\vphi}{{\bm \phi}}
\newcommand{\inpd}[1]{{\left\langle #1 \right\rangle}}

\newcommand{\abs}[1]{{\left| {#1} \right|}}

\renewcommand\Re{\operatorname{\mathfrak{Re}}}
\renewcommand\Im{\operatorname{\mathfrak{Im}}}

\newcommand{\pmat}[1]{\begin{pmatrix} #1 \end{pmatrix}}
\newcommand{\smat}[1]{\left(\begin{smallmatrix} #1 \end{smallmatrix}\right)}

\numberwithin{equation}{section}

\title{Limit theorems for cuspidal orbits on Fuchsian groups}

\date{\today}

\author{Sandro Bettin}
\address{Dipartimento di Matematica, Università di Genova, via Dodecaneso 35, 16146 Genova, Italy}
\email{bettin@dima.unige.it}

\author{Sary Drappeau}
\address{Laboratoire de Mathématiques Blaise Pascal, Université Clermont-Auvergne, Institut Universitaire de France, 3 place Vasarely, 63178 Aubière, France}
\email{sary.drappeau@uca.fr}

\author{Jungwon Lee}
\address{Caltech, Division of PMA, 
1200 E California Blvd, Pasadena, CA 91125, USA}
\email{jungwon@caltech.edu}

\begin{document}

\begin{abstract}
  Let~$\Gamma\subset \SL(2, \R)$ be discrete, co-finite with cusps, and let~$\cb\in \R\cup\{\infty\}$ be a cusp of~$\Gamma$. Let~$f: \Gamma\cb \to \R$ be a quantum modular form in the sense that the maps
  $$ \phi_\gamma(x) = f(x) - f(\gamma x) \qquad (x\in \Gamma\cb, \gamma \in \Gamma) $$
  extend to Hölder-regular functions. Under bounded moments assumptions of the shape~$\sum_n n^{-2} \int \abs{\phi_{\gamma^n}(x)}^\alpha \df x < \infty$ for some~$\alpha>0$ and every fixed parabolic element~$\gamma\in \Gamma$, 
  we prove asymptotic formulas for the characteristic functions of the values~$f(x)$ for~$x\in \Gamma\cb$, ordered by growing denominators. These yield Gaussian limits when~$\alpha>2$ and other stable limit laws
  under suitable further assumptions. This generalizes to Fuchsian groups the work of Baladi-Vallée on rational orbits of the Gauss map.

  In companion papers, we use this result, and in particular its flexible assumptions on~$\phi_\gamma$, to deduce limit theorems for automorphic $L$-series at the central point.

  One ingredient in our proof is a particular cuspidal acceleration of the Bowen-Series map, whose inverse branches admit a convenient parametrization, and for which we describe explicitly a natural extension.
\end{abstract}

\maketitle

\tableofcontents

\section{Introduction}

Given an irrationnal number~$x\in (0, 1)\cap (\R \setminus \Q)$, the statistical behaviour of its continued fraction coefficients~$(a_j(x))$ in
\begin{equation}
  x = \frac1{a_1 + \frac1{a_2 + \dotsb}}.\label{eq:rational-CF}
\end{equation}
is an old theme in diophantine approximation. It is known that choosing~$x$ according to the absolutely continuous measure~$\frac{\df x}{(1+x)\log 2}$, the coefficients~$a_j(x)$ are distributed according to the Gauss-Kuzmin law on~$\N_{>0}$, and that~$(a_j(x), a_k(x))$ tend to be independently distributed as~$\abs{k-j}\to \infty$, to the effect that for sufficiently regular observables~$\phi:\N_{>0}\to\R$, the partial sums
$$ f_N(x) := \frac1N\sum_{n=0}^{N-1} \phi(a_j(x)) $$
tend to be distributed, as~$N\to \infty$, according to a Gaussian law with mean and variance of the order of~$N$. This follows from the dynamical properties of the Gauss map, and more precisely the spectral properties of its transfer operator, see~\cite{Khintchine1963,Broise1996,AaronsonDenker2001}.

Instead of picking~$x$ irrational, choose~$x\in (0, 1)\cap \Q$ rational of denominator~$\den(x) \leq Q$, independently at random among the~$(\frac{3}{\pi^2} + o(1)) Q^2$ such possibilities.
The decomposition~\eqref{eq:rational-CF} is then finite with length~$O(\log Q)$.
It is natural to ask what remains of the statistical independence of~$(a_j(x))_j$, or the limiting distribution of random variables of the shape
\begin{equation}
  f(x) = \sum_j \phi(a_j(x))\label{eq:f-sumcocycles-intro}
\end{equation}
given a fixed map~$\phi:\N \to \R$.

This question was initially studied due its link with the complexity, in various aspects, of the Euclidean algorithm for GCD~\cite{Hensley1994}.
It was fully resolved in the works of Vallée and Baladi-Vallée~\cite{Vallee2003,BaladiVallee2005}, who relate the problem to the spectra of a two-parameter family of twists of the Gauss-Kuzmin operator.

In~\cite{BettinDrappeau2022,LeeSun2019,DrappeauNordentoft2022}, these ideas were pushed further, motivated by questions of a different kind, namely the value distribution of trigonometric series formed with coefficients of modular forms,
\begin{equation}
  f(x) = \sum_{n\geq 1} \frac{a_n}{\sqrt{n}} \e^{2\pi i n x}, \qquad (x\in \Q) \label{eq:L-series}
\end{equation}
where~$a_n$ are Fourier coefficients of an automorphic form for a congruence subgroup of~$\SL(2, \Z)$.
The main observation which relates the two questions is the fact that the map
\begin{equation}
  h(x) := f(x) - f(-1/x)\label{eq:QMF-rational-intro}
\end{equation}
initially defined on~$\Q\setminus\{0\}$, extends to a continuous map on~$\R$. The relation~\eqref{eq:QMF-rational-intro}, along with the periodicity of~$f$, shows that~$f(x)$ is essentially of the shape~\eqref{eq:f-sumcocycles-intro} with~$\phi(a_j(x))$ replaced by~$h(T^j x)$, where~$T(x) := \{1/x\}$ is the Gauss map. The~$1$-periodic maps for which the period function defined in~\eqref{eq:QMF-rational-intro} satisfies some regularity property (in a purposedly vague sense) are referred to as quantum modular forms in Zagier's seminal work~\cite{Zagier2010}.

\medskip{}

Series of the shape~\eqref{eq:L-series} exist for automorphic forms on Fuchsian groups~$\Gamma$ other than~$\SL(2, \Z)$, and motivated by specific examples we are led wonder about the statistical distribution of those generalizations (Theorem~\ref{th:corollary-gl2} and Corollaries~\ref{cor:r2} and~\ref{cor:theta} below). It is natural to try to extend the proof of~\cite{BettinDrappeau2022,LeeSun2019} to this setting. The main obstacle comes from the dynamical systems arguments used in the proof, which originate from Baladi-Vallée's work~\cite{BaladiVallee2005} on the Gauss map. The purpose of the present paper is to work out a generalization of Baladi-Vallée's work, with flexible hypotheses on the observables, which can be used to understand series of the shape~\eqref{eq:L-series} for automorphic forms on more general Fuchsian groups.

Let~$\Gamma \subset \SL(2,\R)$ be a discrete and co-finite subgroup. Denote~$C(\Gamma) \subset \P^1(\R)$ the set of cusps of~$\Gamma$ (the fixed points of parabolic matrices), and assume that~$\infty$ is a cusp of~$\Gamma$.
The set of cusps is equiped with a notion of denominator, denoted~$\den(x)$, which will be recalled below, and which coincides with the denominator of numbers in~$C(\Gamma) = \Q$ in the case~$\Gamma=\SL(2,\Z)$.
Given a finite set~$S\subset C(\Gamma)$, a map~$f:C(\Gamma)\backslash S \rightarrow \C$ is called a quantum modular form (QMF) of weight~$0$ for~$\Gamma$ if for any~$\gamma=\smat{a &b \\ c & d} \in \Gamma$, the ``error to modularity'' function
\begin{equation*}
  h_\gamma(x):= f(x) - f(\gamma x) , \quad x \in C(\Gamma) \backslash (S \cup \gamma^{-1}S)  
\end{equation*}
satisfies certain regularity property (in this introduction we focus on the trivial multiplier system).

The main result of the present paper can be described informally as follows:

\medskip{}

\emph{Suppose that for some parameters~$\alpha,\kappa_0, \lambda>0$, the maps~$(h_\gamma)$ extend to~$\kappa_0$-Hölder continuous maps on~$\R\setminus\{\gamma^{-1}\infty\}$, in such a way that for every fixed parabolic matrix~$\gamma\in \Gamma$ we have
  $$ \sum_{n\geq 1} \frac1{n^2}(\|h_{\gamma^n}\|_\infty^\alpha + \|h_{\gamma^n}\|_{(\kappa_0)}^\lambda) < \infty $$
  where~$\|\cdot\|_{(\kappa_0)}$ denotes the Hölder semi-norm. Then as~$x$ varies among elements of~$C(\Gamma)$ of denominators at most~$Q$, and~$Q\to \infty$, the values~$f(x)$ tend, under suitable regularity assumptions, to equidistribute according to a stable law, whose parameters are determined from the asymptotic behaviour of~$h_{\gamma^n}(x)$ in certain ranges in~$n$ and~$x$.}

\medskip{}

We now introduce precisely our notations. The precise statement of our main theorem will be presented in Section~\ref{sec:mainst}, with an overview of its proof given in Section~\ref{sec:overview}.

\subsection{Disk model and assumptions on the group}\label{sec:disk-model}

First we switch to the disk model for the hyperbolic surface, which will be more convenient for our later arguments. Let~$\D$ be the open unit disc, and
$$ p_\D : \overline\C \to \overline\C,  \qquad p_\D(z) := \frac{z-i}{z+i}. $$
This maps conformally the set~$\H$ to~$\D$. We also denote~$p_\D$ the element of~$\GL(2, \C)$ which corresponds to the homography~$p_\D$.

We consider a discrete subgroup~$G \subset \SU(1, 1) = \big\{ \big(\begin{smallmatrix} \alpha & \bar\beta \\ \beta & \bar\alpha \end{smallmatrix}\big), \abs{\alpha}^2 - \abs{\beta}^2 = 1 \big\}$. We assume that~$G$ satisfies the following hypotheses:
\begin{enumerate}[(a)]
  \item $G$ is finitely generated.
  \item The surface~$G\backslash \D$ has finite volume.
  \item There is a cusp at~$1$.
  \item There is no elliptic element in~$G$.
  \item $-\id \notin G$.
\end{enumerate}
The main assumptions are~(a)-(c). The main point of condition~(c) is to ensure that we have at least one cusp, with respect to which the denominators of other cusps will be defined.
The fact that the cusp is at~$1$ simplifies the presentation of some of our arguments, and can be assumed without loss of generality by conjugation.
In all applications we consider, the remaining assumptions (d) and (e) will be implemented without loss of generality by restricting to a suitable subgroup; they are used to simplify some aspects of our arguments.

\subsection{Orbits, denominators and a probability space}

Let~$C(G)$ be the set of cusps of~$G$ on~$\dD$. Consider a cusp~$\cb\in C(G)$, with an associated scaling matrix~$s_\cb \in \SU(1, 1)$, meaning that the stabilizer~$G_\cb$ of~$\cb$ in~$G$ satisfies
$$ \cb = s_\cb 1, \qquad G_\cb = s_\cb \big\{\eta_\cb^k
 \big(\begin{smallmatrix} 1+ki/2 & -ki/2 \\ ki/2 & 1-ki/2 \end{smallmatrix} \big), k \in \Z\big\}s_\cb^{-1}, $$
 where $\eta_\cb\in\{\pm1\}$ is the sign of the trace of a generator of $G_\cb$.
 Given~$x\in G\cb$, we define a map by setting
 \begin{equation}
   \Upsilon(x) = \Upsilon_\cb^G(x) := \abs{\alpha + \beta}, \qquad \big(x = g\cb,\quad g\in G,\quad gs_\cb = \smat{\alpha & \beta \\ * & *}\big). \label{eq:def-Y}
 \end{equation}
From now on, we will keep~$G$ and~$\cb$ fixed and drop them from the notation.

A routine computation shows that~$\abs{\alpha + \beta}$ depends only the class~$g\in G / G_\cb$, which justifies the notation~$\Upsilon(x)$. Moreover, the quantity
\begin{equation}
  q(x) := \frac{\abs{x-1}}2\Upsilon(x) \label{eq:def-qx}
\end{equation}
depends only on the class~$x$ modulo left-translation by~$G_1$.

Next, we let
$$ \psi : \dD \to \R_+ $$
be a smooth and not identically zero test function. We assume that~$1\not\in\supp\psi$.
We define a set, depending on~$\cb$, by
$$ \Omega_Q := \{x\in G\cb, 0<q(x) \leq Q\}, \qquad (Q\geq 2). $$

Translating to the upper-half plane model, writing~$gs_\cb = p_\D \smat{a&b\\c&d} p_\D^{-1}$ for $g\in G$ we get
\begin{equation}
  \Upsilon(x)=\sqrt{a^2+c^2},  \qquad q(x)=|c|.\label{eq:Upsilon-SL2R}
\end{equation}
The set~$G_1 \backslash \Omega_Q$ thus corresponds, in the disk model, to the set~$T_{\infty\cb}(Q)$ defined in Sections~2.2 and~3.3 of~\cite{PetridisRisager2018} up to scaling at~$\infty$: more precisely, denoting~$\Gamma_\infty := p_\D^{-1}G_1 p_\D= \{\eta_1^k\smat{1&k\omega\\0&1} , k\in\Z\}$ with~$\omega>0$, we check that\footnote{Note that~$T_{\infty, \cb}$ in the right-hand side depends on a choice of scaling matrix~$\sigma_\infty$ at~$\infty$, although the left-hand side does not; the relation~\eqref{eq:rel-OmegaQ-T} as stated holds with the choice~$\sigma_\infty = \smat{\sqrt{\omega}&0\\0&1/\sqrt{\omega}}$.}
\begin{equation}
  G_1 \backslash \Omega_Q = \omega T_{\infty, \cb}(Q\sqrt{\omega}).\label{eq:rel-OmegaQ-T}
\end{equation}
By~\cite[Theorem~3.8]{PetridisRisager2018}, the set~$G_1\backslash \Omega_Q$ is finite.
On the other hand, since~$\psi$ vanishes on a neighborhood of~$1$, we have
$$ \sup_{x\in \dD} \card\{\tau\in G_1, \psi(\tau x)\neq 0\} < \infty. $$
Summing over a set of representatives of~$G_1 \backslash \Omega_Q$, we deduce that
$$ \card\{x\in \Omega_Q: \psi(x) \neq 0\} < \infty. $$

Using a result from~\cite{PetridisRisager2018}, we estimate the total mass of~$\psi$ on~$\Omega_Q$ by
\begin{equation}
  \sum_{x\in \Omega_Q} \psi(x) \sim c_\psi Q^2\label{eq:psi-mass-rough}
\end{equation}
as~$Q\to\infty$, for some explicit constant~$c_\psi > 0$ depending only on~$\psi$ and~$G$.
This will be made precise in Lemma~\ref{lem:psi-totalmass} below.

In view of the above, we may define, for~$Q$ large enough, a probability measure on~$\Omega_Q$ proportional to~$\psi$.
We denote by~$\P_Q$ and~$\E_Q$ the associated probability and expectation, so that for any map~$f$ on~$C(G)$, we have
$$ \E_Q[f(x)] = \frac{\sum_{x\in\Omega_Q} \psi(x) f(x)}{\sum_{x\in\Omega_Q} \psi(x)}. $$
\begin{remark}
  In~\cite{BaladiVallee2005,BettinDrappeau2022,LeeSun2019}, the average is instead performed over~$\Omega_Q \cap I$ where~$I$ is a fundamental interval for the action of~$G_1$, \emph{e.g.} the unit interval in the upper-half plane model. For $G_1$-invariant observables, we can recover this case by taking~$\psi$ such that~$\sum_{\tau\in G_1} \psi(\tau x) = 1$ (a smooth partition of unity), since~$\Omega_Q$ is invariant under left translation by~$G_1$.
  
Our method also applies to the case where cusps are ordered by~$\Upsilon(x)$, as in~\cite{PetridisRisager2004}, rather than by~$q(x)$. The only change needed in the argument is to remove the factor~$|(x-1)/2|^{-2s}$ in~\eqref{eq:def-psis}.
\end{remark}

\subsection{Fundamental domain}\label{sec:fundamental-domain}

Let~$R$ be a fundamental domain in~$\D$ under the action of~$G$, whose sides are non-intersecting bi-infinite geodesics.
The fact that such a domain~$R$ exists is proved in~\cite[Proposition p.~15]{Tukia1973} under the hypotheses (a)-(d).
Since~$R$ has finite volume, by Siegel's Theorem~\cite[Theorem~4.1.1]{Katok1992}, the closure of~$R$ in~$\overline{\D} = \D\cup \dD$ intersects~$\dD$ at a finite set~$V$, and therefore~$R$ is an ideal polygon.

A side~$s$ of~$R$ is paired with a unique side of~$R$, say~$\hat s$, by an element of~$g_s\in G\setminus\{1\}$, in such a way that
$$ s = g_s \hat s. $$
We gather in Section~\ref{sec:usef-prop-fund} below some useful facts about this pairing. In particular, the map~$s\mapsto \hat s$ is an involution with no fixed point. Given~$s$, let~$I(s)\subset \dD$ be the closed interval whose ends are the vertices delimiting~$s$, and which does not intersect other vertices of~$R$. The intervals~$\{ I(s) \}$ then form a partition of~$\dD$ up to vertices.

For each side~$s$, we define the subset~$L_s \subset \dD$ of the circle, consisting of all the arcs which do not touch~$I(\hat s)$:
$$ L_s := \bigcup_{s': I(s') \cap I(\hat s) = \varnothing} I(s'). $$
Under the action of~$g_s$, the set~$L_s$ is mapped into the interior of~$I(s)$. We let
$$ K_s := g_s(L_s), \qquad K := \bigcup_{s:\text{ side}} K_s. $$
Each~$K_s$ is an interval, bounded away from the vertices of~$R$, and the set~$(K_s)$ are disjoint.
In Figure~\ref{fig:example-Gamma3} on page~\pageref{fig:example-Gamma3}, we draw a fundamental domain for the principal congruence subgroup~$\Gamma(3)$.

\subsection{Multiplier system}

Fix an integer~$d\geq 1$ and let~$\Theta\subset \GL(d, \R)$ be a finite subgroup. We also let
\begin{align}\label{dfd'}
  \dD':=\dD\setminus V.
\end{align}
We consider a family of maps indexed by the elements of~$G$,
\begin{equation}
  (\vth_g)_{g\in G}, \qquad \vth_g : \dD' \to \Theta,\label{eq:def-vthg}
\end{equation}
satisfying the cocycle property, or derivative rule:
\begin{equation}
  \vth_{g_1 g_2}(x) = \vth_{g_2}(x) \vth_{g_1}(g_2 x), \qquad (x,\ g_2x \notin V).\label{eq:vth-cocycle}
\end{equation}

We assume two properties on the map~$\vth_g$, which are informally:
\begin{itemize}
  \item The map~$\vth_g$ being constant on each arc~$I(s)$, for each~$g \in G$;
  \item A form of strong mixing, in the sense that any value~$\vth_0\in \Theta$ is in the range of~$\vth_g$ with some constraints on the shape of~$g$.
\end{itemize}
The precise assumptions will be described in Section~\ref{sec:skew-products} below.

\subsection{Functional space}\label{fsp0}

The following norms will be relevant for our later arguments.

\begin{definition}[Piecewise~$\kappa$-Hölder space]
  Given~$\kappa\in (0, 1]$, let~$H^\kappa$ be the set of bounded maps~$f:K \times\Theta \to \R^d$ which have a finite~$\kappa$-Hölder semi-norm
  $$ \| f \|_{(\kappa)} := \sup_{\substack{s:\text{ side} \\ x, y \in K_s \\ \vth_0\in \Theta}} \frac{\abs{f(x, \vth_0) - f(y, \vth_0)}}{\abs{x-y}^\kappa} < \infty.$$
\end{definition}
See also~\cite[section~5.1]{Evans2010}. We equip~$H^\kappa$ with two norms:
\begin{align*}
  &\| f \|_\infty := \sup_{(x,\vth_0) \in K} |f(x,\vth_0)|, && \| f \|_\kappa := \| f \|_\infty + \| f \|_{(\kappa)}.
\end{align*}
With slight abuse of notation, we use the same notations for~$H^\kappa$ and its norms also for functions on $K$, extending them trivially to $K\times\Theta$.

\subsection{Quantum modular forms}

Consider a family of maps from~$K$ to~$\R^d$,
$$ (\vphi_g)_{g\in G} \subset H^{\kappa_0} $$
for some $\kappa_0>0$. We are interested in maps
$$ \qf : C(G) \to \R^d $$
which are quantum modular forms of weight~$0$ in the following sense.\footnote{For the behaviour and distribution of quantum modular forms of non-zero weight we refer to~\cite{BettinDrappeau2025}.}
\begin{hypothesis} \label{def:qmf}
  For all~$g\in G$, we have
  \begin{equation} 
    \vth_g(x) \qf(gx) = \qf(x) + \vphi_g(x), \qquad (x\in C(G)\cap K, g\in G) . \label{eq:qmf}
  \end{equation}
  Moreover, 
  \begin{enumerate}[(a)]
    \item the multiplier~$(\vth_g)$ satisfies the stability and mixing hypotheses described in Section~\ref{sec:skew-products} below;
    \item there exist~$\kappa_0, \alpha, \lambda>0$ such that the period functions~$(\vphi_g)$ belong to~$H^{\kappa_0}$ and have moments and Hölder-norm moments of order~$\alpha$ and~$\lambda$ respectively, where these moments are defined in Section~\ref{sec:moments-conditions} below.
  \end{enumerate}
\end{hypothesis}

We note that this hypothesis as stated is stronger than needed: the actual hypothesis concerns only a subsequence of~$(\vphi_g)$ along specific matrices related to~$K$; we defer to Sections~\ref{cfn} and~\ref{sec:moments-conditions} for the precise hypothesis and set-up.

Hypothesis~\ref{def:qmf} implies that the map~$(g, x) \mapsto \vphi_g(x)$ is a cocycle in the sense that
\begin{align}\label{xax}
  \vphi_{g_1 g_2}(x) = \vphi_{g_2}(x) + \vth_{g_2}(x) \vphi_{g_1}(g_2x)
\end{align}
for all~$g_1, g_2 \in G$, $x,g_2x\in K$.

\subsection{Main statement}\label{sec:mainst}

Our main result is an asymptotic estimate for the characteristic function of~$\qf(x)$, viewed as a random variable on~$\Omega_Q$, as~$Q\to \infty$.

For a vector~$\vt \in \R^d$, we let~$\nvt$ be the Euclidean norm of~$\vt$ and denote by $\inpd{\cdot,\cdot}$ the corresponding scalar product. 

\begin{theorem}\label{th:limitlaw:quasi}
  Assume~$\qf$ satisfies Hypothesis~\ref{def:qmf} with parameters~$\kappa_0,\alpha,\lambda$.
  Then there exist~$t_0, \delta_0>0$, and two maps
  $$ U_1, U_2 : [-t_0, t_0]^d \to \C, $$
  such that for all~$\vt\in \R^d$ with~$\nvt<t_0$, we have
  $$ \E_Q(\e^{ i \inpd{\vt, \qf(x)}}) = \exp\Big(U_1(\vt) \log Q + U_2(\vt) + O(Q^{-\delta_0}) \Big), $$
  where the implicit constant depends only on~$\kappa_0$, the values of the moments in Hypothesis~\ref{def:qmf}, the finite group~$\Theta$, and an upper bound for $|\qf|$ at finitely many cusps, as specified in Remark~\ref{fire}.

  Additionally, there is a constant~$\fd>0$ independent from~$\qf$ for which the following holds. There exists a random variable~$\Phi \in \R^d$ such that for all~$\eps\in(0,\alpha)$, we have
  $$ \E(\abs{\Phi}^{\alpha - \eps}) < \infty, $$
  and the following holds:
  \begin{itemize}
    \item The maps~$U_1, U_2$ satisfy
    \begin{align*}
      U_1(\vt) ={}& 2\fd^{-1}\E(\e^{i\inpd{\vt, \Phi}} - 1) + O_\eps(\nvt^2 + \nvt^{2\alpha - 2\eps}), \\
      U_2(\vt) ={}& O_\eps(\nvt + \nvt^{\alpha -\eps}).
    \end{align*}
    \item Moreover if~$\alpha > 2$, then there exists~$\mu_\Phi \in \R^d$ and a real~$d\times d$ positive semi-definite matrix~$\Sigma_\Phi$ such that
    $$ U_1(\vt) = i \inpd{\vt, \mu_\Phi} - \tfrac12 \vt^T \Sigma_\Phi \vt + O_\eps(\nvt^3 + \nvt^{\alpha - \eps}). $$
    We have~$\mu_\Phi = 2\fd^{-1}\E(\Phi)$.
  \end{itemize}
\end{theorem}

If in Hypothesis~\ref{def:qmf}, condition (b) is satisfied with~$\alpha>2$, then Theorem~\ref{th:limitlaw:quasi} implies, via Lévy's continuity theorem, convergence to a non-degenerate Gaussian law of the ratio
$$ \frac{\qf(x) - \mu_\Phi \log Q}{\sqrt{\log Q}} $$
as long as~$\vphi_g(x)$ is not essentially~$\log|\frac{d}{dx}g x|$ for all~$g\in G$, up to bounded telescoping terms (cf.~\cite[Corollary~3.2]{BettinDrappeau2022}).

We will give, in the course of our arguments, explicit formulas for~$\Phi$. However, the expression that we obtain depends on a non-canonical choice of induction process.

\begin{remark}
  In~\cite[Section~5]{BaladiVallee2005}, Baladi and Vallée also prove a local limit theorem for observables which are lattice-valued.
  Some of their arguments could be adapted to our setting, however the need to avoid observables proportional to~$\log$, and the possible combined behaviour in several dimensions, makes it prohibitive to give a simple representative statement. Since we do not have in mind presently any particular arithmetic application which could motivate such a statement, we refrain to carry this out here.
\end{remark}

\subsection{Application to trigonometric series with $GL(2)$ automorphic coefficients}

We briefly mention the applications which motivated Theorem~\ref{th:limitlaw:quasi}. The details will appear in~\cite{BettinEtAla} and~\cite{BettinEtAl3}.

Our main motivating application is to the limiting distribution of the values
\begin{equation}
  L(x) \overset{AC}{:=} \sum_{n=1}^\infty \frac{a(n)}{\sqrt{n}} \e^{2\pi i n x}\label{eq:def-Lx}
\end{equation}
where~$a(n)$ are the Fourier coefficients of a non-zero automorphic form~$\phi$ of integer weight for a cofinite subgroup~$\Gamma\subset \SL(2,\R)$ with a cusp at~$\infty$, and~$x$ varies among cusps of~$\Gamma$. For all such~$x$, the series~\eqref{eq:def-Lx} is defined by means of analytic continuation of the series~$\sum_n a(n) \e^{2\pi i n x} / n^s$ initially defined for large~$\Re(s)$; we write~$AC$ above the~$:=$ sign to indicate this. An important point is that this analytic continuation is highly sensitive to the properties of~$x$ as a cusp of~$\Gamma$. We normalize the Fourier coefficients~$a(n)$ so that the Rankin-Selberg theorem implies~$\sum_{n\leq x} \abs{a(n)}^2 = x^{1 + o(1)}$, and we let~$\lambda_\phi$ denote the Laplace eigenvalue of~$\phi$.

\begin{theorem}\label{th:corollary-gl2}
  Suppose that either~$\phi$ is cuspidal, or~$\lambda_\phi\geq 1/4$. Then there exists~$\sigma_\phi>0$ and~$\beta_\phi\in \{0, 1, 3\}$ such that, as~$x$ varies among cusps of~$\Gamma$ of denominators bounded by~$Q$ and~$Q\to\infty$, the values
  $$ \frac{L(x)}{\sigma_\phi (\log Q)^{1/2} (\log\log Q)^{\beta_\phi/2}} $$
  tend to equidistribute according to a standard complex Gaussian density.
\end{theorem}

This statement generalizes various earlier results notably by Petridis-Risager~\cite{PetridisRisager2018}, Nordentoft~\cite{Nordentoft2021}, Lee-Sun~\cite{LeeSun2019}, Bettin-Drappeau~\cite{BettinDrappeau2022} and Drappeau-Nordentoft~\cite{DrappeauNordentoft2022};  see also the related works of Constantinescu~\cite{Constantinescu2022} and Constantinescu-Nordentoft~\cite{ConstantinescuNordentoft2025}.
The value of~$\beta_\phi$ can be determined from the growth of~$\phi$ at cusps, and we have~$\beta_\phi = 0$ if~$\phi$ is cuspidal. Moreover, when~$\phi$ is cuspidal, we can obtain an explicit expression for~$\sigma_\phi$ in terms of~$\phi$. We refer to~\cite{BettinEtAla} for more details.

Two new examples of sequences~$a(n)$, defined in elementary ways, for which the present work applies, are the following.

\begin{corollary}\label{cor:r2}
  The statement of Theorem~\ref{th:corollary-gl2} holds true, with~$\beta_\phi = 1$, for the series
  $$ L(x) \stackrel{\text{AC}}{=} \sum_{n=1}^\infty \frac{r_2(n)}{\sqrt{n}} \e^{2\pi i n x} $$
  where~$r_2(n)$ is the number of ways that~$n$ can be written as sums of two squares.
\end{corollary}

\begin{corollary}\label{cor:theta}
  The statement of Theorem~\ref{th:corollary-gl2} holds true, with~$\beta_\phi = 0$ and~$\sigma_\phi = 1/\sqrt{2}$, for the quadratic Weyl sum
  $$ L(x) \stackrel{\text{AC}}{=} \sum_{n=1}^\infty \frac{\e^{2\pi i n^2 x}}{\sqrt{n}}. $$
\end{corollary}

The corresponding automorphic forms are an Eisenstein series of weight~$1$ for~$\Gamma_0(4)$, and a Theta series of half-integer weight for the same group, respectively.
In corollaries~\ref{cor:r2} and \ref{cor:theta}, the parameter~$x$ varies among rational numbers in~$[0, 1]$ of usual denominator at most~$Q$. We refer to~\cite{BettinEtAla} for more details and references.

In forthcoming work~\cite{BettinEtAl3}, we will study applications of the present methods to partial Fourier series with automorphic coefficients.

\subsection{Overview}\label{sec:overview}

Theorem~\ref{th:limitlaw:quasi} is a wide extension of Theorem~3.1 of~\cite{BettinDrappeau2022} and Theorem~A~\cite{LeeSun2019}, which in turn generalized work of Baladi and Vallée~\cite{BaladiVallee2005}. We briefly explain the new features in the proof.

The first question is to find a proper analogue of the Gauss map: a Markov, hyperbolic map on~$\dD$, acting locally as homographies associated with elements of~$G$. This problem, stated precisely in these terms, was resolved by Bowen--Series~\cite{BowenSeries1979} using a geometric construction based on a fundamental domain for~$G$. An induction of the Bowen-Series map, which we denote~$\cF_K$, then possesses spectral properties analogous to those of the Gauss map.

Generalizing the arguments of Baladi-Vallée~\cite{BaladiVallee2005} to this setting requires solving several issues. We properly parametrize the coding and domains of inverse branches of the induced map~$\cF_K$; along the way, we take the opportunity to give an explicit expression for the natural extension of~$\cF_K$ and its invariant measure, by exploiting the flexibility we have of choosing a convenient fundamental domain, namely an ideal hyperbolic polygon. Then we are to provide precise estimates for the size of inverse branches. Once this is at hand, the rest of the arguments of Baladi-Vallée, and in particularly the Dolgopyat method (with the strong UNI property), could be generalized without major issue, if one is ready to accept strong assumptions on the ``observables''~$(\vphi_g)$.

One of our main challenges, however, lies in the fact that for our applications in~\cite{BettinEtAla}, the maps~$(\vphi_g)$ are difficult to analyze, for lack of a simple explicit expression. In particular, we are to avoid a hypothesis relying on estimates uniform in~$g$ varying among the inverse branches of~$\cF_K$.
This issue is addressed by considering a new induction~$\cF_c$ of the Bowen-Series map, whose inverse branches, along with their domains, are very easy to characterize: inverse branches of~$\cF_c$ are parametrized by homographies essentially of the shape~$g^n x$, for~$n\geq 1$ arbitrary but~$g$ varying over a \emph{finite} set of parabolic elements of~$G$.
We define this map in Section~\ref{sec:cusp-acceleration} below, and make a precise study of its natural extension and its invariant measure. The original Bowen-Series accelerated map~$\cF_K$ is an induction of the map~$\cF_c$.

The map~$\cF_c$ is closely related to the cuspidal acceleration introduced by Artigiani, Marchese and Ulcigrai in~\cite{ArtigianiEtAl2016} (see also~\cite{Marchese2018}), which we denote~$\cFAMU$ here. The differences between the two settings are discussed precisely in Section~\ref{sec:comparing-three-maps} below.
We have taken the opportunity here to work out, additionally, the natural extension and the invariant measure for~$\cFAMU$.

The map~$\cF_c$ is used in our arguments on two occasions: firstly, as a fundamental tool to establish analytic continuation of a certain generating series (Section~\ref{sec:conv-res-below-1}); and secondly, as a way to effectively describe the random variable~$\Phi$ in Theorem~\ref{th:limitlaw:quasi} (Section~\ref{sec:sub-clt-case}).

\subsubsection*{Organisation of the paper}

We outline the contents as follows.
\begin{itemize}
  \item In Section~\ref{sec:bsmap} and~\ref{sec:bs:accel}, we define the instance of the Bowen--Series map we will use. We then introduce and discuss several accelerations of the Bowen-Series map away from vertices of the fundamental domains, in particular the maps~$\cF_K$ and~$\cF_c$ alluded to in the above overview. We explicitely describe their inverse branches, and construct associated natural extensions and invariant densities.
  \item In Section~\ref{sec:est:apriori}, we recall the main metric properties on inverse branches of the accelerated maps. This includes contraction estimates, and also estimates relative to the verification of the uniform non-integrability conditions (UNI).
  \item In Section~\ref{sec:transfer} we introduce the various twists of the Ruelle-Perron-Frobenius operators~$\LL$ and~$\CC$ associated with~$\cF_K$ and~$\cF_c$, and relate~$\CC$ with~$\LL$.
  \item In Section~\ref{sec:moments-conditions}, we define rigorously the sup-norm and~Hölder-norm moments bounds on~$(\vphi_g)_g$ which constitute Hypothesis~\ref{def:qmf}.(b).
  \item In Section~\ref{sec:transfer:pert}, we show a contraction estimate for large twists of the transfer operator~$\LL$, following the arguments in~\cite{Dolgopyat1998,BaladiVallee2005,BettinDrappeau2022,LeeSun2019}, and in Section~\ref{sec:asympt-behav-lead} we analyse the asymptotic behaviour of the leading eigenvalue.
  \item Finally, in Sections~\ref{sec:continuation-quasiinv} and~\ref{sec:est:mgf}, we define the generating series relevant for our problem and prove its meromorphic continuation to the left of~$\Re(s)=1$, and in Section~\ref{sec:main-results}, we deduce the main theorem~\ref{th:limitlaw:quasi}.
\end{itemize}

\subsection*{Notations}

We use the notation~$\1_A$ for the indicator of a set or a logical condition.
We denote~$f = O(g)$, or equivalently~$f \ll g$, to say that there is a constant~$C>0$ for which~$|f| \leq C g$. The domain of validity will be clear from the context, and the implicit constant~$C$ may depend on the group~$G$, the choice of fundamental domain~$R$, and the fixed finite
group~$\Theta$. In cases where~$C$ depends on other paremeters in the analysis, this dependence will be explicited in subscript, such as~$f \ll_* g$ or~$f = O_*(g)$, or else will be described explicitely.
For all~$z\in \C$, we define~$\e(z) = \e^{2\pi i z}$. Given a map $F$, images under its $n$-th iterate $F^n$ are taken only over points for which $F^n$ is defined.

We let~$\spr(\TT)$ denote the spectral radius of an operator~$\TT$.
We let~$\|f\|_{\infty}$ we indicate the usual sup-norm on the domain of~$f$.
Given a norm or seminorm~$\|\cdot\|_*$ on a functional space, we will write~$\|\cdot\|_{*, D}$ the same norm on the subspace of functions restricted to~$D$. We use~$\abs{\cdot}$ to indicate the usual Euclidean norm on~$\R^d$.

\subsection*{Acknowledgements}

We thank the organizers of the April 2023 MPI Bonn workshop ``Arithmetic statistics, automorphic forms and ergodic methods'' (N.\ Diamantis, M.\ Lee, P. Moree and Y.\ Petridis), and also the DIMA (Univ.\ of Genova), the I2M and CIRM (Univ.\ Aix-Marseille) where part of this work was carried out since 2022.

We thank Luca Marchese, Francesco Veneziano, Benoît Saussol and Jérôme Los for many discussions related to the topic of this work. 

Various LLM's were used to generate the TikZ figures, and to check for errors in an earlier version of this manuscript.

This work benefited from financial support from FWF-I4945-N/ANR-20-CE91-0006, SNF-10.003.145/ANR-24-CE93-0016, H2020-EU.1.1 (grant agreement 787304), ERC-Advanced Grant
833802 Resonances, and from the INdAM group GNAMPA.

During the preparation of this work, S.B. was partially supported by PRIN 2022 ``The arithmetic of motives and $L$-functions'', by the Curiosity Driven grant “Value distribution of quantum modular forms” of the University of Genoa, funded by the European Union -- NextGenerationEU, and by the MIUR Excellence Department Project awarded to Dipartimento di Matematica, Università di Genova, CUP D33C23001110001.

\section{Bowen--Series map} \label{sec:bsmap}

\subsection{Denominators and orbits}

Recall that the quantity~$\Upsilon(x)$ was defined  in~\eqref{eq:def-Y}, for~$x\in G \cb$. Given~$g = \smat{* & * \\ u & v}$, we define
$$ j(g, x) = u x + v. $$
For all~$g_1, g_2 \in G$ and~$x\in \bar\D$, we have
$$ j(g_1g_2, x) = j(g_1, g_2 x) j(g_2, x), $$
see~\cite[eq.~(2.3)]{Iwaniec1997}.
It is then easily seen that
\begin{equation}
  \Upsilon(g x) = \abs{j(g, x)} \Upsilon(x), \qquad (x\in G \cb, g \in G).\label{eq:denom-cocycle}
\end{equation}

\begin{lemma}\label{lem:conv-sum-q}
  For all~$\sigma>1$, we have
  $$ \sum_{x\in G\cb} \Upsilon(x)^{-2\sigma} < \infty. $$
\end{lemma}
\begin{proof}
  We translate the problem to the~$\SL(2, \R)$ setting, where the analogous estimate is essentially~\cite[Proposition~2.8]{Iwaniec2002}. We first assume that~$\cb = 1$ with width~$1$, and take $s_1=\id$. We keep the notation~$g = \smat{\alpha & \beta \\ * & *}$ throughout the computation, and given such~$g$, define~$\gamma = p_\D^{-1} g p_\D\in \SL(2, \R)$, which we write~$\gamma = \smat{a&b\\c&d}$. Let~$\Gamma = p_\D^{-1} G p_\D$, so that~$\Gamma_\infty = p_\D^{-1}G_1p_\D$. 
  
  By~\eqref{eq:Upsilon-SL2R} our sum becomes
  \begin{align*}
    \sum_{g\in G / G_1} \abs{\alpha + \beta}^{-2\sigma}
    ={}& \sum_{\gamma \in \Gamma / \Gamma_\infty} \abs{a - i c}^{-2\sigma} 
    ={} 1 + \ssum{\gamma \in \Gamma_\infty\backslash \Gamma / \Gamma_\infty \\ c\neq 0} \sum_{n\in \Z} \big((a+cn)^2 + c^2\big)^{-\sigma} \\
    \ll {}& 1 + \ssum{\gamma \in \Gamma_\infty\backslash \Gamma / \Gamma_\infty \\ c\neq 0} \abs{c}^{-2\sigma},
  \end{align*}
  where~$\Gamma_\infty = \{\eta_1^k\smat{1&k\\0&1}, k\in\Z\}$.
  At this point we invoke~\cite[eq.~(2.33)]{Iwaniec2002}, to the effect that the sum over~$\gamma$ is bounded as long as~$\sigma>1$.

  Assume now that~$\cb$ is arbitrary and let~$G' := s_\cb^{-1} G s_\cb$. By definition, $1$ is a cusp of~$G'$ of width $1$, and by what we have just proved, we have
  \begin{equation}
    \sum_{x\in G'1} \Upsilon_1^{G'}(x)^{-2\sigma} < \infty, \qquad (\sigma > 1).\label{eq:denom-bd-case1}
  \end{equation}
  Our original sum can be related this one by
  $$ \sum_{x\in G \cb} \Upsilon_\cb^G(x)^{-2\sigma} = \sum_{x \in G' 1} \Upsilon_\cb^G(s_\cb x)^{-2\sigma} = \sum_{x \in G' 1} \Upsilon_1^{G'}(x)^{-2\sigma} \abs{j(s_\cb, x)}^{-2\sigma}, $$
  where the last step follows by a computation identical to~\eqref{eq:denom-cocycle}.
  Here, it is easily checked that the map~$j(s_\cb, \cdot)$ is bounded and non-zero on the circle, so that~$\abs{j(s_\cb, x)} \asymp 1$. This along with the bound~\eqref{eq:denom-bd-case1} concludes the argument.
\end{proof}

Finally, we make precise the estimate~\eqref{eq:psi-mass-rough} referred in the introduction.

\begin{lemma}\label{lem:psi-totalmass}
  For some~$\delta>0$ the following holds. Let~$\psi$ be smooth on~$\dD$ with~$1\not\in\supp(\psi)$. Then
  $$ \sum_{x\in \Omega_Q} \psi(x) = \frac{Q^2}{\pi \vol(G\backslash\D)} \int_{\dD} \psi(z) \frac{2\abs{\df z}}{\abs{1-z}^2} + O_\psi(Q^{2-\delta}) \qquad (Q>0). $$
\end{lemma}
\begin{proof}
  We split the sum over~$\Omega_Q$ into cosets modulo the action of~$G_1 = p_\D \Gamma_\infty p_\H$,  getting
  \begin{align*}
    \sum_{x\in\Omega_Q} \psi(x) = {} \sum_{x\in G_1 \backslash \Omega_Q}\sum_{\tau\in G_1}  \psi(\tau x) 
    = {} \sum_{r\in T_{\infty\cb}(Q\sqrt{\omega})} \sum_{\tau \in \Gamma_\infty} \psi(p_\D \tau \omega r) 
    = {} \sum_{r\in T_{\infty\cb}(Q\sqrt{\omega})} \sum_{n\in\Z} \psi_*(r+n),
  \end{align*}
  where~$\psi_*(t) := \psi(p_\D \omega t)$.

  The map~$h:\R/\Z \to \R_+$ defined by~$h(r) := \sum_{n\in \Z} \psi_*(r+n)$ is smooth. It satisfies
  $$ \sum_{n\in\Z} \bigl|\hat{\psi_*}(n)\bigr| (1+\abs{n})^{1/2} < \infty $$
  where~$\hat{\psi_*}$ is the Fourier transform of~$\psi_*$. We therefore have, by~\cite[Theorem~3.8]{PetridisRisager2018},
  $$ \sum_{x\in \Omega_Q} \psi(x) = \sum_{r\in T_{\infty\cb}(Q\sqrt{\omega})} h(r) = \frac{\omega Q^2}{\pi \vol(G\backslash\D)} \int_{\R/\Z} h(r) \df r + O_\psi(Q^{2-\delta}). $$
  Inserting the definition of~$h$ and changing variables in the integral yields the claimed main term.
\end{proof}

\subsection{Useful properties of the fundamental domain}\label{sec:usef-prop-fund}

Let~$R$ be a fundamental domain in~$\D$ under the action of~$G$, whose sides are non-intersecting geodesics. The fact that such a domain~$R$ exists is proved in~\cite{Tukia1973} under the hypotheses (a)-(d).
Since~$R$ has finite volume, by Siegel's Theorem~\cite[Theorem~4.1.1]{Katok1992}, the closure of~$R$ in~$\overline{\D} = \D\cup \dD$ intersects~$\dD$ at a finite set of points~$V$, and therefore~$R$ is an ideal polygon.

Each side~$s$ of~$R$ is paired with a unique side of~$R$, say~$\hat s$, by an element of~$g_s\in G\setminus\{1\}$, in such a way that
$$ s = g_s \hat s. $$
Under this pairing, we have~$\hat s \neq s$. Indeed, if~$g_s$ leaves~$s$ invariant then it either swaps or fixes its ends. In the first case, $g_s$ would have a fixed point inside~$\D$, which contradicts our hypotheses on~$G$. In the second, $g_s$ would fix both ends of~$s$, which are cusps, thus implying that~$g_s=1$, hence a contradiction. Since clearly~$\hat{\hat s\hspace{0.1em}} = s$, the map~$s\mapsto \hat s$ is thus an involution with no fixed point. In particular, the number of sides of~$R$ is even, say~$2n$ with~$n\geq 2$.

We label the sides of~$R$ by the alphabet~$\Sigma \simeq \Z / 2n\Z$ proceeding counter-clockwise, so that two consecutive elements~$i, i+1\in \Z / 2n\Z$ label consecutive sides~$s_i, s_{i+1}$. For~$a\in\Sigma$, let~$s_a$ be the side associated to~$a$, and define~$[a]\subset \dD$ be the closed arc whose ends are the vertices delimiting~$s_a$, and which does not intersect other vertices of~$R$. The arcs~$\{ [a]\}_{a\in\Sigma}$ then form a partition of~$\dD$ up to vertices.

Given~$a\in\Sigma$, define~$\hat a$ to be the label associated to~$ \hat s_a,$ and~$g_a := g_{s_a}$.
We note that~$[a] = \overline{g_a (\dD \setminus [\hat a])}. $
More generally, given a finite word~$w\in \Sigma^m$, we let~$\abs{w} := m$, and we say that~$w$ is admissible if~$a_{i+1} \neq \hat{a_i}$ for all~$i\in \{1, \cdots, m-1\}$. The empty word is denoted by~$\epsilon$.
We define the set of finite admissible words
$$ \Sigma^* = \{w\in \cup_{m \geq 0}\Sigma^{m}, w \text{ is admissible}\}. $$
Given~$w=a_1\dots a_n\in\Sigma^*$, we let~$g_w=g_{a_1}\cdots g_{a_n}$. Since~$G$ is a free group and the set~$\{g_s\}_s$ consists of a basis and its inverses, every~$g\in G$ can be written in a unique way as~$g=g_w$ for some~$w\in \Sigma^*$.
Also, for~$w\in \Sigma^*$, we define the fundamental interval associated to~$w$ as
\begin{align*}
  [w] = \begin{cases} \dD, & (\abs{w}=0), \\
    [a], & (\abs{w}=1, w=a), \\
    g_{a_1} g_{a_2} \dotsc g_{a_{n-1}}[a_n], & (w = a_1 \dotsb a_n).
  \end{cases}
\end{align*}
Note that if~$w = aw'$ with~$a\in\Sigma$ and~$w'\in \Sigma^*\setminus\{\epsilon\}$ not starting with~$\hat a$, then~$[w] = g_a [w']$. Moreover, an easy induction shows that the admissibility condition ensures~$[w] \subsetneq [a]$.

We define next the set of infinite admissible words
$$ \Sigma^\omega = \{a_1 a_2 \dotsb \mid a_i \in \Sigma, a_{i+1} \neq \hat{a_i}\ \forall i\}. $$
At this stage, our setting is identical to that described in Sections~2--4 of~\cite{Series1981}. Since in our case, the group is free with respect to the generators~$\{g_a\}_{a\in\Sigma}$, the graph of~$G$ in the terminology of~\cite{Series1981} is a tree, and admissible sequences coincide with words in~$\Sigma^*$.
By~\cite[Proposition 4.5]{Series1981}, for any~$w = a_1a_2\dotsb \in\Sigma^\omega$, the intersection
$$ [w] := \bigcap_{n\geq 1} [a_1, \dotsc, a_n] $$
consists of a single point in~$\dD$, which we still denote by~$[w]$.

Define the map
$$ \pi : \Sigma^\omega \to \dD, \qquad \pi(w) = [w]. $$
\begin{lemma}
  The map~$\pi$ is surjective. For all~$x\in \dD$, we have~$|\pi^{-1}(x)|=1$ except when~$x$ belongs to the set of cusps~$C(G)$, in which case~$|\pi^{-1}(x)| = 2$.
\end{lemma}
\begin{proof}
  This is essentially~\cite[Proposition~4.6]{Series1981}; with the terminology used there, we are only left to show that the words~$w = a_1 a_2 \dotsb\in \Sigma^\omega$ which are eventually extreme left or right paths are precisely those for which~$\pi(w)\in C(G)$. 
  
  Suppose that~$\pi(w) \in C(G)$. Then~$\pi(w)$ is equivalent to one of the vertices of~$R$, say~$x_0$. Write~$\pi(w) = gx_0$ with~$g\in G$, and write~$g = g_{b_1} \dotsc g_{b_m}$. We pick~$x_0$ and~$g$ such that~$m$ is minimal. When~$m\geq1$, this ensures that~$x_0\not\in [\hat b_m]$,  otherwise~$x_0' = g_{b_m}x_0$ is also a vertex of~$R$, and the expression~$\pi(w) = g' x_0'$ with~$g' = g_{b_1} \dotsc g_{b_{m-1}}$ contradicts the minimality of~$m$. Since~$x_0$ is a vertex of~$R$, which is an ideal polygon, it can be reached through two extreme paths~$w_L, w_R$, which start with consecutive letters, neither of which is equal to~$\hat b_m$ when~$m\geq1$. Then with~$u = b_1 \dotsb b_m$  (and~$u=\epsilon$ if~$m=0$), the words~$u w_L$ and~$u w_R$ are admissible, distinct, and by construction~$\pi(w) = \pi(u w_L) = \pi(u w_R)$. By ~\cite[Proposition~4.6]{Series1981}, we get~$w\in\{u w_L, u w_R\}$.

  Suppose conversely that~$w$ is an eventually extreme path in~$\Sigma^\omega$. This implies there exist~$w_1,w_2\in\Sigma^*$, $w_2\neq\epsilon$, such that~$w=w_1 w'$ with~$w'=w_2w_2\cdots$ and~$w'$ an extreme (left or right) path. Writing~$w' = a'_1a'_2 \dotsc$, the arcs~$[a'_1, \dotsc, a'_j]$ all share an endpoint (left or right), which by taking the limit is seen to be~$[w']$. Thus~$[w']$ is one of the endpoints of~$[a_1']$, thus in~$V$ and so~$\pi(w)=g_{w_1}[w']\in C(G)$.
\end{proof}

We also note a cycle relation, which describes the relations among the~$2n$ vertices~$V$ of~$R$ in terms of side-pairings.  Choose a starting vertex~$v_0$ and let~$s_0$ be the side of~$R$ which ends in~$v_0$ and comes before~$v_0$ in clockwise order. Then let~$v_1=g_{s_0}^{-1} v_0\in V$. We repeat this process until the sequence of vertices eventually repeats
$$ v_0 \xrightarrow{g_{s_0}^{-1}}  v_1=g_{s_0}^{-1} v_0  \xrightarrow{g_{s_1}^{-1}}  v_2 \xrightarrow{g_{s_2}^{-1}}  \cdots   \xrightarrow{g_{s_{r-1}}^{-1}} v_r=v_0. $$
In particular, $g=(g_{s_0}\cdots g_{s_{r-1}})^{-1}$ fixes~$v_0$ and in fact, if~$r\geq1$ is minimal, it generates~$G_{v_0}$. Clearly the vertices~$v_0,\dots, v_{r-1}$ are equivalent and the cycles for each of them is obtained rotating the one for~$v_0$.
We also remark that choosing the sides after the vertices would generate the same sequence in the opposite direction with corresponding matrix~$g^{-1}$.

\begin{example}
  Consider the principal congruence subgroup
  $$ \Gamma = \Gamma(3) = \{\gamma\in \SL(2, \Z), \gamma \equiv \id \pmod{3}\}. $$
  Its conjugate inside~$\SU(1, 1)$ is
  $$ G(3) = \Bigg\{g = \smat{\alpha&\beta \\ \bar\beta & \bar \alpha},  \alpha, \beta \in \tfrac12\Z[i], \bigg\{ \begin{array}{l} \alpha \pm \beta \in 1 + 3\Z[i], \\ \abs{\alpha}^2 - \abs{\beta}^2 = 1 \end{array} \Bigg\}. $$
  
  A computation by SageMath~\cite{sagemath} provides a fundamental domain for $G(3)$ consisting of the ideal polygon whose vertices are given by
  $$ p_\D(\{0, \tfrac12, 1, 2, 3, \infty\}) = \{-1, \tfrac{-3-4i}{5}, -i, \tfrac{3-4i}{5}, \tfrac{4-3i}{5}, 1\}. $$
  The set~$\Sigma$ has six letters, $\{a,b,c,\hat a,\hat b,\hat c\}$, and we may use the pairing matrices
  $$ g_a = \smat{1-\frac32 i& \frac32 i\\ * & *}, \quad g_b = \smat{-\frac72 - 3i & \frac92 \\ * & *}, \quad g_c = \smat{1 + 3i & -3 \\ * & *}. $$
  Observe that~$p_\D(1), p_\D(\infty)$ are inequivalent to other vertices, whereas the other vertices are related by the following two cycle relations: 
  $$p_\D(0) \rightarrow p_\D(3)\rightarrow p_\D(0) \rightarrow \cdots,\qquad p_\D(1/2) \rightarrow p_\D(2)\rightarrow p_\D(1/2) \rightarrow \cdots.$$
  In Figure~\ref{fig:example-Gamma3}, we have the drawing of the fundamental domain, highlighting also the set~$K$ described in Section~\ref{ssec:bs:accel} below.

  \begin{figure}[h]
    \centering
    \begin{tikzpicture}[line cap=round,line join=round,>=triangle 45,x=3.5cm,y=3.5cm]
      
      \fill[lightgray] (1,0) -- (-1,0) 
      arc (90:-36.87:0.5) 
      arc (143.13:0:1/3) 
      arc (180:36.87:1/3) 
      arc (216.87:53.13:1/7) 
      arc (233.13:90:1/3) -- cycle;

      \begin{scope}
        \clip (0,0) circle (1);
        
        \angeodesic{0}{-18.925}
        \angeodesic{-18.925}{-22.620}
        \angeodesic{-22.620}{-28.072}
        \angeodesic{-28.072}{-31.891}
        \angeodesic{-31.891}{-36.870}

        \angeodesic{-36.870}{-41.112}
        \angeodesic{-41.112}{-41.706}
        \angeodesic{-41.706}{-42.075}
        \angeodesic{-42.075}{-43.603}
        \angeodesic{-43.603}{-53.130}

        \angeodesic{-53.130}{-67.380}
        \angeodesic{-67.380}{-73.740}
        \angeodesic{-73.740}{-75.750}
        \angeodesic{-75.750}{-77.320}
        \angeodesic{-77.320}{-90}

        \angeodesic{-90}{-102.680}
        \angeodesic{-102.680}{-106.260}
        \angeodesic{-106.260}{-112.620}
        \angeodesic{-112.620}{-118.072}
        \angeodesic{-118.072}{-126.870}

        \angeodesic{-126.870}{-136.397}
        \angeodesic{-136.397}{-137.925}
        \angeodesic{-137.925}{-138.888}
        \angeodesic{-138.888}{-143.130}
        \angeodesic{-143.130}{180}

        \angeodesic{180}{90}
        \angeodesic{90}{53.130}
        \angeodesic{53.130}{43.603}
        \angeodesic{43.603}{36.870}
        \angeodesic{36.870}{0}
      \end{scope}

      \draw [line width=1pt] plot[domain=0.:1.,variable=\t]({1-2*\t}, {0.});
      \draw [line width=1pt] plot[domain=5.63968:7.85938,variable=\t]({-1.+0.5*cos(\t r)}, {-0.5+0.5*sin(\t r)});
      \draw [line width=1pt] plot[domain=0.:2.49809,variable=\t]({-0.333+0.333*cos(\t r)}, {-1.+0.333*sin(\t r)});
      \draw [line width=1pt] plot[domain=0.643501:3.141592,variable=\t]({0.333+0.333*cos(\t r)}, {-1.+0.333*sin(\t r)});
      \draw [line width=1pt] plot[domain=0.927295:3.785093,variable=\t]({0.714285+0.142857*cos(\t r)}, {-0.714285+0.142857*sin(\t r)});
      \draw [line width=1pt] plot[domain=1.570796:4.06888,variable=\t]({1.+0.333*cos(\t r)}, {-0.333+0.333*sin(\t r)});

      \draw [line width=1pt] (0,0) circle (1);

      \draw [line width=2pt, color=blue] (-31.891:1.005) arc (-31.891:-18.925:1.005);   
      \draw [line width=2pt, color=blue] (-43.603:1.005) arc (-43.603:-41.112:1.005);   
      \draw [line width=2pt, color=blue] (-77.320:1.005) arc (-77.320:-67.380:1.005);   
      \draw [line width=2pt, color=blue] (-118.072:1.005) arc (-118.072:-102.680:1.005); 
      \draw [line width=2pt, color=blue] (-143.130:1.005) arc (-143.130:-136.397:1.005); 
      \draw [line width=2pt, color=blue] (36.870:1.005) arc (36.870:90:1.005);          

      \draw[fill=black] (1,0) circle (1.5pt) node[right] {\scriptsize{$p_\D(\infty)$}};
      \draw[fill=black] (0.6,-0.8) circle (1.5pt) node[below] {\ \ \scriptsize{$p_\D(2)$}};
      \draw[fill=black] (-0.6,-0.8) circle (1.5pt);
      \draw (-0.73,-0.88) node {\scriptsize{$p_\D(1/2)$}};
      \draw[fill=black] (-1,0) circle (1.5pt) node[left] {\scriptsize{$p_\D(0)$}};
      \draw[fill=black] (0.8,-0.6) circle (1.5pt) node[right] {\scriptsize{$p_\D(3)$}};
      \draw[fill=black] (0,-1) circle (1.5pt) node[below] {\scriptsize{$p_\D(1)$}};

      \draw (0.,1.15) node {$[a]$};
      \draw (1.05,.35) node {\tiny $[a, a]$};
      \draw (.88,.67) node {\tiny $[a, b]$};
      \draw (.78,.79) node {\tiny $[a, c]$};
      \draw (.48,.98) node {\tiny $[a, \hat{c}]$};
      \draw (-.88,.67) node {\tiny $[a, \hat b]$};
      \draw (-1.,-0.45) node {$[b]$};
      \draw (-0.4,-1.05) node {$[c]$};
      \draw (0.4,-1.05) node {$[\hat c]$};
      \draw (0.81,-0.76) node {$[\hat b]$};
      \draw (1.05,-0.3) node {$[\hat a]$};
      
    \end{tikzpicture}
    \caption{Disk model for~$\Gamma(3)$ with shaded fundamental domain and highlighting~$K$ (in blue). The level 2 subarcs of $[a]$ are also indicated.}
    \label{fig:example-Gamma3}
  \end{figure}
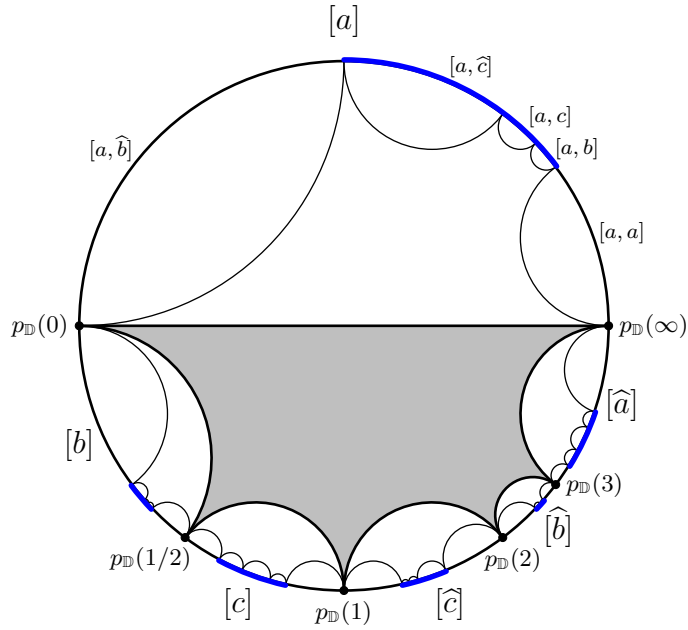

\end{example}

\subsection{Bowen-Series map}\label{bsm}

For all~$x\in \dD'=\dD\setminus V$, there is a unique~$a\in \Sigma$ such that~$x\in [a]$. Define the Bowen-Series map
$$ \cB : \dD' \to \dD $$
by letting~$\cB(x) := g_{a}^{-1} x$, where~$a$ is determined by~$x\in [a]$.

In the terminology of~\cite{Series1981}, this map can be seen as the shift map on the subset~$\Sigma'\subset\Sigma^\omega$ of words which are not extreme paths starting from~$0$. Further, the proof of Theorem~4.12 therein, show that~$\cB$ is a Markov map with respect to the partition~$\dD = \cup_{a\in \Sigma} [a]$.

\section{Induced map and cuspidal accelerations} \label{sec:bs:accel}

\subsection{An induced map}\label{ssec:bs:accel}

We say that two letters~$a,b\in\Sigma$ are \emph{separated} (or that~$a$ is separated from~$b$, or that~$|a-b|>1$) if~$a-b\notin  \{-1, 0, 1\}$, where we recall that~$\Sigma$ is identified with~$\Z/2n\Z$.
We say that two letters~$a,b\in\Sigma$ are \emph{non-cuspidal} if the pair~$(\hat{a}, b)$ is separated, and conversely we say that~$(a, b)$ is \emph{cuspidal} if~$\abs{\hat{a} - b} = 1$.
Following~\cite[p.~164]{BowenSeries1979}, we define the set~$K$ by letting
$$ K_a := \bigcup_{\substack{b\in \Sigma \\ |b - \hat a|>1}} [a, b], \qquad K := \bigcup_{a\in \Sigma} K_a. $$
In concrete terms, $K$ can be described as follows: for each~$a\in \Sigma$, the map~$g_a$ maps~$\dD\setminus[\hat a]$ into~$[a]$, and the image of~$V$ under~$g_a$ partitions~$[a]$ into~$2n-1$ intervals. From these intervals, retain the~$2n-3$ intervals in the middle, which do not meet~$V$. The union of these middle intervals is~$K_a$. We also denote by $W$ the set of endpoints of all these middle intervals, as $a$ varies. Equivalently,
\begin{align}\label{defW}
W=\bigcup_{\substack{a,b\in \Sigma \\ |b - \hat a|>1}} \partial [a, b]=\{x\in \dD'\mid \cB(x)\in V\}.
\end{align}

Given an admissible word~$w = a_1 \dotsb a_n \in \Sigma^*$ with~$n\geq1$, we call the \emph{$K$-length} of~$w$ the quantity
\begin{equation}
  \ell_K(w) :=1+ \abs{\{1\leq i < n, |a_{i+1} - \hat a_i|>1\} }.\label{eq:def-ellK}
\end{equation}
and set~$\ell_K(\epsilon):=0$.

Let~$\DcF \subset \dD'$ be the set consisting of~$C(G)\setminus V$ along with those points~$x\not\in C(G)$ for which~$\cB^r(x)\in K$ for some~$r\geq 1$.
We define the induced map of~$\cB$ on~$K\cup V$ as 
$$ \cF_K : \DcF \to K \cup V, $$
by
$$ \cF_K(x) := \cB^j(x), \quad j = \min\{t\geq 1, \cB^t(x) \in K \cup V\}. $$
By~\cite[Proposition~4.2]{Stadlbauer2004}, the set~$\DcF$ is of full Lebesgue measure in~$\dD$.

We will index the inverse branches of~$\cF_K$ by elements of the following set
$$ \Lambda_K(1) := \{ w = a_1 \dotsb a_n \in\Sigma^* , n\geq 1, \ell_K(a_2 \dotsb a_n) \leq 1 \}. $$
By definition, this means that given~$w\in \Lambda_K(1)$ and~$x \in K_b$, with~$b\in\Sigma\setminus\{\hat a_n\}$ and~$|\hat a_n-b|=1$ if~$n=|w|\geq2$, then
$$ \cF_K(g_{a_1\dotsb a_n} x) = x. $$
Notice that the condition on~$b$ is equivalent to~$\ell_K(a_2 \dotsb a_n b)=1$ and~$wb\in \Sigma^*$, i.e.~$wb$ is admissible.
For all~$w\in \Lambda_K(1)$, $w=a_1\dotsb a_n$, we let
$$ D_K(w, 1) := \bigcup_{\substack{b\in \Sigma,\\  \ell_K(a_2\dotsb a_n b)=1}} K_b = \left\{ \begin{aligned} &\bigcup_{b\neq \hat a_n} K_b, & (n=1), \\ &\bigcup_{\abs{b - \hat{a_n}}=1} K_b, & (n\geq 2). \end{aligned} \right. $$
Note that for~$w\neq \epsilon$, we have~$D_K(w, 1) \neq \varnothing$ if and only if~$w\in \Lambda_K(1)$.
For all~$w\in \Lambda_K(1)$, we define
\begin{equation}
  h_{w, 1} : D_K(w, 1) \to \dD, \qquad h_{w, 1}(x) = g_{a_1} \dotsc g_{a_n} x,\label{eq:def-hw}
\end{equation}
and
\begin{equation}\label{eq:def-H1}
  \cH := \{h_{w, 1},  w\in \Lambda_K(1)\}.
\end{equation}
The set~$\cH$ is the set of inverse branches of~$\cF_K$.

More generally, for any~$r\geq 1$, the set of inverse branches of order~$r\geq 1$ of~$\cF_K$ is parametrized by 
$$ \Lambda_K(r) := \{w = a_1 \dotsb a_n\in \Sigma^*, n\geq r: \ell_K(a_2\cdots a_n) \in \{r-1, r\}\}. $$
Let
$$ D_K(w, r) := \bigcup_{\substack{b\in \Sigma,\\  \ell_K(a_2\dotsb a_nb)=r}} K_b, $$
so that~$D_K(w, r) \neq \varnothing$ if and only if~$w\in \Lambda_K(r)$, and
$$ h_{w, r} : D_K(w, r) \to \dD, \qquad h_{w, r}(x) = g_{a_1} \dotsc g_{a_n} x ,$$
where~$w = a_1 \dotsb a_n$. See Figure~\ref{fig:gamma32_sub3} for an example of $D_K(w, r)$.
Finally let
$$ \cH^r := \{h_{w, r} : w\in \Lambda_K(r)\} $$
with the notation~\eqref{eq:def-hw}; in particular, $\cH=\cH^1$.
We remark that if~$\ell_K(a_2 \dotsb a_n)=r-1$, then the word~$w = a_1\dotsb a_n$ belongs both to~$\Lambda_K(r)$ and~$\Lambda_K(r-1)$. The domains of the corresponding inverse branches~$h_{w, r}$ and~$h_{w, r-1}$ however are different and in fact disjoint.

In any case, given~$w\in \Lambda_K(r)$, we will denote from now on
\begin{align*}
  \Dom(h) ={}& D_K(w, r) \qquad (h = h_{w, r}), \\
  \im(h) ={}& h(\Dom(h)).
\end{align*}
Finally, we denote~$\cH^0 = \{\id\}$ and
$$ \cH^* := \bigcup_{r\geq 0} \cH^r. $$

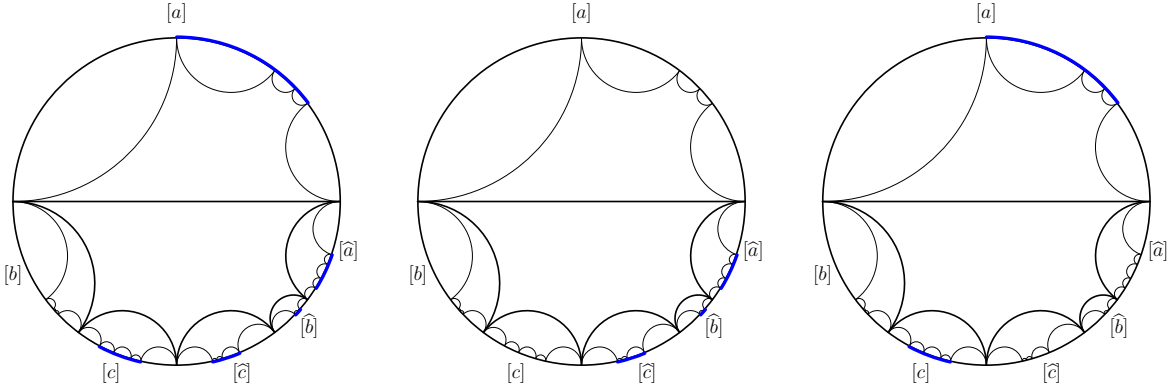
\begin{figure}[htbp]
  \centering
  
  \begin{subfigure}[t]{0.32\textwidth}
    \centering
    \resizebox{\linewidth}{!}{
      \begin{tikzpicture}[line cap=round,line join=round,>=triangle 45,x=3.5cm,y=3.5cm]
        \begin{scope}
          \clip (0,0) circle (1);
          \angeodesic{0}{-18.925} \angeodesic{-18.925}{-22.620} \angeodesic{-22.620}{-28.072} \angeodesic{-28.072}{-31.891} \angeodesic{-31.891}{-36.870}
          \angeodesic{-36.870}{-41.112} \angeodesic{-41.112}{-41.706} \angeodesic{-41.706}{-42.075} \angeodesic{-42.075}{-43.603} \angeodesic{-43.603}{-53.130}
          \angeodesic{-53.130}{-67.380} \angeodesic{-67.380}{-73.740} \angeodesic{-73.740}{-75.750} \angeodesic{-75.750}{-77.320} \angeodesic{-77.320}{-90}
          \angeodesic{-90}{-102.680} \angeodesic{-102.680}{-106.260} \angeodesic{-106.260}{-112.620} \angeodesic{-112.620}{-118.072} \angeodesic{-118.072}{-126.870}
          \angeodesic{-126.870}{-136.397} \angeodesic{-136.397}{-137.925} \angeodesic{-137.925}{-138.888} \angeodesic{-138.888}{-143.130} \angeodesic{-143.130}{180}
          \angeodesic{180}{90} \angeodesic{90}{53.130} \angeodesic{53.130}{43.603} \angeodesic{43.603}{36.870} \angeodesic{36.870}{0}
        \end{scope}

        \draw [line width=1pt] plot[domain=0.:1.,variable=\t]({1-2*\t}, {0.});
        \draw [line width=1pt] plot[domain=5.63968:7.85938,variable=\t]({-1.+0.5*cos(\t r)}, {-0.5+0.5*sin(\t r)});
        \draw [line width=1pt] plot[domain=0.:2.49809,variable=\t]({-0.333+0.333*cos(\t r)}, {-1.+0.333*sin(\t r)});
        \draw [line width=1pt] plot[domain=0.643501:3.141592,variable=\t]({0.333+0.333*cos(\t r)}, {-1.+0.333*sin(\t r)});
        \draw [line width=1pt] plot[domain=0.927295:3.785093,variable=\t]({0.714285+0.142857*cos(\t r)}, {-0.714285+0.142857*sin(\t r)});
        \draw [line width=1pt] plot[domain=1.570796:4.06888,variable=\t]({1.+0.333*cos(\t r)}, {-0.333+0.333*sin(\t r)});
        \draw [line width=1pt] (0,0) circle (1);

        \draw [line width=2pt, color=blue] (-31.891:1.005) arc (-31.891:-18.925:1.005);   
        \draw [line width=2pt, color=blue] (-43.603:1.005) arc (-43.603:-41.112:1.005);   
        \draw [line width=2pt, color=blue] (-77.320:1.005) arc (-77.320:-67.380:1.005);   
        \draw [line width=2pt, color=blue] (-118.072:1.005) arc (-118.072:-102.680:1.005); 
        \draw [line width=2pt, color=blue] (36.870:1.005) arc (36.870:90:1.005);          

        \draw (0.,1.15) node {$[a]$}; \draw (-1.,-0.45) node {$[b]$}; \draw (-0.4,-1.05) node {$[c]$};
        \draw (0.4,-1.05) node {$[\hat c]$}; \draw (0.81,-0.76) node {$[\hat b]$}; \draw (1.05,-0.3) node {$[\hat a]$};
      \end{tikzpicture}
    }
    \caption{$n=1$ (whence $r=1$).}
  \end{subfigure}
  \hfill
  \begin{subfigure}[t]{0.32\textwidth}
    \centering
    \resizebox{\linewidth}{!}{
      \begin{tikzpicture}[line cap=round,line join=round,>=triangle 45,x=3.5cm,y=3.5cm]
        \begin{scope}
          \clip (0,0) circle (1);
          \angeodesic{0}{-18.925} \angeodesic{-18.925}{-22.620} \angeodesic{-22.620}{-28.072} \angeodesic{-28.072}{-31.891} \angeodesic{-31.891}{-36.870}
          \angeodesic{-36.870}{-41.112} \angeodesic{-41.112}{-41.706} \angeodesic{-41.706}{-42.075} \angeodesic{-42.075}{-43.603} \angeodesic{-43.603}{-53.130}
          \angeodesic{-53.130}{-67.380} \angeodesic{-67.380}{-73.740} \angeodesic{-73.740}{-75.750} \angeodesic{-75.750}{-77.320} \angeodesic{-77.320}{-90}
          \angeodesic{-90}{-102.680} \angeodesic{-102.680}{-106.260} \angeodesic{-106.260}{-112.620} \angeodesic{-112.620}{-118.072} \angeodesic{-118.072}{-126.870}
          \angeodesic{-126.870}{-136.397} \angeodesic{-136.397}{-137.925} \angeodesic{-137.925}{-138.888} \angeodesic{-138.888}{-143.130} \angeodesic{-143.130}{180}
          \angeodesic{180}{90} \angeodesic{90}{53.130} \angeodesic{53.130}{43.603} \angeodesic{43.603}{36.870} \angeodesic{36.870}{0}
        \end{scope}

        \draw [line width=1pt] plot[domain=0.:1.,variable=\t]({1-2*\t}, {0.});
        \draw [line width=1pt] plot[domain=5.63968:7.85938,variable=\t]({-1.+0.5*cos(\t r)}, {-0.5+0.5*sin(\t r)});
        \draw [line width=1pt] plot[domain=0.:2.49809,variable=\t]({-0.333+0.333*cos(\t r)}, {-1.+0.333*sin(\t r)});
        \draw [line width=1pt] plot[domain=0.643501:3.141592,variable=\t]({0.333+0.333*cos(\t r)}, {-1.+0.333*sin(\t r)});
        \draw [line width=1pt] plot[domain=0.927295:3.785093,variable=\t]({0.714285+0.142857*cos(\t r)}, {-0.714285+0.142857*sin(\t r)});
        \draw [line width=1pt] plot[domain=1.570796:4.06888,variable=\t]({1.+0.333*cos(\t r)}, {-0.333+0.333*sin(\t r)});
        \draw [line width=1pt] (0,0) circle (1);

        \draw [line width=2pt, color=blue] (-31.891:1.005) arc (-31.891:-18.925:1.005);   
        \draw [line width=2pt, color=blue] (-43.603:1.005) arc (-43.603:-41.112:1.005);   
        \draw [line width=2pt, color=blue] (-77.320:1.005) arc (-77.320:-67.380:1.005);   

        \draw (0.,1.15) node {$[a]$}; \draw (-1.,-0.45) node {$[b]$}; \draw (-0.4,-1.05) node {$[c]$};
        \draw (0.4,-1.05) node {$[\hat c]$}; \draw (0.81,-0.76) node {$[\hat b]$}; \draw (1.05,-0.3) node {$[\hat a]$};
      \end{tikzpicture}
    }
    \caption{$n>1$ and $\ell_K(w')=r-1$.}
  \end{subfigure}
  \hfill
  \begin{subfigure}[t]{0.32\textwidth}
    \centering
    \resizebox{\linewidth}{!}{
      \begin{tikzpicture}[line cap=round,line join=round,>=triangle 45,x=3.5cm,y=3.5cm]
        \begin{scope}
          \clip (0,0) circle (1);
          \angeodesic{0}{-18.925} \angeodesic{-18.925}{-22.620} \angeodesic{-22.620}{-28.072} \angeodesic{-28.072}{-31.891} \angeodesic{-31.891}{-36.870}
          \angeodesic{-36.870}{-41.112} \angeodesic{-41.112}{-41.706} \angeodesic{-41.706}{-42.075} \angeodesic{-42.075}{-43.603} \angeodesic{-43.603}{-53.130}
          \angeodesic{-53.130}{-67.380} \angeodesic{-67.380}{-73.740} \angeodesic{-73.740}{-75.750} \angeodesic{-75.750}{-77.320} \angeodesic{-77.320}{-90}
          \angeodesic{-90}{-102.680} \angeodesic{-102.680}{-106.260} \angeodesic{-106.260}{-112.620} \angeodesic{-112.620}{-118.072} \angeodesic{-118.072}{-126.870}
          \angeodesic{-126.870}{-136.397} \angeodesic{-136.397}{-137.925} \angeodesic{-137.925}{-138.888} \angeodesic{-138.888}{-143.130} \angeodesic{-143.130}{180}
          \angeodesic{180}{90} \angeodesic{90}{53.130} \angeodesic{53.130}{43.603} \angeodesic{43.603}{36.870} \angeodesic{36.870}{0}
        \end{scope}

        \draw [line width=1pt] plot[domain=0.:1.,variable=\t]({1-2*\t}, {0.});
        \draw [line width=1pt] plot[domain=5.63968:7.85938,variable=\t]({-1.+0.5*cos(\t r)}, {-0.5+0.5*sin(\t r)});
        \draw [line width=1pt] plot[domain=0.:2.49809,variable=\t]({-0.333+0.333*cos(\t r)}, {-1.+0.333*sin(\t r)});
        \draw [line width=1pt] plot[domain=0.643501:3.141592,variable=\t]({0.333+0.333*cos(\t r)}, {-1.+0.333*sin(\t r)});
        \draw [line width=1pt] plot[domain=0.927295:3.785093,variable=\t]({0.714285+0.142857*cos(\t r)}, {-0.714285+0.142857*sin(\t r)});
        \draw [line width=1pt] plot[domain=1.570796:4.06888,variable=\t]({1.+0.333*cos(\t r)}, {-0.333+0.333*sin(\t r)});
        \draw [line width=1pt] (0,0) circle (1);

        \draw [line width=2pt, color=blue] (-118.072:1.005) arc (-118.072:-102.680:1.005); 
        \draw [line width=2pt, color=blue] (36.870:1.005) arc (36.870:90:1.005);          

        \draw (0.,1.15) node {$[a]$}; \draw (-1.,-0.45) node {$[b]$}; \draw (-0.4,-1.05) node {$[c]$};
        \draw (0.4,-1.05) node {$[\hat c]$}; \draw (0.81,-0.76) node {$[\hat b]$}; \draw (1.05,-0.3) node {$[\hat a]$};
      \end{tikzpicture}
    }
    \caption{$n>1$ and $\ell_K(w')=r$.}
    \label{fig:gamma32_sub3}
  \end{subfigure}
  
  \caption{Domain $D_K(w, r)$ of $h_{w,r}\in\cH^r$ (underlined in blue), when the last letter of $w=a_1\dotsb a_n$ is $a_n=\hat b$, depending on the value of $n$ and $\ell_K(w')$, where $w'=a_2\dotsb a_n$. If $r=1$ only cases (A) and (C) appear, since $\ell_K(w')=0$ implies $n=1$.}
  \label{fig:domainH}
\end{figure}

\begin{lemma}\label{lem:invbranch-FK} \
  \begin{enumerate}
    \item For all~$x\in K$ and~$r\in\Z_{\geq 0}$, we have
    $$ \cF_K^{-r}(x) = \{h(x): h \in \cH^r, x\in \Dom(h)\}, $$
    and the values~$\{h(x)\}$ in the right-hand side are distinct.
    \item For all~$h\in \cH^r$, we have a unique decomposition in the form~$h = h_1 \circ \dotsb \circ h_r|_{\Dom (h)}$, where~$h_j \in \cH$ and
    \begin{align}\label{adda}
    h_{i+1}\circ\dotsb \circ h_r(\Dom(h)) \subset \Dom(h_i),\qquad i=1,\dots,r
    \end{align}
  \end{enumerate}
 \end{lemma}
 \begin{remark}
The consistency condition~\eqref{adda} makes it clear that the domain of~$h$, $\Dom(h)$, may be smaller that the full domain of~$h_r$. More generally, for inverse
branches~$h_1,\dots,h_m$, we denote by
$\Dom(h_1\circ\dotsb\circ h_m)$ the largest subset
of~$\Dom(h_m)$ on which all successive compositions
are defined.
  \end{remark}

\begin{proof}
  Let~$x\in K_b$ for some~$b\in \Sigma$, and~$r\geq 1$. Then by definition, we have~$\cF_K^r(y) = x$ if and only if~$y$ can be expressed as
  $$ y = g_{w_1} g_{w_2} \dotsb g_{w_r} x, $$
with  $w_1\dotsb w_r b$ admissible and, denoting~$w_i = a_{i,1} a_{i,2} \dotsb a_{i,n_i}$ with~$n_i \geq 1$ and~$a_{r+1,1}:=b$, we have:
  \begin{itemize}
    \item $\abs{\hat{a_{i,j}} - a_{i,j+1}} = 1$ ($1\leq i \leq r$, $2\leq j \leq n_i-1$),
    \item $\abs{\hat{a_{i,n_i}} - a_{i+1,1}} = 1$ ($1\leq i \leq r-1$, $n_i\geq 2$),
    \item $\abs{\hat{a_{r,n_r}} - b} = 1$ ($n_r\geq2$),
    \item $\abs{\hat{a_{i,1}} - a_{i,2}} > 1$ ($2\leq i \leq r$, $n_i \geq 2$),
    \item $\abs{\hat{a_{i,1}} - a_{i+1,1}} > 1$ ($2\leq i \leq r$, $n_i = 1$).
  \end{itemize}
  Let~$w = w_1\dotsb w_r$, and let~$w' = a_{1,2} \dotsb a_{1,n_1} w_2 \dotsb w_r$ be obtained from~$w$ by removing the first letter. Moreover, in the word~$w' b$, there are exactly~$r-1$ non-cuspidal pairs of consecutive letters, which are the following pairs:
  \begin{itemize}
    \item $(a_{i,1}, a_{i,2})$ for~$2\leq i \leq r$ where~$n_i \geq 2$,
    \item $(a_{i,1}, a_{i+1,1})$ for~$2\leq i < r$ where~$n_i = 1$,
    \item $(a_{r,1}, b)$ if~$r\geq2$ and~$n_r = 1$.
  \end{itemize}
  We deduce~$\ell_K(w' b) = r$, and therefore also~$\ell_K(w') \in \{r-1, r\}$ (depending on whether~$n_r = 1$ or~$n_r \geq 2$). Since clearly we also have~$\abs{w}\geq r$, we conclude that~$x \in D_K(w, r)$ and~$w \in \Lambda_K(r)$. This shows that~$y = h(x)$ for~$h = h_{w, r} \in \cH^r$.

  Conversely, let~$h\in \cH^r$, and write~$h = h_{w, r}$ for some~$w \in \Lambda_K(r)$. Assume~$x \in D_K(w, r)$. This means~$\ell_K(w' b) = r$, where~$w'$ is obtained from~$w$ by removing the first letter. We break
  \begin{equation}
    w = w_1 w_2 \dotsb w_r,\label{eq:lemma-K-decomp-w}
  \end{equation}
  where~$w_2, \dotsc, w_r$ start at the first letter of a non-cuspidal pair in~$w'b$. Then we easily see by induction that for all~$1\leq i \leq r$,
  $$ \cF_K^i(g_{w_1} \dotsb g_{w_r} x) = g_{w_{1+i}} \dotsc g_{w_r} x $$
  where the product on the right-hand side is empty if~$i=r$. For~$i=r$ this shows that~$h(x) \in \cF_K^{-r}(x)$.
  
  Finally, suppose that~$w_1, w_2\in \Lambda_K(r)$ are such that~$x\in D_K(w_1, r) \cap D_K(w_2, r)$ and~$y:=h_{w_1, r}(x) = h_{w_2, r}(x)$. Letting~$b\in\Sigma$ be such that~$x\in K_b$, we have that~$y$ belong to the arcs~$[w_1b]$ and~$[w_2b]$, which implies (up to swapping~$w_1$ and~$w_2$) that~$w_1b$ is a prefix of~$w_2b$. Let~$w'_i$ be obtained from~$w_i$ by removing the first letter. Since~$x\in K_b$, we have~$x \in [b, c]$ with~$(b,c)$ non-cuspidal, and this implies that~$\ell_K(w'_1bc) = r+1$. If~$w_1 \neq w_2$, then~$w'_1 b c$ is a prefix of~$w'_2 b$, and we deduce therefore~$\ell_K(w'_2 b)\geq r+1$, a contradiction. We conclude that~$w_1 = w_2$. This shows that the values~$\{h(x), h\in\cH^r, x\in \Dom(h)\}$ are distinct.

  We now move to proving (2). Let~$h\in \cH^r$ and~$x\in \Dom(h)$, say~$x\in K_b$ for some~$b\in\Sigma$. Write~$h = h_{w, r}$ where~$w\in \Lambda_K(r)$ and~$x \in D_K(w, r)$, let~$w_1, \dotsc, w_r$ be constructed as in~\eqref{eq:lemma-K-decomp-w}, and let~$b_i$ be the first letter of~$w_i$. Finally let~$b_{r+1} := b$.
  By construction, we have that within each word~$w_i b_{i+1}$, the only possible non-cuspidal pair of letters are the first two, and therefore~$w_i \in \Lambda_K(1)$ for all~$i$, and also~$K_{b_{i+1}} \subset D_K(w_i, 1)$.
  By a descending induction, we deduce~$g_{w_i} \dotsb g_{w_r} K_b \subset [b_i]$ for~$1\leq i \leq r$. Let~$h_i := h_{w_i, 1} \in \cH$. We deduce that~$K_b \subset \Dom(h_r)$, and for all~$1\leq i \leq r-1$, that~$h_{i+1}\circ\dotsb \circ h_r(K_b)\subset K_{b_{i+1}}\subset \Dom(h_i) $. This shows the existence of the decomposition~$h = h_1 \circ \dotsb \circ h_r$. The uniqueness is easily established.
\end{proof}

\subsection{Cuspidal acceleration}
\label{sec:cusp-acceleration}

Most of our analysis will be carried out using the map~$\cF_K$ in the previous section.
For our later arguments and for some applications, however, it will be crucial for us to consider another acceleration of~$\cB$, which is slower, but whose inverse branches are easier to describe.
What follows is inspired by the notion of cuspidal decomposition introduced in the work~\cite{ArtigianiEtAl2016} (see also~\cite{ArtigianiEtAl2020,Marchese2018}).

Given a word~$w = a_1 \dotsb a_r \in \Sigma^*\setminus\{\epsilon\}$, we say that~$w$ is \emph{cuspidal} if there is~$\delta\in \{\pm 1\}$ such that~$\hat a_i =  a_{i+1} + \delta$ for all~$i=1,\dots,r-1$. If~$\delta=1$, we say that~$w$ is \emph{right-cuspidal}, and if~$\delta=-1$ we say that~$w$ is \emph{left-cuspidal}. Note that if~$\abs{w}=1$ then~$w$ is both left and right-cuspidal (and conversely). Note also that~$w$ is cuspidal if and only if the arcs~$[a_1],[a_1,a_2],\dots,[a_1,\dots,a_r]$ share an endpoint (the right one if~$w$ is right-cuspidal and the left one otherwise). These notions coincide with the corresponding notions in~\cite{ArtigianiEtAl2016}.
We say that~$w$ is a \emph{cuspidal cycle} if~$wa_1$ is cuspidal and that a cuspidal cycle is~\emph{minimal} if it has no prefix which is a cuspidal cycle. Notice there are exactly $2|\Sigma|$ minimal cuspidal cycles.

Given~$w\in\Sigma^*$ we define the number of \emph{cuspidal flips}~$\ell_c(w)$ of~$w$ as follows.
If~$w=\epsilon$ let~$\ell_c(w)=0$, whereas for~$w = a_1 \dotsb a_n \in \Sigma^*$ with~$n\geq1$, we let
\begin{equation}
  \ell_c(w) :=1+ \abs{\{1\leq i < n-1, \delta_i\delta_{i+1}\neq1\} }.\label{eq:def-ellc}
\end{equation}
with~$\delta_i=0$ if~$|a_{i+1}-\hat a_i|\neq1$ and otherwise~$\delta_i:=a_{i+1}-\hat a_i\in\{\pm1\}$. Notice that, if we denote~$\hat w:=\hat a_n\dots  \hat a_1$, we have
\begin{align}\label{eq:invlc}
  \ell_c(w)=\ell_c(\hat w).
\end{align}

In~\cite[Definition~4.4]{ArtigianiEtAl2016}, the authors introduce the notion of decomposition into maximal cuspidal words. Note that the quantity~$\ell_c(w)$ defined here \emph{does not} correspond, in general, to the number of cuspidal blocks in such a decomposition of~$w$. For instance if~$w = a_1a_2b_1b_2$ with~$a_1a_2$ and~$b_1b_2$ right-cuspidal, and~$a_2b_1$ left-cuspidal, then~$(\delta_i)_{1\leq i \leq 3} = (1, -1, 1)$, and therefore~$\ell_c(w) = 3$; however the decomposition into cuspidal blocks as defined in~\cite{ArtigianiEtAl2016} has two blocks, namely~$w = (a_1a_2) (b_1b_2)$.

Before we actually define the accelerated map that will be of interest for us, we first discuss the cuspidal acceleration defined in~\cite{ArtigianiEtAl2016}.
For non-cuspidal $x=[a_1,a_2,\dotsb]\in\DcF$, this map is given by~$\cFAMU(x)=[a_{r+1}, a_{r+2}, \dotsb ]$, where~$r\geq 1$ is the maximal index such that~$a_1\dotsb a_r$ is cuspidal. More generally, we let

\begin{minipage}{0.5\textwidth}
  \begin{align*}
    I ^\pm:={}&\bigcup_{a\in\Sigma} [a,\hat a \pm 1]^*; \\
    I_a^\pm :={}& [\hat a\pm1,\hat{\hat a\pm1}\pm1 ]^*; \\
    I_x^\pm:={}&I^\pm_a, \qquad (x\in[a]\setminus V, a\in\Sigma),
  \end{align*}
\end{minipage}
\hfill
\begin{minipage}{0.5\textwidth}
  \begin{tikzpicture}[scale=2.3]
    \draw (0, 0) circle (1);

    \intervalpartial{-90}{240}{0.5};
    \interval{-90}{-30} node[midway, above left] {$[a]$};
    \draw[line width=1.5px] (-72:1) arc (-72:-48:1);
    \interval{-90}{-72};
    \draw[->] (0.4, -1.3) node[right] {$I^+\cap[a]$} -- (0.15, -1.05);
    \interval{-48}{-30};
    \draw[->] (1.1, -0.85) node[right] {$I^-\cap[a]$} -- (0.85, -0.7);
    \intervalpartial{-30}{-15}{0.5};

    \intervalpartial{20}{360}{0.5};

    \interval{20}{60};
    \draw[line width=1.5px] (33:1) arc (33:47:1);
    \interval{20}{33};\interval{47}{60};
    \draw[->] (1.1, 0.9) node[right] {$I_a^-$} -- (0.63, 0.85);

    \interval{60}{90} node[midway, below] {$[\hat a]$};

    \interval{90}{130};
    \draw[line width=1.5px] (103:1) arc (103:117:1);
    \interval{90}{103}; \interval{117}{130};
    \draw[->] (-0.5, 1.15) node[left] {$I_a^+$} -- (-0.20, 1.05);

    \intervalpartial{130}{155}{0.5};
  \end{tikzpicture}
\end{minipage}
where~${}^*$ indicates that the border of the arc belonging also to~$K$ is removed.
Then on all of $\DcF$, the map~$\cFAMU:\DcF\to\dD$ can be defined as
\[ \cFAMU(x) :=\begin{cases}
    \cB(x), & \text{if $x\in K$}\\
    \cB^{r+1}(x),&\text{if $x\in I^\pm $, where $r = \min\{t\geq 1, \cB^t(x) \notin I^\pm_{\cB^{t-1}(x)}\}.$}
  \end{cases}
\]
For non-cuspidal $x=[a_1,a_2,\dotsc]\in\DcF$ we can also express~$\cFAMU$ in terms of~$\ell_c$ as
\[
  \cFAMU(x) :=\begin{cases}
    \cB(x), & \text{if $x\in K$},\\
    [a_{j+1}, a_{j+2}, \dotsc], & \text{if } x\not\in K, \text{where $j\geq 1$ is maximal such that $\ell_c([a_1\dotsb a_{j}])=1$}.
  \end{cases}
\]

For our purposes, it is however more natural to consider a variation of~$\cFAMU$ which we denote by~$\cF_c$, defined for~$x\in \dD'$ in the following way : 
\begin{equation}
  \cF_c(x) :=\begin{cases}
    \cB(x), & \text{if $x\in K$}\\
    \cB^r(x),&\text{if $x\in I^\pm $, where $r = \min\{t\geq 1, \cB^t(x) \notin I^\pm_{\cB^{t-1}(x)}\}.$}
  \end{cases}\label{eq:def-Fc}
\end{equation}
The difference, with respect to~$\cFAMU$, is that when~$x\notin K$ the exponent in the iteration is one less.
We remark that the function~$\cF_K$ defined in the previous section is also the induced map of~$\cF_c$ on~$K\cup V$; this is not true for the map~$\cFAMU$.

Similarly as~$\cFAMU$, the map~$\cF_c$ can also be expressed by means of~$\ell_c$, in the following way~: if~$x = [a_1, a_2, \dotsc]\in\dD'$, then
\begin{align}\label{eq:altm}
  \cF_c(x)=\begin{cases}
    \cB(x), & \text{if $x\in W$},\\
    [a_{j+1}, a_{j+2}, \dotsc],  & \begin{aligned}
        &\text{if }x\notin W,\text{ where }j\geq1
          \text{ is maximal such that}\\
        &\ell_c(a_1\dotsb a_{j+1})=1\text{ and }[a_{j+1},a_{j+2},\dotsb]\notin V.
      \end{aligned}  \end{cases}
\end{align}

We now classify the inverse branches of~$\cF_c$. Similarly as for~$\cF_K$ we parametrize the inverse branches of order $r\geq1$ of~$\cF_c$ by the following set of words,
\[ \Lambda_c(r) := \{w \in \Sigma^*:\ell_c(w)\in\{r-1, r\},\ |w|\geq r, [\hat w]\not\subset K\text{ if $\ell_c(w)=r$}\}. \]
Notice that $ [\hat w]\not\subset K$ is always satisfied if $|w|=1$ and if $|w|\geq2$ it is equivalent to the last two letters of $w$ being a cuspidal pair. If $r=1$ we write
$$ \Lambda_c :=\Lambda_c(1)= \{w \in \Sigma^*:\ell_c(w)=1, [\hat w]\not\subset K \}. $$
We also let
$$ D_c(w,r) := \bigcup_{\substack{b_1, b_2 \in \Sigma \\ \ell_c(wb_1) = r \\ \ell_c(wb_1b_2)=r+1}} [b_1, b_2],$$
where if $|w|\geq2$ and $\ell_c(w)=r$, we remove the vertices $V$ from this union.
Note that~$w\in \Lambda_c(r)$ if and only if~$D_c(w,r)\neq\varnothing$.
We define the corresponding inverse branches through
\begin{equation}
  h^c_{w,r} : D_c(w,r) \to \dD, \qquad h^c_{w,r}(x) := g_{a_1} \dotsb g_{a_n} x\label{eq:def-hwc}
\end{equation}
if~$w=a_1 \dotsb a_n$. Finally, let~$\cH_c^0 = \{\id\}$,
\begin{equation}\label{eq:def-Hcr}
  \cH_c^r := \{h^c_{w,r}, w \in \Lambda_c(r) \},\qquad   \cH_c^* := \bigcup_{r\geq 0} \cH_c^r.
\end{equation}
Given~$h\in \cH_c^r$, we let~$\Dom(h)$ and~$\im(h)$ be the domain and range of~$h$, so that~$\Dom(h) = D_c(w,r)$ if~$h = h^c_{w,r}$. Finally, we drop the dependency on $r$ from the notation when~$r=1$, as was the case for~$\cH$.

We record the properties of the above discussion in the following lemma. 
\begin{lemma}\label{ladcaxx}
  For all~$x\in \dD$ and~$r\in\Z_{\geq 0}$, we have
\begin{align}\label{asxax} \cF_c^{-r}(x) = \{h(x): h\in \cH^r_c, x \in \Dom(h)\}, 
\end{align}
  and the values~$\{h(x)\}$ in the right-hand side are distinct.
\end{lemma}
\begin{proof}
The proof is very similar to Lemma~\ref{lem:invbranch-FK}, with only the equality in~\eqref{asxax} for $x=v\in V$, $r\geq1$, requiring a more detailed explanation.

Let $w=a_1\dotsb a_n\in\Lambda_c(r)$ with $v\in D_c(w,r)\cap V$. We now prove that $\cF_c^r(h_{w,r}^c(v))=v$, which gives the inclusion $\supset$ in~\eqref{asxax}. 

 First we notice that $v$ can't be an endpoint of $ [\hat a_n]$. Indeed, the depth 2 cylinders $[b_1,b_2]$ containing such points and for which $wb_1$ is admissible satisfy $b_1= \hat a_n+\delta$, $b_2=\hat b_1+\delta$ for some $\delta\in\{\pm1\}$, but for these $\ell_c(wb_1)=\ell_c(wb_1b_2)$ contradicting the condition required in $D_c(w,r)$. Since $v\notin[\hat a_n]$, it then follows that $g_{a_n}v\in W\subset K$. 
If $n=1$ (and thus $r=1$),~\eqref{eq:altm} then yields $\cF_c(h_{w,r}^c(v))=v$. Thus, assume $n\geq2$ so that, by definition of $D_c(w,r)$  and since $v\in V$, we must have $\ell_c(w)=r-1$. Then, since for all $i<n$ we have $g_{a_{i}\dotsb a_n}v\notin W\cup V$, by~\eqref{eq:altm} we get $\cF_c^{r-1}(h_{w,r}^c(v))=g_{a_n}v\in W$, whence $\cF_c^r(h_{w,r}^c(v))=v$, as desired. 

Conversely, let $z\in\cF_c^{-r}(v)$, and let $n\geq 1$ be minimal such that $\cB^n(z)=v$, and let $w=a_1\dotsb a_n$ be the corresponding word, so that $z=g_wv$. The same argument as above gives $r=1$ if $n=1$, and $r=\ell_c(w)+1$ if $n\geq2$.
Since $g_{a_n}v\in W\cap K_{a_n}$, we can choose $b_1,b_2\in\Sigma$ such that $wb_1b_2$ is admissible, $v\in[b_1,b_2]$ and $|b_1-\hat a_n|>1$. This implies that $\ell_c(wb_1)=r$,  $\ell_c(wb_1b_2)=r+1$ and so $w\in\Lambda_c(r)$, $v\in D_c(w,r)$. Therefore $z=h_{w,r}^c(v)$.
\end{proof}

We will also need to consider the inverse branches of~$\cH_c$ with a restricted domain. Given~$w\in \Lambda_c$, define 
$$  D_c^\sharp(w) :=D_c(w)\setminus K, $$
and
$$ h^{c\sharp}_w := h^c_w |_{D_c^\sharp(w)}. $$
Finally, let
\begin{equation}\label{eq:def-Hcsharp}
  \cH_c^{\sharp} := \{h^{c\sharp}_w, w\in \Lambda_c\}.
\end{equation}
We remark that the same word~$w\in\Sigma^*$ can index elements of~$\cH$, $\cH_c$ and~$\cH_c^\sharp$. While the corresponding inverse branch~$h_w$ is defined by the same M\"obius transformation in all cases, its domain does depend on whether we consider~$h_w$ to be an element of~$\cH$, $\cH_c$ or~$\cH_c^\sharp$.

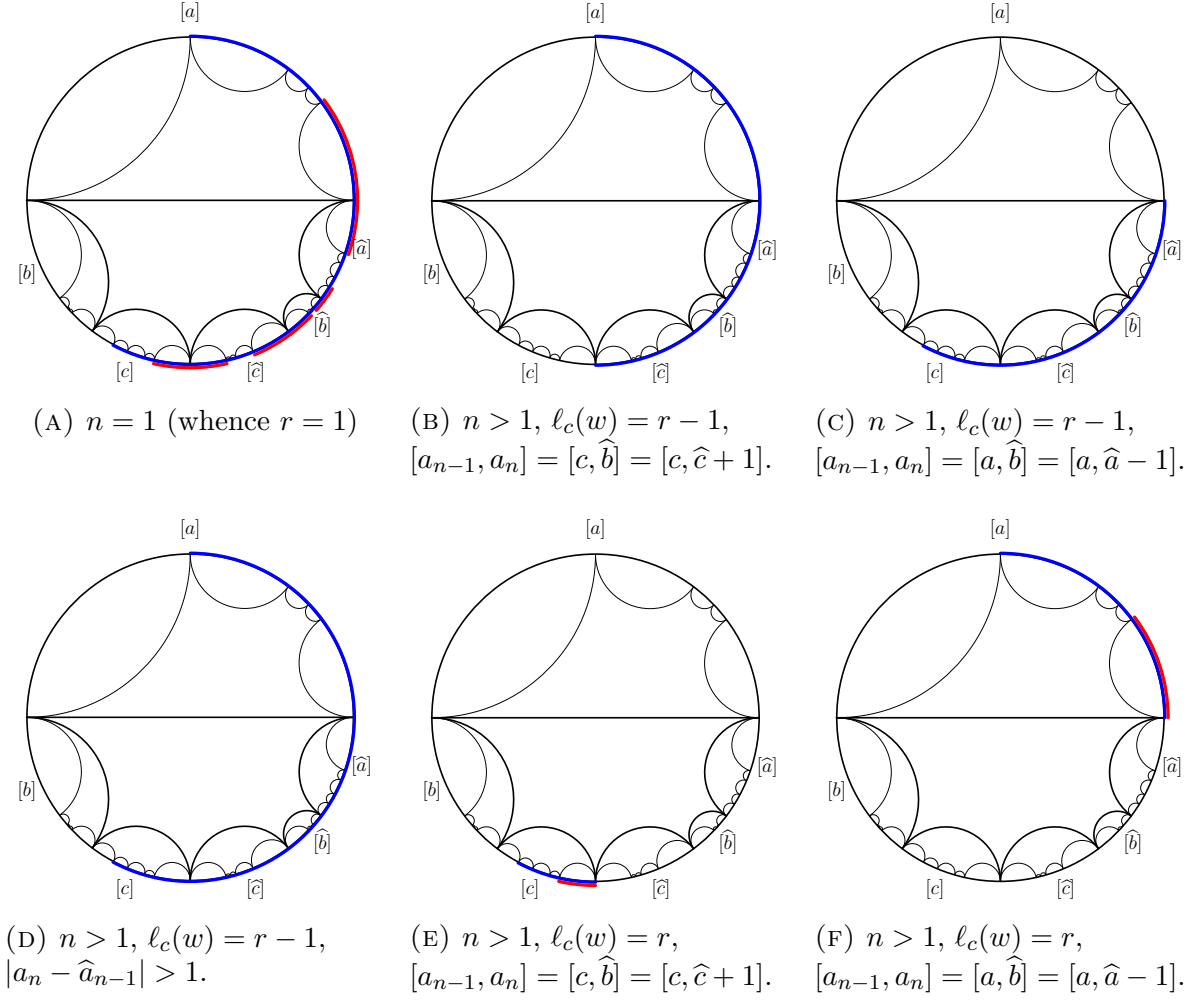
\begin{figure}[htbp]
  \centering
  
  
  \begin{subfigure}[t]{0.32\textwidth}
    \centering
    \vspace{0pt} 
    \resizebox{\linewidth}{!}{
      \begin{tikzpicture}[line cap=round,line join=round,>=triangle 45,x=3.5cm,y=3.5cm]
        \begin{scope}
          \clip (0,0) circle (1);
          \angeodesic{0}{-18.925} \angeodesic{-18.925}{-22.620} \angeodesic{-22.620}{-28.072} \angeodesic{-28.072}{-31.891} \angeodesic{-31.891}{-36.870}
          \angeodesic{-36.870}{-41.112} \angeodesic{-41.112}{-41.706} \angeodesic{-41.706}{-42.075} \angeodesic{-42.075}{-43.603} \angeodesic{-43.603}{-53.130}
          \angeodesic{-53.130}{-67.380} \angeodesic{-67.380}{-73.740} \angeodesic{-73.740}{-75.750} \angeodesic{-75.750}{-77.320} \angeodesic{-77.320}{-90}
          \angeodesic{-90}{-102.680} \angeodesic{-102.680}{-106.260} \angeodesic{-106.260}{-112.620} \angeodesic{-112.620}{-118.072} \angeodesic{-118.072}{-126.870}
          \angeodesic{-126.870}{-136.397} \angeodesic{-136.397}{-137.925} \angeodesic{-137.925}{-138.888} \angeodesic{-138.888}{-143.130} \angeodesic{-143.130}{180}
          \angeodesic{180}{90} \angeodesic{90}{53.130} \angeodesic{53.130}{43.603} \angeodesic{43.603}{36.870} \angeodesic{36.870}{0}
        \end{scope}

        \draw [line width=1pt] plot[domain=0.:1.,variable=\t]({1-2*\t}, {0.});
        \draw [line width=1pt] plot[domain=5.63968:7.85938,variable=\t]({-1.+0.5*cos(\t r)}, {-0.5+0.5*sin(\t r)});
        \draw [line width=1pt] plot[domain=0.:2.49809,variable=\t]({-0.333+0.333*cos(\t r)}, {-1.+0.333*sin(\t r)});
        \draw [line width=1pt] plot[domain=0.643501:3.141592,variable=\t]({0.333+0.333*cos(\t r)}, {-1.+0.333*sin(\t r)});
        \draw [line width=1pt] plot[domain=0.927295:3.785093,variable=\t]({0.714285+0.142857*cos(\t r)}, {-0.714285+0.142857*sin(\t r)});
        \draw [line width=1pt] plot[domain=1.570796:4.06888,variable=\t]({1.+0.333*cos(\t r)}, {-0.333+0.333*sin(\t r)});
        \draw [line width=1pt] (0,0) circle (1);

        \draw [line width=2pt, color=blue] (-118.072:1.005) arc (-118.072:90:1.005);    
        \draw [line width=2pt, color=red] (-102.680:1.025) arc (-102.680:-90:1.025); 
        \draw [line width=2pt, color=red] (0:1.025) arc (0:36.870:1.025); 
\draw [line width=2pt, color=red] (-90:1.025) arc (-90:-77.320:1.025);
\draw [line width=2pt, color=red] (-67.380:1.025) arc (-67.380:-53.130:1.025);
\draw [line width=2pt, color=red] (-53.130:1.025) arc (-53.130:-43.603:1.025);
\draw [line width=2pt, color=red] (-41.112:1.025) arc (-41.112:-36.870:1.025);
\draw [line width=2pt, color=red] (-36.870:1.025) arc (-36.870:-31.891:1.025);
\draw [line width=2pt, color=red] (-18.925:1.025) arc (-18.925:0:1.025);
  
        \draw (0.,1.15) node {$[a]$}; \draw (-1.,-0.45) node {$[b]$}; \draw (-0.4,-1.05) node {$[c]$};
        \draw (0.4,-1.05) node {$[\hat c]$}; \draw (0.81,-0.76) node {$[\hat b]$}; \draw (1.05,-0.3) node {$[\hat a]$};
      \end{tikzpicture}
    }
    \caption{$n=1$ (whence $r=1$)}
    \label{fig:gamma3_sub1}
  \end{subfigure}
  \hfill
  \begin{subfigure}[t]{0.32\textwidth}
    \centering
    \vspace{0pt}
    \resizebox{\linewidth}{!}{
      \begin{tikzpicture}[line cap=round,line join=round,>=triangle 45,x=3.5cm,y=3.5cm]
        \begin{scope}
          \clip (0,0) circle (1);
          \angeodesic{0}{-18.925} \angeodesic{-18.925}{-22.620} \angeodesic{-22.620}{-28.072} \angeodesic{-28.072}{-31.891} \angeodesic{-31.891}{-36.870}
          \angeodesic{-36.870}{-41.112} \angeodesic{-41.112}{-41.706} \angeodesic{-41.706}{-42.075} \angeodesic{-42.075}{-43.603} \angeodesic{-43.603}{-53.130}
          \angeodesic{-53.130}{-67.380} \angeodesic{-67.380}{-73.740} \angeodesic{-73.740}{-75.750} \angeodesic{-75.750}{-77.320} \angeodesic{-77.320}{-90}
          \angeodesic{-90}{-102.680} \angeodesic{-102.680}{-106.260} \angeodesic{-106.260}{-112.620} \angeodesic{-112.620}{-118.072} \angeodesic{-118.072}{-126.870}
          \angeodesic{-126.870}{-136.397} \angeodesic{-136.397}{-137.925} \angeodesic{-137.925}{-138.888} \angeodesic{-138.888}{-143.130} \angeodesic{-143.130}{180}
          \angeodesic{180}{90} \angeodesic{90}{53.130} \angeodesic{53.130}{43.603} \angeodesic{43.603}{36.870} \angeodesic{36.870}{0}
        \end{scope}

        \draw [line width=1pt] plot[domain=0.:1.,variable=\t]({1-2*\t}, {0.});
        \draw [line width=1pt] plot[domain=5.63968:7.85938,variable=\t]({-1.+0.5*cos(\t r)}, {-0.5+0.5*sin(\t r)});
        \draw [line width=1pt] plot[domain=0.:2.49809,variable=\t]({-0.333+0.333*cos(\t r)}, {-1.+0.333*sin(\t r)});
        \draw [line width=1pt] plot[domain=0.643501:3.141592,variable=\t]({0.333+0.333*cos(\t r)}, {-1.+0.333*sin(\t r)});
        \draw [line width=1pt] plot[domain=0.927295:3.785093,variable=\t]({0.714285+0.142857*cos(\t r)}, {-0.714285+0.142857*sin(\t r)});
        \draw [line width=1pt] plot[domain=1.570796:4.06888,variable=\t]({1.+0.333*cos(\t r)}, {-0.333+0.333*sin(\t r)});
        \draw [line width=1pt] (0,0) circle (1);

        \draw [line width=2pt, color=blue] (-90:1.005) arc (-90:90:1.005); 

        \draw (0.,1.15) node {$[a]$}; \draw (-1.,-0.45) node {$[b]$}; \draw (-0.4,-1.05) node {$[c]$};
        \draw (0.4,-1.05) node {$[\hat c]$}; \draw (0.81,-0.76) node {$[\hat b]$}; \draw (1.05,-0.3) node {$[\hat a]$};
      \end{tikzpicture}
    }
    \caption{$n>1$, $\ell_c(w)=r-1$,\\ $[a_{n-1},{a_n}]=[c,\hat b]=[c,\hat c+1]$.}
    \label{fig:gamma3_sub2}
  \end{subfigure}
  \hfill
  \begin{subfigure}[t]{0.32\textwidth}
    \centering
    \vspace{0pt}
    \resizebox{\linewidth}{!}{
      \begin{tikzpicture}[line cap=round,line join=round,>=triangle 45,x=3.5cm,y=3.5cm]
        \begin{scope}
          \clip (0,0) circle (1);
          \angeodesic{0}{-18.925} \angeodesic{-18.925}{-22.620} \angeodesic{-22.620}{-28.072} \angeodesic{-28.072}{-31.891} \angeodesic{-31.891}{-36.870}
          \angeodesic{-36.870}{-41.112} \angeodesic{-41.112}{-41.706} \angeodesic{-41.706}{-42.075} \angeodesic{-42.075}{-43.603} \angeodesic{-43.603}{-53.130}
          \angeodesic{-53.130}{-67.380} \angeodesic{-67.380}{-73.740} \angeodesic{-73.740}{-75.750} \angeodesic{-75.750}{-77.320} \angeodesic{-77.320}{-90}
          \angeodesic{-90}{-102.680} \angeodesic{-102.680}{-106.260} \angeodesic{-106.260}{-112.620} \angeodesic{-112.620}{-118.072} \angeodesic{-118.072}{-126.870}
          \angeodesic{-126.870}{-136.397} \angeodesic{-136.397}{-137.925} \angeodesic{-137.925}{-138.888} \angeodesic{-138.888}{-143.130} \angeodesic{-143.130}{180}
          \angeodesic{180}{90} \angeodesic{90}{53.130} \angeodesic{53.130}{43.603} \angeodesic{43.603}{36.870} \angeodesic{36.870}{0}
        \end{scope}

        \draw [line width=1pt] plot[domain=0.:1.,variable=\t]({1-2*\t}, {0.});
        \draw [line width=1pt] plot[domain=5.63968:7.85938,variable=\t]({-1.+0.5*cos(\t r)}, {-0.5+0.5*sin(\t r)});
        \draw [line width=1pt] plot[domain=0.:2.49809,variable=\t]({-0.333+0.333*cos(\t r)}, {-1.+0.333*sin(\t r)});
        \draw [line width=1pt] plot[domain=0.643501:3.141592,variable=\t]({0.333+0.333*cos(\t r)}, {-1.+0.333*sin(\t r)});
        \draw [line width=1pt] plot[domain=0.927295:3.785093,variable=\t]({0.714285+0.142857*cos(\t r)}, {-0.714285+0.142857*sin(\t r)});
        \draw [line width=1pt] plot[domain=1.570796:4.06888,variable=\t]({1.+0.333*cos(\t r)}, {-0.333+0.333*sin(\t r)});
        \draw [line width=1pt] (0,0) circle (1);

        \draw [line width=2pt, color=blue] (-118.072:1.005) arc (-118.072:0:1.005);   

        \draw (0.,1.15) node {$[a]$}; \draw (-1.,-0.45) node {$[b]$}; \draw (-0.4,-1.05) node {$[c]$};
        \draw (0.4,-1.05) node {$[\hat c]$}; \draw (0.81,-0.76) node {$[\hat b]$}; \draw (1.05,-0.3) node {$[\hat a]$};
      \end{tikzpicture}
    }
    \caption{$n>1$, $\ell_c(w)=r-1$,\\ $[a_{n-1},{a_n}]=[a,\hat b]=[a,\hat a-1]$.}
    \label{fig:gamma3_sub3}
  \end{subfigure}
  
  \vspace{1em} 
  
  
  \begin{subfigure}[t]{0.32\textwidth}
    \centering
    \vspace{0pt}
    \resizebox{\linewidth}{!}{
      \begin{tikzpicture}[line cap=round,line join=round,>=triangle 45,x=3.5cm,y=3.5cm]
        \begin{scope}
          \clip (0,0) circle (1);
          \angeodesic{0}{-18.925} \angeodesic{-18.925}{-22.620} \angeodesic{-22.620}{-28.072} \angeodesic{-28.072}{-31.891} \angeodesic{-31.891}{-36.870}
          \angeodesic{-36.870}{-41.112} \angeodesic{-41.112}{-41.706} \angeodesic{-41.706}{-42.075} \angeodesic{-42.075}{-43.603} \angeodesic{-43.603}{-53.130}
          \angeodesic{-53.130}{-67.380} \angeodesic{-67.380}{-73.740} \angeodesic{-73.740}{-75.750} \angeodesic{-75.750}{-77.320} \angeodesic{-77.320}{-90}
          \angeodesic{-90}{-102.680} \angeodesic{-102.680}{-106.260} \angeodesic{-106.260}{-112.620} \angeodesic{-112.620}{-118.072} \angeodesic{-118.072}{-126.870}
          \angeodesic{-126.870}{-136.397} \angeodesic{-136.397}{-137.925} \angeodesic{-137.925}{-138.888} \angeodesic{-138.888}{-143.130} \angeodesic{-143.130}{180}
          \angeodesic{180}{90} \angeodesic{90}{53.130} \angeodesic{53.130}{43.603} \angeodesic{43.603}{36.870} \angeodesic{36.870}{0}
        \end{scope}

        \draw [line width=1pt] plot[domain=0.:1.,variable=\t]({1-2*\t}, {0.});
        \draw [line width=1pt] plot[domain=5.63968:7.85938,variable=\t]({-1.+0.5*cos(\t r)}, {-0.5+0.5*sin(\t r)});
        \draw [line width=1pt] plot[domain=0.:2.49809,variable=\t]({-0.333+0.333*cos(\t r)}, {-1.+0.333*sin(\t r)});
        \draw [line width=1pt] plot[domain=0.643501:3.141592,variable=\t]({0.333+0.333*cos(\t r)}, {-1.+0.333*sin(\t r)});
        \draw [line width=1pt] plot[domain=0.927295:3.785093,variable=\t]({0.714285+0.142857*cos(\t r)}, {-0.714285+0.142857*sin(\t r)});
        \draw [line width=1pt] plot[domain=1.570796:4.06888,variable=\t]({1.+0.333*cos(\t r)}, {-0.333+0.333*sin(\t r)});
        \draw [line width=1pt] (0,0) circle (1);

        \draw [line width=2pt, color=blue] (-118.072:1.005) arc (-118.072:90:1.005);    
        
        \draw (0.,1.15) node {$[a]$}; \draw (-1.,-0.45) node {$[b]$}; \draw (-0.4,-1.05) node {$[c]$};
        \draw (0.4,-1.05) node {$[\hat c]$}; \draw (0.81,-0.76) node {$[\hat b]$}; \draw (1.05,-0.3) node {$[\hat a]$};
      \end{tikzpicture}
    }
    \caption{$n>1$, $\ell_c(w)=r-1$,\\ $|a_n-\hat a_{n-1}|>1$.}
    \label{fig:gamma3_sub4}
  \end{subfigure}
  \hfill
  \begin{subfigure}[t]{0.32\textwidth}
    \centering
    \vspace{0pt}
    \resizebox{\linewidth}{!}{
      \begin{tikzpicture}[line cap=round,line join=round,>=triangle 45,x=3.5cm,y=3.5cm]
        \begin{scope}
          \clip (0,0) circle (1);
          \angeodesic{0}{-18.925} \angeodesic{-18.925}{-22.620} \angeodesic{-22.620}{-28.072} \angeodesic{-28.072}{-31.891} \angeodesic{-31.891}{-36.870}
          \angeodesic{-36.870}{-41.112} \angeodesic{-41.112}{-41.706} \angeodesic{-41.706}{-42.075} \angeodesic{-42.075}{-43.603} \angeodesic{-43.603}{-53.130}
          \angeodesic{-53.130}{-67.380} \angeodesic{-67.380}{-73.740} \angeodesic{-73.740}{-75.750} \angeodesic{-75.750}{-77.320} \angeodesic{-77.320}{-90}
          \angeodesic{-90}{-102.680} \angeodesic{-102.680}{-106.260} \angeodesic{-106.260}{-112.620} \angeodesic{-112.620}{-118.072} \angeodesic{-118.072}{-126.870}
          \angeodesic{-126.870}{-136.397} \angeodesic{-136.397}{-137.925} \angeodesic{-137.925}{-138.888} \angeodesic{-138.888}{-143.130} \angeodesic{-143.130}{180}
          \angeodesic{180}{90} \angeodesic{90}{53.130} \angeodesic{53.130}{43.603} \angeodesic{43.603}{36.870} \angeodesic{36.870}{0}
        \end{scope}

        \draw [line width=1pt] plot[domain=0.:1.,variable=\t]({1-2*\t}, {0.});
        \draw [line width=1pt] plot[domain=5.63968:7.85938,variable=\t]({-1.+0.5*cos(\t r)}, {-0.5+0.5*sin(\t r)});
        \draw [line width=1pt] plot[domain=0.:2.49809,variable=\t]({-0.333+0.333*cos(\t r)}, {-1.+0.333*sin(\t r)});
        \draw [line width=1pt] plot[domain=0.643501:3.141592,variable=\t]({0.333+0.333*cos(\t r)}, {-1.+0.333*sin(\t r)});
        \draw [line width=1pt] plot[domain=0.927295:3.785093,variable=\t]({0.714285+0.142857*cos(\t r)}, {-0.714285+0.142857*sin(\t r)});
        \draw [line width=1pt] plot[domain=1.570796:4.06888,variable=\t]({1.+0.333*cos(\t r)}, {-0.333+0.333*sin(\t r)});
        \draw [line width=1pt] (0,0) circle (1);

        \draw [line width=2pt, color=blue] (-118.072:1.005) arc (-118.072:-90:1.005);    
        \draw [line width=2pt, color=red] (-102.680:1.025) arc (-102.680:-90:1.025); 

        \draw (0.,1.15) node {$[a]$}; \draw (-1.,-0.45) node {$[b]$}; \draw (-0.4,-1.05) node {$[c]$};
        \draw (0.4,-1.05) node {$[\hat c]$}; \draw (0.81,-0.76) node {$[\hat b]$}; \draw (1.05,-0.3) node {$[\hat a]$};
      \end{tikzpicture}
    }
    \caption{$n>1$, $\ell_c(w)=r$,\\ $[a_{n-1},{a_n}]=[c,\hat b]=[c,\hat c+1]$.}
    \label{fig:gamma3_sub5}
  \end{subfigure}
  \hfill
  \begin{subfigure}[t]{0.32\textwidth}
    \centering
    \vspace{0pt}
    \resizebox{\linewidth}{!}{
      \begin{tikzpicture}[line cap=round,line join=round,>=triangle 45,x=3.5cm,y=3.5cm]
        \begin{scope}
          \clip (0,0) circle (1);
          \angeodesic{0}{-18.925} \angeodesic{-18.925}{-22.620} \angeodesic{-22.620}{-28.072} \angeodesic{-28.072}{-31.891} \angeodesic{-31.891}{-36.870}
          \angeodesic{-36.870}{-41.112} \angeodesic{-41.112}{-41.706} \angeodesic{-41.706}{-42.075} \angeodesic{-42.075}{-43.603} \angeodesic{-43.603}{-53.130}
          \angeodesic{-53.130}{-67.380} \angeodesic{-67.380}{-73.740} \angeodesic{-73.740}{-75.750} \angeodesic{-75.750}{-77.320} \angeodesic{-77.320}{-90}
          \angeodesic{-90}{-102.680} \angeodesic{-102.680}{-106.260} \angeodesic{-106.260}{-112.620} \angeodesic{-112.620}{-118.072} \angeodesic{-118.072}{-126.870}
          \angeodesic{-126.870}{-136.397} \angeodesic{-136.397}{-137.925} \angeodesic{-137.925}{-138.888} \angeodesic{-138.888}{-143.130} \angeodesic{-143.130}{180}
          \angeodesic{180}{90} \angeodesic{90}{53.130} \angeodesic{53.130}{43.603} \angeodesic{43.603}{36.870} \angeodesic{36.870}{0}
        \end{scope}

        \draw [line width=1pt] plot[domain=0.:1.,variable=\t]({1-2*\t}, {0.});
        \draw [line width=1pt] plot[domain=5.63968:7.85938,variable=\t]({-1.+0.5*cos(\t r)}, {-0.5+0.5*sin(\t r)});
        \draw [line width=1pt] plot[domain=0.:2.49809,variable=\t]({-0.333+0.333*cos(\t r)}, {-1.+0.333*sin(\t r)});
        \draw [line width=1pt] plot[domain=0.643501:3.141592,variable=\t]({0.333+0.333*cos(\t r)}, {-1.+0.333*sin(\t r)});
        \draw [line width=1pt] plot[domain=0.927295:3.785093,variable=\t]({0.714285+0.142857*cos(\t r)}, {-0.714285+0.142857*sin(\t r)});
        \draw [line width=1pt] plot[domain=1.570796:4.06888,variable=\t]({1.+0.333*cos(\t r)}, {-0.333+0.333*sin(\t r)});
        \draw [line width=1pt] (0,0) circle (1);

        \draw [line width=2pt, color=blue] (0:1.005) arc (0:90:1.005);   
        \draw [line width=2pt, color=red] (0:1.025) arc (0:36.870:1.025); 

        \draw (0.,1.15) node {$[a]$}; \draw (-1.,-0.45) node {$[b]$}; \draw (-0.4,-1.05) node {$[c]$};
        \draw (0.4,-1.05) node {$[\hat c]$}; \draw (0.81,-0.76) node {$[\hat b]$}; \draw (1.05,-0.3) node {$[\hat a]$};
      \end{tikzpicture}
    }
    \caption{$n>1$, $\ell_c(w)=r$,\\ $[a_{n-1},{a_n}]=[a,\hat b]=[a,\hat a-1]$.}
    \label{fig:gamma3_sub6}
  \end{subfigure}
  
  \caption{Domain $D_c(w, r)$ of $h_{w,r}^c\in\cH_c^r$ (underlined in blue), where the last letter of $w=a_1\dotsb a_n$ is $a_n=\hat b$, depending on the value of $n$ and $\ell_c(w)$ and $a_{n-1}$. If $r=1$ only case (A), (E) and (F) appear. The restricted domain $D^\sharp_c(w)$ (for $r=1$) is highlighted in red.}
  \label{fig:domainHC}
\end{figure}

The relevance of~$\cH_c^\sharp$ can be seen from the following simple lemma, which characterizes inverse branches~$\cH$ for~$\cF_K$ in terms of inverse branches~$\cH_c, \cH_c^\sharp$ for~$\cF_c$.
\begin{lemma}\label{lem:decomp-Hc-HK}
  Let~$x\in K$. Then the set
  $$ \{h\in \cH, x\in \Dom(h)\} $$
  is in bijection with the set
  $$ \bigcup_{r\geq 1}\left\{(h_1, \dotsc, h_r): \begin{cases}h_1, \dotsc, h_{r-1} \in\cH^\sharp_c \text{ and } h_r\in\cH_c, \\ h_{i+1}\circ\dotsb \circ h_r(x) \in \Dom(h_i),\ i=1,\dots,r \end{cases} \right\}, $$
  in such a way that corresponding elements satisfy
  \begin{align}\label{eq:decomp-Hc-HK}
    h(x) = h_1 \circ \dotsb \circ h_r(x).
  \end{align}
  \end{lemma}
\begin{remark}
Given $h\in\cH$, the decomposition~\eqref{eq:decomp-Hc-HK}, as well as $r$, is constant on each component $K_a\subset\Dom(h)$, but may depend on $a$.    \end{remark}
\begin{proof}
  Write~$x\in [c_1, c_2] \subset K_{c_1}$ for some~$c_1, c_2 \in \Sigma$. We rephrase the desired correspondence in terms of the associated words writing~$h = h_{w, 1}$ and~$h_i = h_{w_i}^{c\sharp}$ ($i\leq r-1$) and~$h_r = h_{w_r}^c$. We need to prove that sets
  $$ M_1 = \{w\in \Lambda_K(1), [c_1, c_2] \subset D_K(w, 1)\} $$
  and
  $$ M_2 = \bigcup_{r\geq 1} \left\{(w_1, \dotsc, w_r), \begin{cases} w_i \in \Lambda_c, \\ [w_{i+1} \dotsb w_r c_1 c_2] \subset D_c^\sharp(w_i) & (1\leq i \leq r-1), \\ [c_1, c_2] \subset D_c(w_r). \end{cases} \right\} $$
  are in one-to-one correspondance through~$w = w_1 \dotsb w_r$.

  Suppose~$(w_1, \dotsc, w_r) \in M_2$ and let~$w = w_1\dotsb w_r$. Additionnally, let~$w'$ be obtained from~$w$ by removing the first letter. Let~$(b_1, b_2)$ be two consecutive letters in~$w'$. If~$(b_1, b_2)$ both belong to a single subword~$w_i$, then the condition~$w_i \in \Lambda_c$ ensures that~$\ell_c(w_i) = 1$ and thus~$(b_1, b_2)$ is a cuspidal pair. If~$b_1 \in w_i$ and~$b_2 \in w_{i+1}$ for some~$i\leq r-1$, the condition~$[w_{i+1}\dotsb] \subset D_c^\sharp(w_i)$ ensures that~$\ell_c(w_i b_2) = 1$, and thus if~$|w_i|\geq2$ then~$(b_1, b_2)$ are also cuspidal. Finally, if~$|w_i|=1$, then since~$b_1=w_i$ is in~$w'$ we must have~$i\geq2$ and thus~$[w_{i}\dotsb] \subset D_c^\sharp(w_{i-1})$ which implies that~$b_1b_2$ are also cuspidal in this case.  Since this is true for any consecutive pair of letters in~$w'$, we conclude that that~$\ell_K(w') \leq 1$, and therefore~$w\in \Lambda_K(1)$. Finally, since~$[c_1, c_2] \subset D_c(w_r)$, we have in particular~$\ell_c(w_r c_1) = 1$, and therefore also~$\ell_K(w' c_1) = 1$. This proves that~$x \in D_K(w,1)$, and thus~$w\in M_1$.

  Conversely, let~$w \in M_1$, which we write~$w=a_1\dotsb a_m\in\Sigma^*$ with~$a_i \in \Sigma$ and~$m\geq 1$. Let also~$a_{m+1} := c_1, a_{m+2} := c_2$, and let~$w'$ be obtained from~$w$ by removing the first letter.
  By hypothesis, we have~$\ell_K(a_2\dotsb a_m c_1) = 1$, which means that all consecutive pairs of letters in~$w'c_1$ are cuspidal. Let~$r := \ell_c(w c_1) \geq 1$.
For~$1\leq i\leq m+1$, let~$\delta_i$ be defined as
in~\eqref{eq:def-ellc}, applied to the word~$wc_1c_2$.
  By definition of~$r$, there are exactly~$r-1$ values of~$j\in \{2, \dotsc, m\}$ for which~$\delta_{j-1} \delta_j \neq 1$. We let~$2\leq m_2< \dotsb<m_r\leq m$ denote these values, and~$m_1 := 1$, $m_{r+1} := m+1$.
  Finally, for all~$i\in \{1, \dotsc, r\}$ we define the words
  $$ w_i := a_{m_i} a_{m_i+1} \dotsc a_{m_{i+1}-1}. $$
  We clearly have by construction~$\abs{w_i} \geq 1$, and~$w = w_1 \dotsb w_r$. Moreover, for all~$1\leq i \leq r$, we have~$\delta_j \delta_{j+1} = 1$ whenever~$m_i \leq j < m_{i+1}-2$, therefore~$\ell_c(w_i)=1$ and thus~$w_i\in \Lambda_c$.

  Next we consider~$i\in \{1, \dotsc, r\}$ and we compute~$\ell_c(w_i a_{m_{i+1}})$. If~$\abs{w_i} = 1$, then automatically~$\ell_c(w_i a_{m_{i+1}}) = 1$. Otherwise, we have~$m_i<m_{i+1}-1$, and therefore by definition~$\delta_{m_{i+1}-2} \delta_{m_{i+1}-1} = 1$. This, along with~$\ell_c(w_i) = 1$, ensures that~$\ell_c(w_i a_{m_{i+1}}) = 1$. In both cases, we have shown that~$\ell_c(w_i a_{m_{i+1}}) = 1$. On the other hand, by definition of~$m_{i+1}$, we have~$\delta_{m_{i+1}-1} \delta_{m_{i+1}} \neq 1$, and therefore~$\ell_c(w_i a_{m_{i+1}} a_{m_{i+1}+1}) = 2$. We have thus shown that~$[w_{i+1} \dotsb w_r c_1 c_2] \subset D_c(w_i)$.
  Finally, if~$i\leq r-1$, the condition~$\ell_K(w' c_1)=1$ ensures that~$[a_{m_{i+1}}, a_{m_{i+1}+1}] \not\subset K$, and therefore~$[w_{i+1} \dotsb w_r c_1c_2] \subset D_c^\sharp(w_i)$. This concludes the proof that~$(w_1, \dotsc, w_r) \in M_2$.
\end{proof}

Given~$y\in\DcF$, we let
\begin{align}\label{eq:def-return-Fc}
  r_y:=\min\{r\geq 1\mid \cF_c^{r}(y)\in K\cup V\}
\end{align}
be the the first (positive) time at which the orbit of  $y$ under~$\cF_c$ enters~$K\cup V$. Clearly, for~$h\in\cH$ and~$x\in\Dom(h)$ we have that 
\begin{align}\label{eq:return-h-r}
  r_{h(x)}=r
\end{align}
with~$r$ as in the decomposition~\eqref{eq:decomp-Hc-HK} of~$h$. By the definitions of $\cF_K$ and $\cF_c$, we have $\cF_K(x)=\cF_c^{r_x}(x)$ for $x\in\DcF$.

It will also be useful to go in the opposite direction, relating  the inverse branches $\cH$ with those in~$\cH$ and $\cH_c^\sharp$.

\begin{lemma}\label{lem:decomp-Hc}
  Let~$x\in \dD$. Then the set
  $$ \{h\in \cH_c^*, x\in \Dom(h)\} $$
  is in bijection with the set
  $$ \bigcup_{r,s\geq 0}\left\{(h_1, \dotsc, h_{r+s}): \begin{cases}h_1, \dotsc, h_{r} \in\cH \text{ and } h_{r+1},\dotsc,h_{r+s}\in\cH^\sharp_c, \\h_{i+1}\circ\dotsb \circ h_{r+s}(x) \in \Dom(h_i),\ i=1,\dots,r+s. \end{cases} \right\}, $$
  in such a way that corresponding elements satisfy
  \begin{align}\label{eq:decomp-Hc-HK2}
    h(x) = h_1 \circ \dotsb \circ h_{r+s}(x).
  \end{align}
\end{lemma}
\begin{remark}
  Notice that if $x\in K$, then necessary $s=0$. Moreover, if~$r,s\geq1$, then $h_{r+1}\circ\dotsb\circ h_{r+s}(x)\in K$.
\end{remark}
\begin{proof}
The identity corresponds to the empty tuple. We therefore
take $h\in\cH_c^m$ with $m\geq1$ and $x\in\Dom(h)$, and put $y_j:=\cF_c^j(h(x))$ for $j=0,\dots, m$. In particular, $y_m=x$ and a vertex can occur only at the last step.
Let $0<j_1<\dotsb<j_r\leq m$ be the indices $j\in\{1,\dots,m\}$ such that $y_j\in K$. Since $\cF_K$ is induced by $\cF_c$ on $K\cup V$, we have $\cF_K(y_{j_{i-1}})=y_{j_i}$ for $1\leq i\leq r$ and with $j_0=0$.
Thus, by Lemma~\ref{lem:invbranch-FK}, there are unique $h_i\in\cH$ such that $y_{j_i}\in\Dom(h_i)$ and $y_{j_{i-1}}=h_i(y_{j_i}).$
We then let $s:=m-j_r$. For $1\leq i\leq s$, we have $y_{j_r+i}\notin K$. In particular, the unique cuspidal inverse branch sending $y_{j_r+i}$ to $y_{j_r+i-1}$ provided by Lemma~\ref{ladcaxx} restricts to
a branch  $h_{r+i}\in\cH_c^\sharp$. By construction $h(x)=h_1\circ\dotsb\circ h_{r+s}(x)$ which yields a map between the two sets in the statements. Conversely, expanding each $\cH$ branch by Lemma~\ref{lem:decomp-Hc-HK} recovers the cuspidal decomposition, so this map is a bijection.
   \end{proof}

\subsection{Comparing the three maps}\label{sec:comparing-three-maps}

We conclude the section with an example to help visualize the difference between the three maps. Consider the example~$x=[w]$ with
\[ w= a\, a^+\, a^{++}\,
  b\, b^{+}\, b^{+-}\,b^{+-+}\, b^{+-++}\,
  c\, c^-\,
  d
  \dotsb\in\Sigma^\omega \]
where for all~$m\in\Sigma$, we wrote~$m^{ \pm}:=\hat{m}\pm1$ and where~$|\hat{a^{++}}-b|,|\hat{b^{+-++}}-c|,|\hat{c^{-}}-d|>1$.
Then the first few elements in the orbit of~$w$ under the three maps introduced above are as follows:
\[
  \begin{array}{r|llllllll}
    r & 1 & 2 & 3 & 4 & 5 & 6 & 7  \\ \hline
    \cF_K^r(w) & a^{++}b \dotsb & b^{+-++}c\dotsb & c^{-} d\dotsb & & & &  \\
    \cF_c^r(w) & a^{++}b\dotsb & bb^+\dotsb & b^{+}b^{+-}\dotsb & b^{+-}b^{+-+}\dotsb & b^{+-++} c\dotsb & c c^-\dotsb & c^-d\dotsb  \\
    \cFAMU^r(w) & bb^+\dotsb & b^{+-}b^{+-+} \dotsb & c c^-\dotsb & d\dotsb & & & 
  \end{array}
\]
where for each point of the orbit we indicated the first letters of the corresponding word in~$\Sigma^\omega$.

\subsection{Natural extensions and invariant measures}\label{sec:natur-extens-invar}

\subsubsection{The invariant measure for~$\cB$ and~$\cF_K$}

Let
$$ \df s(x) = \frac{\df x}{ix} = \df\theta \qquad( x = \e^{i\theta} \in \dD) $$
be the arc length measure on the circle. It is a routine computation to check that the measure~$\zeta$ on~$\dD^2$ given by
$$ \zeta(\df x, \df y) = \frac{2\df x \df y}{(x-y)^2} = \frac{-2xy\df s(x) \df s(y)}{(x-y)^2} = \frac{\df s(x)\df s(y)}{1-\Re(y\bar{x})} $$
is invariant under the transformation~$(x, y)\mapsto (g x, g y)$ if~$g\in \SU(1, 1)$; see e.g.~\cite[p.~528]{ArnouxSchmidt2019} for the computation on the half-plane model. 

Following the argument of~\cite[Section~3.1]{Series1985} (see also~\cite{AdlerFlatto1991} and~\cite[Section~6]{ArnouxSchmidt2019}), we can compute invariant measures for~$\cF_K, \cF_c$ and~$\cFAMU$ by integrating~$\zeta(\df x, \df y)$ along a well-chosen set.

\begin{proposition}\label{prop:xi-invar}
  For all~$x\in \dD\setminus V$, define
  \begin{equation}
    \label{eq:def-inv-meas}
    \xi(x) := \int_{y\in \dD \setminus [a]} \frac{\zeta(\df x, \df y)}{\df s(x)} = \int_{y\in \dD \setminus[a]} \frac{\df s(y)}{1-\Re(y \bar{x})}, \qquad (x\in [a]).
  \end{equation}
  \begin{itemize}
    \item The measure~$\xi(x) \df s(x)$ on~$\dD$ is~$\cB$-invariant.
    \item The measure~$\xi(x)\df s(x)$ restricted to~$K$ is~$\cF_K$-invariant.
  \end{itemize}
\end{proposition}

\begin{proof}
  See~\cite[Section~4]{Stadlbauer2004}. The value of the measure is obtained from a stable region for an associated natural extension, which is denoted~$Y$ in~\cite{Stadlbauer2004}.
  Note that this construction uses the fact that~$R$ is an ideal polygon. We refer to~\cite{KatokUgarcovici2017,AbramsEtAl2026} for the construction of a natural extension in cases when the fundamental domain has vertices interior to~$\D$.
\end{proof}

Note that~$\xi(x)$ takes positive values, and is smooth and bounded outside any neighborhood of~$V$. In particular, we have
$$ \xi(x) \asymp 1, \qquad (x\in K). $$
We also notice that by the invariance of $\zeta(\df x, \df y)$ and a change of variable we have
\begin{align}\label{chva}
  |g_w'(x)|    \xi(g_{w}(x)) =  \int_{ y\in [\hat w]} \frac{\zeta(\df x, \df y)}{\df s(x)}
\end{align}
for all $w=a_1\dotsb a_m\in\Sigma^*$ and $x\in[a]$ with $wa$ admissible.

\subsubsection{The invariant measures for~$\cF_c$ and~$\cFAMU$}\label{inme}

We proceed in a similar fashion for~$\cF_c$.
For all~$a\in \Sigma$ and~$x\in [a]$, we define
$$ B_x := \begin{cases} [a], & (x\in K_a), \\ [a] \cup [a-1], & (x\in [a,\hat a + 1]\setminus K_a), \\ [a] \cup [a+1] , & (x\in [a,\hat a - 1]\setminus K_a). \end{cases} $$
Then, we define
\begin{align}\label{eq:def-invar-Fc}
  \varphi(x) = \int_{\dD \setminus B_x} \frac{\zeta(\df x, \df y)}{\df s(x)}.
\end{align}
By construction, we have~$\inf_{x\in \dD} \dist(x, \dD\setminus B_x) > 0$, and thus
$$ \varphi(x) \asymp 1, \qquad (x\in \dD). $$
Notice also that~$\varphi$ is positive, smooth outside of the border of~$\partial K$, and that~$\varphi(x) = \xi(x)$ for all~$x\in K$.

\begin{proposition}\label{prop:varphi-invar}
  The measure~$\varphi(x) \df s(x)$ is invariant under~$\cF_c$.
\end{proposition}
\begin{proof}
  Although this can be checked directly, we introduce the underlying natural extension, from which the definition of~$\varphi(x)$ was derived.

  Let~$A\subset \dD$ consist of those elements~$x = [a_1, a_2, \dotsc]$ not in~$C(G)$, for which~$[a_j, a_{j+1}]\subset K$ for infinitely many values of~$j$. This is the recurrence set for~$\cF_K$, and it is a full measure set~\cite[Theorem~2]{Stadlbauer2004}. For all~$x\in A$, we let~$g_x\in G$ be the element for which~$\cF_c(x) = g_x^{-1} x$, and we define
  $$ \cT_c(x, y) := (g_x^{-1} x, g_x^{-1} y), \qquad ((x, y) \in A \times A). $$
  We define a non-trivial subset of~$A\times A$ on which~$\cT_c$ acts and is invertible. In concrete examples, this set can be guessed numerically using the method in~\cite{ArnouxSchmidt2019}. In our case, we define
  \begin{align*} 
    N_0 &{} := \bigcup_{\substack{a_1,a_2,b\in\Sigma\\ \ell_c(\hat a_2\hat a_1 b )=2}}[a_1,a_2]\times[b] = \bigcup_{x\in \dD} \{x\} \times (\dD\setminus B_x), 
  \end{align*}
  and
  $$ N := N_0 \cap (A\times A). $$
  Notice that~$\varphi(x) = \int_{y\in\dD} \1_{(x, y) \in N} \frac{\zeta(\df x, \df y)}{\df s(x)}$. In order to prove the proposition, following the argument in \emph{e.g.}~\cite[section~3.1]{Series1985}, it suffices to show that~$\cT_c$ is one-to-one from~$N$ to itself.

  We construct the inverse~$\cS$ of~$\cT_c$ as follows. Let~$(x,y)\in  N$ with~$x=[a_1a_2\dotsb]$ and~$y=[b_1b_2\dotsb]$. Then, let~$j\geq1$ be the maximal integer such that~$\ell_c(\hat a_1 w)=1$ with~$w=b_1\dotsb b_j$. Since~$y\in A$ such a~$j$ exists. Moreover, since~$\ell_c(\hat a_2\hat a_1 b_1)=2$ it follows that also~$\ell_c(\hat a_2\hat a_1 w)=2$.
  By~\eqref{eq:invlc} we have that~$x\in D_c(\hat w)$ and we let 
  $$
  \cS(x,y):=(h_{\hat w}( x), g_w^{-1}(y))=\big([\hat b_j \dotsb \hat b_1 a_1 \dotsb],[b_{j+1}b_{j+2}\dotsb]\big).
  $$

  This belongs to~$N$ since, by definition of~$j$, we have~$\ell_c(b_{j-1}b_jb_{j+1})=2$ (or~$\ell_c(\hat a_1 b_1 b_2)=2$ if~$j=1$).
  By construction~$\cT_c \circ \cS=\id $. We need to show that also~$\cS \circ \cT_c=\id$.
  Let~$(x,y)\in N$ as above and let~$j\geq1$ be such that~$1=\ell_c(w a_{j+1})<\ell_c(w a_{j+1}a_{j+2})=2$ with~$w=a_1\dotsb a_{j}$.
  By~\eqref{eq:altm} we have
  $$
  \cT_c(x,y)=(g_w^{-1} x,g_w^{-1} y)=\big([a_{j+1}a_{j+2}\dotsb], [\hat a_{j}\dotsb \hat a_{1}b_1b_2\dotsb]\big).
  $$
  This also belongs to~$N$ since~$\ell_c(\hat{a_{j+2}} \hat{a_{j+1}} \hat{a_j})=2$ by~\eqref{eq:invlc}, and it immediately follows that~$(\cS\circ \cT_c)(x,y)=(x,y)$.
\end{proof}

For completeness, and since it could be of independent interest, we also compute an invariant measure for~$\cFAMU$, although it will play no role in the paper.

Given $a\in\Sigma$, let $w_*^\pm=[a_1,a_2,\dotsb]$ with $a_1=a\pm1$ and $a_{j+1}=\hat a_{j}\pm(-1)^{j}$ for $j\geq1$. Notice that $w_*^\pm$ are periodic points of the Bowen series map but are not cusps, because the period is not a cuspidal cycle~\cite[Lemma~4.9]{ArtigianiEtAl2016}. In particular, they are hyperbolic points.

For~$a\in\Sigma$, let~$Y_a$ be the interval in~$\dD$ from~$w_*^-$ to~$w_*^+$ in the counter-clockwise direction. We note that~$Y_a$ containes a neighborhood of~$[a]$. We can write the complement to~$Y_a$ in~$\dD$ as
\[
  \dD\setminus Y_a=\bigcup_{\substack{b_1\dotsb b_m\in\Sigma^*,\text{ $m\geq1$ odd}\\
      b_1=a\pm 1 \Rightarrow  b_{j+1}=\hat b_{j}\pm (-1)^j,\ 1\leq j<m-1
      \\  b_m\neq \hat b_{m-1}\pm(-1)^{m-1},
    }}[b_1,\dots,b_m]
\]
where~$b_{m-1}$ in the last line of the subscript should be interpreted as~$b_0:=\hat a$ if~$m=1$, so that when~$m=1$ the condition becomes~$\abs{b_1-a}>1$.

We let
$$ \varphi_M(x) := \int_{\dD\setminus Y_a} \frac{\zeta(\df x, \df y)}{\df s(x)} \qquad x\in[a]. $$
Since~$\dist([a], \dD \setminus Y_a)>0$ for all~$a\in \Sigma$, we have~$\varphi_M(x) \asymp 1$ for all~$x\in \dD$.

\begin{proposition}
  The measure~$\varphi_M(x)\df s(x)$ on~$\dD$ is invariant under~$\cFAMU$.
\end{proposition}
\begin{proof}
  Analogously to the earlier cases, we let~$\cT_M(x, y) := (g_x^{-1} x, g_x^{-1} y)$
  with~$g_x$ such that~$\cFAMU(x) = g_x^{-1} x$, and we define
  \begin{align*}
    N_0:={}& \bigcup_{a\in \Sigma} [a] \times (\dD \setminus Y_a), \\
    N :={}& (A \times A) \cap N_0.
  \end{align*}
  Again it suffices to show that the map~$\cT_M$ is one-to-one on~$N$.
  Let~$(x,y)\in  N$ with~$x=[a_1a_2\dotsb]$ and~$y=[b_1b_2\dotsb]$.
  We define 
  $$
  \cS(x,y)=(g_w^{-1}( x), g_w^{-1}(y))=\big([\hat b_j \dotsb \hat b_1 a_1 \dotsb],[b_{j+1}b_{j+2}\dotsb]\big),
  $$
  for~$w=b_1\dots b_j$ and~$j\geq1$ defined as follows.
  \begin{itemize}
    \item If~$|\hat b_1-b_2|>1$, then~$j:=1$.
    \item If~$b_2=\hat b_1 + \delta$ for some~$\delta\in \{\pm 1\}$, then we let~$n\geq1$ be the minimal integer such that~$b_{n+2}\neq \hat b_{n+1} + \delta$. Let~$r\geq1$ be the minimal integer such that~$b_{n+r+1}\neq \hat b_{n+r} + \delta (-1)^r$. 
    Then we take~$j:=n$ if~$r$ is even and~$j:=n+1$ if~$r$ is odd.
  \end{itemize}
  By construction it can be checked that~$\cS(x, y) \in N$ and~$\cT_M\circ \cS(x, y) = (x, y)$. We leave to the interested reader to verify that~$\cT_M(N) \subset N$ and~$\cS\circ \cT_M = \id$.
\end{proof}

\begin{remark}
  In~\cite{Marchese2018}, a technical assumption is made on the fundamental domain of the group (Assumption 2.1 \emph{ibid.}). We believe that the metric estimates on inverse branches which we work out for~$\cF_c$ in Section~\ref{sec:estim-cusp-inverse-branch} could be useful to relax this assumption.
\end{remark}

\section{Metric properties of inverse branches} \label{sec:est:apriori}

\subsection{Distortion and expansion estimates}

We quote the following uniform bounded distortion property from~\cite[Lemma~4.3]{Stadlbauer2004},
see also~\cite[p.~164-165]{BowenSeries1979}.
\begin{lemma} \label{lem:bounded-dist}
  There exists~$C\geq0$ such that for all~$h\in \cH^*$ we have
  \begin{equation}
    \abs{h''(x)} \leq C \abs{h'(x)}, \qquad (x\in \Dom(h)). \label{eq:bounded-dist}
  \end{equation}
\end{lemma}
\begin{proof}
  The set~$B = K$ satisfies Lemma~4.3 from~\cite{Stadlbauer2004}. Let~$C>0$ be the value given by this lemma. Let~$h\in \cH^*$, $h=h_{w,r}$ for some~$r\geq 0$. Thus, for any~$x\in \Dom(h)$, if~$n:=\abs{w}$, then~$\cB^{n}(h(x)) = x \in \Dom(h) \subset K$. We deduce\footnote{In the proof of~\cite[Lemma~4.3]{Stadlbauer2004} it is assumed that $\im(h)\subset K$, however this is not actually used in the proof, nor is present in the statement of the Lemma. See also Proposition~\ref{prop:contraction} below.}
  $$ \abs{\frac{h''}{h'}(x)} = \abs{\frac{D^2(\cB^{n})}{D(\cB^{n})^2}(h(x))} \leq C $$
  where~$D$ stands for the derivative, and the last bound is~\cite[Lemma~4.3]{Stadlbauer2004}.
\end{proof}
\begin{corollary}\label{cor:acs}
  For~$h, k\in \cH^*\setminus\{\id\}$ and~$x\in \Dom(h)\cap\Dom(k)$, we have
  \begin{align}
    & \abs{\frac{h''}{h'}(x) - \frac{k''}{k'}(x)} \asymp \abs{h^{-1}(0) - k^{-1}(0)}, \label{eq:comparison-distortion} \\
    & \abs{h'(x)} \asymp \abs{h'(0)}, \label{eq:h'-x-0} \\
    & \abs{h(x) - h(0)} \ll \abs{h'(x)}. \label{eq:diffh-0-h'}
  \end{align}
\end{corollary}
\begin{proof}
  Write
  $$ h(x) = \frac{\alpha x + \beta}{\bar \beta x + \bar \alpha} $$
  and similarly we denote~$\alpha', \beta'$ the components of~$k$. Since~$h\neq \id$, we have~$\beta\neq 0$, and since~$G$ does not contain elliptic elements, we have~$\alpha\neq 0$. The same holds for~$\alpha', \beta'$. By discreteness of~$G$, we have~$\beta \gg 1$, and therefore~$\abs{\alpha/\beta} = \sqrt{1+1/\abs{\beta}^2} \asymp 1$. The bounded distortion property~\eqref{eq:bounded-dist} implies in our case
  $$ \abs{x + \bar\alpha / \bar\beta}\asymp \abs{x + \beta/\alpha} \asymp 1 $$
  and similarly for~$\alpha', \beta'$.
  Indeed the lower-bound is exactly~\eqref{eq:bounded-dist}, while the upper-bound holds since both~$x$ and~$\alpha/\beta$ are bounded. Then,
  $$ \abs{\frac{h''}{h'}(x) - \frac{k''}{k'}(x)} = \frac{2\abs{\beta/\alpha-\beta'/\alpha'}}{\abs{(x+\beta/\alpha)(x+\beta'/\alpha')}} \asymp \abs{\beta/\alpha - \beta'/\alpha'}. $$
  This proves the first item. The second item follows by the simple computation
  $$ \abs{h'(x)} = \abs{h'(0)} \cdot \bigl|\tfrac{\alpha}{\beta}\bigr|^2 \cdot \abs{x+\bar\alpha/\bar\beta}^{-2} \asymp \abs{h'(0)}. $$
  Similarly, the third item follows since
  \begin{equation*} \abs{h(x) - h(0)} = \abs{h'(x)} \cdot \abs{\tfrac{\beta}{\alpha}}\cdot \abs{x+\bar\alpha/\bar\beta} \asymp \abs{h'(x)}. \qedhere
  \end{equation*}
\end{proof}

We proceed to show the uniform expansion property, expressed in terms of inverse branches. This is essentially extracted from~\cite[Proposition~4.4]{Stadlbauer2004}.
\begin{lemma} \label{lem:contraction}
  For some~$\rho\in [0, 1)$, all~$r\geq 1$ and~$h\in \cH^r$, we have
  \begin{equation}
    \abs{h'(x)} \ll \rho^r\qquad \forall x\in\Dom(h)\label{eq:contraction}
  \end{equation}
  with an absolute implied constant.
\end{lemma}
\begin{proof}
  Let~$h\in \cH^r$, $h=h_{w, r}$ with~$w\in \Lambda_K(r)$. Then proceeding as in~\cite[p.~145]{Aaronson1997} we have
  $$ \meas(\Dom(h)) |h'(x)| = \int_{\Dom(h)} |h'(x)|\df s(y) \ll \int_{\Dom(h)} |h'(y)| \df s(y) = \meas(\im(h)) $$
  where we used~\eqref{eq:h'-x-0} and $\meas$ denotes the measure on $\dD$. We have~$\meas(\Dom(h))\gg1$, and therefore
  $$ |h'(x)| \ll \meas([w]). $$
  This implies at once
  \begin{equation}
    \|h'\|_\infty \ll 1 \qquad (h\in \cH^r).\label{eq:hprime-bound-unif}
  \end{equation}
  We then argue as in~\cite[Prop.~4.4]{Stadlbauer2004}: By~\cite[Prop.~4.5]{Series1981}, $\sup_{\abs{w}=n} \meas([w]) = o(1)$ as~$n\to \infty$. Therefore, since~$\abs{w} \geq r$, we may find~$r_0\geq 1$ such that
  \begin{equation}
    \rho := \sup_{h\in \cH^{r_0}} \|h'\|_\infty < 1.\label{eq:hprime-def-rho}
  \end{equation}
  Then choosing~$r\geq 1$ arbitrary, we write~$r = qr_0 + s$ with~$q\geq 0$, $0\leq s < r_0$. We write accordingly
  $$ h = h_1 \circ \dotsb \circ h_q \circ h_{q+1}, $$
  where~$h_j\in \cH^{r_0}$ for~$j\leq q$ and~$h_{q+1} \in \cH^s$ (or~$h_{q+1} = \id$ if~$s=0$).
  We deduce
  $$ \|h'\|_\infty \leq \prod_j \|h_j'\|_\infty \ll \rho^q \ll (\rho^{1/r_0})^r $$
  where we have used~\eqref{eq:hprime-bound-unif} for~$j=q+1$ and~\eqref{eq:hprime-def-rho} for~$j\leq q$. Reinterpreting~$\rho$ gives our claim.
\end{proof}

\subsection{Estimates for cuspidal inverse branches}
\label{sec:estim-cusp-inverse-branch}

When studying the convergence of the relevant transfer operator, we shall need an estimate for the derivatives of the elements of $\cH_c$. First, given a cuspidal cycle $w$, we describe $[w^\infty]:=[w, w, \dotsc]$ and its stabilizer.

\begin{lemma}\label{lem:parabolic1}
  Let~$w=a_1\dotsb a_r$ be a cuspidal cycle. Then:
  \begin{enumerate}
    \item The point~$[w^\infty] = [w, w, \dotsc]$ belongs to~$V$; it is the right-end boundary of~$[a_1]$ if~$w$ is right-cuspidal, and otherwise it is the left-end boundary of~$[a_1]$.
    \item The element~$g_w\in G$ is parabolic, and fixes~$[w^\infty]$.
    \item If $w$ is minimal, then the stabilizer of $[w^\infty]$ is generated by $g_w$.
  \end{enumerate}
\end{lemma}
\begin{proof}
  Assume that~$w$ is right cuspidal, the other case being analogous. Since~$w$ is a right-cuspidal cycle, the two sides~$[a_1]$ and~$[\hat a_r]$ share a common vertex~$v$, say, and~$v$ is the left endpoint of~$[a_{\hat r}]$. Let~$b\in \Sigma$ be such that~$v\in [\hat a_r, b]$. Since the arcs~$[\hat a_r]$ and~$[\hat a_r, b]$ share their left-endpoint, so do the arcs~$g_{a_r} [\hat a_r] = \dD \setminus [a_r]$ and~$g_{a_r}[\hat a_r, b] = [b]$, and therefore the left endpoint~$[b]$ coincides with the right endpoint of~$[a_r]$. This implies~$b = a_r + 1 = \hat a_{r-1}$. We conclude that~$v\in [a_1] \cap [\hat a_r, \hat a_{r-1}]$ and therefore in particular~$g_{a_r}v \in [a_r] \cap [\hat a_{r-1}]$. Iterating this argument~$r$ times, we deduce that~$g_w v \in [a_1]\cap [\hat a_r] = \{v\}$, and so~$v = g_w v$. Since~$v\in C(G)$, we deduce that~$g_w$ is parabolic. On the other hand, we also deduce,
  $$ v \in \bigcap_{\ell\geq 1} \underset{\ell \text{ times}}{\underbrace{[w, \dotsc, w]}} = [w^\infty].$$
  This shows the first two points. 

  Finally, let $\gamma\neq \id $ be in the stabilizer of $v=[w^\infty]$. Since $\{g_{a}\}_{a\in\Sigma}$ generates $G$, there exists $m\geq 1$ and $b_1,\dots,b_m\in G$ such that $\gamma=g_{b_1}\cdots g_{b_m}$. We assume that $m$ is minimal, so that for $i=1,\dots, m-1$ we have $g_{b_i}\neq g_{b_{i+1}}^{-1}$, i.e. $b_{i}\neq \hat b_{i+1}$. In particular, $w':=b_1\dotsb b_m$ is admissible. We have $v=g_{w'}v\in[w']$ and $v=g_{w'}^{-1}v\in[\hat w']$, so that $v\in[w']\cap[\hat w']$ and this intersection is non-empty if and only if $w'$ is a cuspidal cycle. Since $w$ is minimal, this implies that either $w'=w^n$ or $w'=\hat w^{n}$ for some $n\geq1$, thus proving the claim.
\end{proof}
We now prove an asymptotic estimate for $h=h_w\in \cH_c$ when $w$ is a cuspidal cycle. 

\begin{lemma}\label{lem:parabolic}
  Let~$w=[a_1,\dots,a_r]$ be a cuspidal cycle. Then,
  \begin{enumerate}
    \item For all $x\in\dD$ with $\dist(x-[w^\infty])\gg1$ we have 
    $$ |x-g_{w^\ell}^{-1}\infty| \gg_w 1. $$
    uniformly in $\ell\geq1$.
    \item 
    There  exists a constant $C(w)>0$ satisfying $C(w)=C(\hat w)$ such that for all~$x\in\dD \setminus [w^\infty]$, we have
    $$ |g_{w^\ell}'(x)| \sim \frac{C(w)}{\ell^2 d(x, [w^\infty])^2} $$
    as~$\ell\to\infty$. More precisely, 
    $
    |g_{w^\ell}'(x)| = \frac{C(w)}{\ell^2 d(x, [w^\infty])^2}(1+O_w(1/\ell))
    $
    uniformly in $x$ satisfying~$\dist(x, [w^\infty])\gg 1$.
  \end{enumerate}
\end{lemma}
\begin{proof}
  Write
  $$ g_{w^\ell} = \pmat{\alpha_\ell & \beta_\ell \\ \bar \beta_\ell & \bar \alpha_\ell}, \qquad \abs{\alpha_\ell}^2 - \abs{\beta_\ell}^2 = 1. $$
  Since~$g_w$ is parabolic, we have~$\alpha_\ell + \bar \alpha_\ell = \pm 2$ for some choice of sign~$\pm$, and then
  \begin{equation}\label{eq:v-alphabeta}
    v = \pm 1/ \bar \beta_\ell - \bar\alpha_\ell / \bar\beta_\ell. 
  \end{equation}
  However, again since~$g_w$ is parabolic, we have~$\abs{\beta_\ell} = \ell\abs{\beta_1}\to\infty$. By discreteness of $G$ and since $w\neq\epsilon$ we have $\beta_1\gg1$ and thus
  $$ z_\ell :=g_{w^\ell}^{-1}\infty= -\bar\alpha_\ell / \bar \beta_\ell = v+O(1/\ell). $$
  In particular, $z_\ell\to v$.
  But since~$\abs{\alpha_\ell}^2 > \abs{\beta_\ell}^2$, the point~$z_\ell$ lies outside~$\bar\D$. Let~$B$ be a neighborhood of~$v$ in~$\dD$. We deduce that
  \begin{equation}\label{eq:v-min}
    \min_{x\in\dD\setminus B} d(x, z_\ell) \gg 1, 
  \end{equation}
  which proves (1).
  Moreover, for~$x\in \dD\setminus B$ by~\eqref{eq:v-alphabeta} and~\eqref{eq:v-min} we have
  $$ \abs{g_{w^\ell}'(x)} = \frac1{\abs{\bar \beta_\ell x + \bar \alpha_\ell}^2} = \frac1{\abs{\beta_\ell}^2 d(x, z_\ell)^2} = \frac1{\ell^2 \abs{\beta_1}^2 d(x, v)^2}(1+O(1/\ell)) $$
  uniformly in~$x\in\dD\setminus B$. This yields (2) with $C(w)=|\beta_1|^{-2}$.
\end{proof}

\begin{corollary}\label{cor:parabolic}
  Let~$h=h_{w}^c\in\cH_c$ with~$w$ cuspidal. Let $w_0$ be a minimal cuspidal cycle\footnote{Notice that $w_0$ and $\ell$ are unique unless $|w|=1$} such that $w=w_0^\ell w_1$ for some $\ell\geq0$ and some prefix $w_1\neq \epsilon$ of $w_0$. Then 
  \begin{align*}
    \|h'\|_\infty \asymp (1+\ell)^{-2}.
  \end{align*}
\end{corollary}
\begin{proof}
  Let $[w_0^\infty]=[a_1a_2\dotsb]$ and $|w_1|=m\ll1$.  For~$x\in \Dom(h^c_w)$, we have
  $$ (h^c_w)'(x) = g_{w_0^\ell}'(g_{w_1}(x)) g_{w_1}'(x). $$
  On the one hand, the pole of a matrix in~$\SU(1, 1)\setminus\{\id\}$ lies outside~$\bar\D$, so that~$\|g_{w_1}'\|_{\infty, \dD} < \infty$ for each~$w_1$. Since there are finitely many possibilities for $w_0$ and thus for~$w_1$, we deduce~$|g_{w_1}'(x)| \asymp 1$. Also, $g_{w_1}(\Dom(h^c_w))= [a_1,\dots,a_{m+1}]\setminus [a_1,\dots,a_{m+2}]$ and thus~$\dist(g_{w_1}(\Dom(h^c_w)), [w_0^\infty])\gg1$. By Lemma~\ref{lem:parabolic} we then have~$|g'_{w_0^\ell}(g_{w_1}(x))|\asymp 1/\ell^2$ for $\ell$ sufficiently large. This concludes the proof since the statement is obvious for $\ell\ll1$.
\end{proof}
We also record some results which will be useful in future applications.
\begin{proposition}\label{prop:metric-Hc}
  For all $h\in\cH_c$ we have
  \begin{align}\label{eq:Hc-dist-V}
    \dist(\im(h),V)\asymp \|h'\|_\infty^{1/2}.
  \end{align}
  Moreover, if $h_1\in\cH_c^*\setminus\{\id\}$, $h_2=h_{w_2}^c\in\cH_c$, and~$y\in\Dom(h_2)$ with $z=h_2(y)\in\Dom(h_1)$, then 
    \begin{align}\label{eq:Hc-dist-Domh1}
    \dist(z,\partial\Dom(h_1))\gg \begin{cases}
    \|h_2'\|_\infty^{1/2} &\text{if $|w_2|\geq2$},\\
    \dist(y,V) &\text{if $|w_2|=1$}.
 \end{cases}
  \end{align}
\end{proposition}
\begin{proof}
  Let $w$ be a primitive cuspidal cycle and $\ell\geq1$. Then, with the same notation as in the proof of Lemma~\ref{lem:parabolic}, for $x\in\dD$ with $|x-[w^\infty]|\gg1$ we have
  \begin{align*}
    g_{w^\ell}(x)-[w^\infty]= \frac1{ \bar \beta_\ell}\Big(\pm1 - \frac{1}{ \bar\beta_\ell x+\bar\alpha_\ell }\Big)
  \end{align*}
  and thus, since $|\beta_\ell|\asymp \ell$, $|\bar\beta_\ell x+\bar\alpha_\ell|=|g_{w^{\ell}}'(x)|^{-1/2}\asymp \ell$, we deduce
  \begin{align}\label{eq:dist-gwl-winf}
    |g_{w^\ell}(x)-[w^\infty]|\asymp (1+\ell)^{-1}
  \end{align}
  for $\ell$ large enough, and thus for all~$\ell\geq 0$ since~\eqref{eq:dist-gwl-winf} holds trivially if $\ell \ll 1$. 

  Now, let $h=h_w^c\in\cH_c$ and $x\in\Dom(h)$. Writing $w=w_0^\ell w_1$ as in Corollary~\ref{cor:parabolic} and proceeding as in its proof using~\eqref{eq:dist-gwl-winf} we deduce that
  \begin{align*}
    |h(x)-[w_0^\infty]|\asymp \|h'\|_\infty^{1/2}.
  \end{align*}
  Equation~\eqref{eq:Hc-dist-V} then follows since~$\dist(\im(h),V\setminus\{[w_0^\infty]\})\asymp 1$.

  We now prove~\eqref{eq:Hc-dist-Domh1}. Every boundary point of $\Dom(h_1)$ is either a vertex $v_1\in V$ or a point $g_av_2\in W$ for some $a\in\Sigma$ and $v_2\in V$. In the first case, by~\eqref{eq:Hc-dist-V} we have $\dist(v_1,\im(h_2))\gg \|h_2'\|^{1/2}_\infty$, which also gives the claimed bound when $|w_2|=1$, since in this case $|h_2'|\asymp 1$.
In the second case, since $g_av_2$ is in the interior of $[a]$, we have $|h_2(y)-g_av_2|\gg1$ unless the first letter of $w_2$ is $a$. We may therefore assume that this is the case. Then, if $|w_2|=1$, we have
$$|h_2(y)-g_av_2|=|g_a y-g_a v_2|\gg |y-v_2|\geq \dist(y,V),$$
as desired.  Finally, assume $|w_2|\geq2$. Write $w_2=aw$ for some cuspidal $w\in\Sigma^*$. Then,
  $$|h_2(y)-g_av_2|=|g_{a}g_wy-g_av_2|\gg |g_{w}y-v_2|\gg \|(h^c_w)'\|^{1/2}_\infty\asymp \|h_2'\|^{1/2}_\infty,$$
  where we used~\eqref{eq:Hc-dist-V} once again.
\end{proof}

It will also be convenient to have a version of Lemma~\ref{lem:bounded-dist}, Corollary~\ref{cor:acs} and Lemma~\ref{lem:contraction} for $h\in\cH_c$.

\begin{proposition}\label{prop:contraction}
  For all $h\in\cH_c^*\setminus\{\id\}$ we have 
  \begin{align}
    &\|h''/h'\|_\infty\ll1,\label{eq:Hc-bounded-dist}\\
    &|h'(x)|\asymp |h'(0)|,\qquad \forall x\in\Dom(h).\notag
  \end{align}
  Moreover, there exists~$\rho\in [0, 1)$ such that for all~$r\geq 1$ and~$h\in \cH_c^r$ we have
  \begin{equation}\label{eq:Hc-contraction}
    \|h'\|_\infty \ll \rho^r
  \end{equation}
\end{proposition}
\begin{proof}
  It suffices to prove~\eqref{eq:Hc-bounded-dist} since the remaining claims then follow as for $\cH^*$.

  Let $h=h_{w,r}^c\in\cH^r$ with $w=a_1\dotsb a_n$. We can assume $|w|$ is sufficiently large, since the claim is trivial if $|w|\ll1$.

  Let $\psi_w$ be the indifferent or repelling fixed point of $g_w$. Then, as in~\cite[Lemma~4.3]{Stadlbauer2004}\footnote{We warn the reader that in~\cite{Stadlbauer2004} the roles of $g_w$ and $g_{w}^{-1}$ are swapped.}, we have that $|\psi_w- g_w^{-1}\infty|\to 0$ as $|w|\to\infty$. Since $\frac{h''}{h'}(x)=-2(x-g_w^{-1}\infty)^{-1}$, it is thus sufficient to show that $\dist(\psi_w,\Dom(h))\gg1$.

  Next, repeating the argument in~\cite{Stadlbauer2004}, we notice $g_w[\hat w]=\overline{\dD\setminus[ a_1]}$ and we claim that~$\psi_w\in[\hat a_n]$. Indeed, if $\hat a_n\neq a_1$, then $g_w[\hat w]\supset[\hat \omega]$ and so $g_w$ has a fixed point (which has to be repellent or indifferent) in $ [\hat w]\subset[\hat a_n]$, i.e. $\psi_w\in[\hat a_n]$.
  If instead $\hat a_n= a_1$, then any possible fixed point of $g_w$ in $\dD\setminus[ a_1]=\dD\setminus[\hat a_n]$ would have to be in $[\hat w]\cap (\dD\setminus[\hat a_n])=\emptyset$. Thus, $\psi_w\notin\D\setminus[\hat a_n]$ and so $\psi_w\in[\hat a_n]$.

  Finally, recalling the notation $b^{\pm}:=\hat b\pm 1$, we notice that
  $$\Dom(h)\cap \big([\hat a_n+1]\cup[\hat a_n]\cup [\hat a_n-1]\big)\subset \overline{([ a_n^+]\setminus[a_n^+,a_n^{++}])\cup( [ a_n^-]\setminus[a_n^-,a_n^{--}] )}.$$
  The set on the right is bounded away from $[\hat a_n]$. Since the same also holds for the complementary part of $\Dom(h)$, we get
  $$\dist(\psi_w,\Dom(h))\gg \dist([\hat a_n],\Dom(h))\gg1,$$
  as desired.
\end{proof}

\section{Transfer operators} \label{sec:transfer}

Recall that~$\Theta\subset \GL(d, \R)$ is a finite subgroup. We denote
$$ \dD_\Theta := \dD \times \Theta, \qquad K_\Theta := K \times \Theta  $$
and define $\dD'_\Theta,\dots$ in the same way.

We will be interested in maps
$$ f: \dD_\Theta \to \C, $$
and their orbits under the transfer operator~$\LL$ associated with~$\cF_K$. In order to study this transfer operator, we will also need to consider the transfer operator~$\CC$ associated with~$\cF_c$.

\subsection{Skew-products}
\label{sec:skew-products}

Recall the definition~\eqref{eq:def-vthg} of the family of maps~$(\vth_g)$, and the cocycle property~\eqref{eq:vth-cocycle}.
In the remainder of the paper, we make the following two hypotheses on~$(\vth_g)$.
\begin{enumerate}
  \item For each~$g\in G$, with~$g=g_{a_1}\cdots g_{a_n}$ with~$a_1\dotsb a_n\in\Sigma^*$,
  the map~$\vth_g$ is constant on~$[a]\setminus V$ for all~$a\in \Sigma\setminus\{[\hat{a_n}]\}$.
\end{enumerate}
Regarding the assumption~(1), recall from Section~\ref{sec:usef-prop-fund} that the expression~$g = g_{a_1} \cdots g_{a_n}$ is unique.
Thanks to the assumption~(1) we can write, for any~$w\in \Sigma^*$ and~$a$ with~$wa$ admissible,
\begin{equation}
  \label{eq:vth-constant}
  \vth_w(a) := \vth_{g_w}|_{[a]\setminus V} \in \Theta.
\end{equation}
\begin{enumerate}  \stepcounter{enumi}
  \item There exists~$r_0\in \N$ with the following property. Given any~$b_1, b_2\in \Sigma$ and~$\vth_0\in\Theta$, we may find~$w\in \Sigma^*$ with~$b_1 w b_2$ admissible, such that~$\ell_K(b_1 w b_2) = r_0$ and
  \begin{equation}
    \vth_{b_1 w}(b_2) = \vth_0.\label{eq:vth-mixing}
  \end{equation}
\end{enumerate}
Note that with the notation~\eqref{eq:vth-constant}, the cocycle property implies
\begin{equation}
  \vth_{a_1 \dotsb a_n b_1 \dotsb b_m}(c) = \vth_{b_1 \dotsb b_m}(c) \vth_{a_1 \dotsc a_n}(b_1). \label{eq:vth-cocycle-letter}
\end{equation}
We also note, for potential use in future applications, that assumption~(1) is invoked in our work only for~$[a] \cap \Dom(h) \neq \varnothing$ where~$h = h_g$ is an inverse branch in~$\cH$ or~$\cH_c$.

We define a map~$\cF_\Theta: (\DcF\setminus W)_\Theta \to K_\Theta$ by setting
$$ \cF_\Theta(x, \vth_0) := (\cF_K(x), \vth_0 \vth_g (x)), \qquad (\cF_K(x) = g x, g \in G). $$
Using the cocycle rule~\eqref{eq:vth-cocycle}, the inverse branches of~$\cF_\Theta$ are given by
$$ \{ (y, \vth_1) \mapsto (h(y), \vth_1 \vth_h(y)),\  h\in \cH\}, $$
where for $h=h_w$ in $\cH$, $\cH_c$ or~$\cH_c^\sharp$,
\begin{align*}
  \vth_{h}(y):=\vth_{g_w}(y)=\vth_{g_w^{-1}}(h(y))^{-1},\qquad y\in\Dom(h)\setminus V.
\end{align*}

\subsection{The $s$-twisted transfer operator}\label{sec:transferop-stwist}

Define, for all~$s\in \C$, the $s$-twisted transfer operator acting on a map~$f:\dD_\Theta \to\C$ by
\begin{equation}\label{eq:def-twistedop-s}
  \begin{array}{rccl}
    \LL_s[f]: & K_\Theta &\to& \C,\\ & (y, \vth_1) & \mapsto & \displaystyle \ssum{h\in \cH \\ y\in \Dom(h)} \abs{h'(y)}^s f(h(y), \vth_1 \vth_h(y))
  \end{array}
\end{equation}
whenever the sum converges absolutely.

We will also need to define analogous operators for the map~$\cF_c$. 
We will give two versions, one with the inverse branches~\eqref{eq:def-Hcr} in~$\cH_c$ and one with its restricted version~$\cH_c^{\sharp}$ defined in~\eqref{eq:def-Hcsharp}.
Given~$f:\dD_\Theta \to \C$ as above, we let
\begin{align*}
  \CC_s[f](y, \vth_1)&= \ssum{h\in \cH_c \\ y\in \Dom(h)}  |h'(y)|^s f(h(y), \vth_1 \vth_{h}(y) ) ,\qquad (y,\vth_1)\in\dD'_\Theta
\end{align*}
and let~$\CC^\sharp_s[f](y, \vth_1)$ be defined in the same way but where the sum ranges over~$h\in \cH_c^\sharp$. By definition this means that for $(y,\vth_1)\in\dD'_\Theta$
$$ \CC_s^\sharp[f](y,\vth_1) = \begin{cases} \CC_s[f](y,\vth_1), & (y\not\in K), \\ 0 & (y\in K). \end{cases} $$
Note that by the triangle inequality, for~$\sigma, \tau \in \R$ we have
$$ \abs{\CC_{\sigma+i\tau}[f]} \leq \CC_\sigma[\abs{f}] $$
whenever the right-hand side is finite.

We also define the untwisted versions of~$\LL$,~$\CC_s$ and~$\CC^\sharp_s$ by
$$
\LL:=\LL_1,\qquad \CC:=\CC_1, \qquad \CC^\sharp:=\CC^\sharp_1.
$$
Notice that for a non-negative~$f$ we have~$\CC^\sharp[f]\leq \CC[f]$.

Recall the definition of~$\varphi$ from~\eqref{eq:def-invar-Fc}.
By Proposition~\ref{prop:varphi-invar}, we have
$$ \CC[\varphi] = \varphi $$
and thus, since~$\varphi\asymp 1$, for bounded~$f$ we have~$\CC^\sharp[f], \CC[f]\ll 1$. Finally, 
for all~$f$ integrable on~$\dD$, we have
\begin{equation}
  \int_{\dD} \CC[f](y) \df s(y) = \int_{\dD} f(y) \df s(y).\label{eq:transferop-invarc}
\end{equation}
This can also be stated as the fact that the dual~$\CC^*$ of~$\CC$ preserves the Lebesgue measure.

The relevance of~$\CC_s$ and~$\CC^\sharp_s$ is evident from Lemma~\ref{lem:decomp-Hc-HK}.
Indeed, one immediately deduces that for~$y\in K$ and~$\vth_1\in\Theta$, we have 
\begin{align}
  \LL_s[f](y, \vth_1)&= \sum_{m=0}^\infty \CC_s (\CC^\sharp_{s})^m[f](y, \vth_1)\notag\\
  &=\CC_s (\Id-\CC^\sharp_s)^{-1}[f](y, \vth_1),\label{eq:rel-LLs-CCs}
\end{align}
whenever the series defining $\LL_s[f]$ is absolutely convergent.

\begin{remark}
  We mention that the family of $s$-twists~\eqref{eq:def-twistedop-s} was defined and studied in~\cite[Section~5]{Morita1997} in relation with the Selberg zeta function of Fuchsian groups (see~\cite{Mayer1991} for the $\SL(2, \Z)$ case).
\end{remark}

\subsection{Convergence for~$\Re(s)$ below~$1$}\label{sec:conv-res-below-1}

Our next goal is to show that there exists~$\sigma_0 < 1$ such that the sum in the definition~\eqref{eq:def-twistedop-s} of~$\LL_s[f]$ converges normally for all bounded~$f$ and all~$s \in \C$ with~$\Re(s)\geq \sigma_0$.
We start with the following lemma which is relative to the case when~$\Re(s)=1$. 

\begin{lemma}\label{lem:CCsharp-contracts}
  We have~$(\CC^\sharp)^2[\varphi ]< \rho  \varphi$ for some~$\rho\in(0,1)$. 
\end{lemma}
\begin{proof}
  We have~$0=\CC^\sharp[\varphi ](y)< \CC[\varphi ](y)=\varphi(y)$ for all~$y\in K$. If~$y\notin K$ the two sums coincides and one has~$\CC^\sharp[\varphi ](y)= \CC[\varphi ](y)=\varphi(y)$. In particular, letting $f:=\CC^\sharp[\varphi ]$, we have~$0\leq f\leq\varphi$ and~$ f<\varphi$ on~$K$. If~$y\in K$ then~$\CC^\sharp[f](y) = 0 < \CC[\varphi](y)$, as desired. On the other hand if~$y\not\in K$, $y=[b, \dotsc]$, then we can always find~$h\in \cH_c$ with~$y\in \Dom(h)$ and~$h(y)\in K$; it suffices to pick~$h = h^c_a$ for some~$a\in \Sigma$ satisfying~$[a, b] \subset K$, where we used the notation~\eqref{eq:def-hwc}. For this~$h$ we have~$f(h(y))=0$, and therefore
$$  \CC^\sharp[f](y)-\varphi(y) =\CC[f-\varphi](y)\leq - |h'(y)|\varphi(h(y))  < -c \varphi(y),$$
for some $c\in(0,1)$, since~$\varphi\asymp1$ and $|{h_a^c}'|\asymp1$ uniformly for~$a\in\Sigma$.
Together with~$\CC^\sharp[f]=0$ on~$K$, this gives the claimed
bound with~$\rho=1-c$.
\end{proof}

\begin{lemma}\label{lem:concn}
  Let $\eps>0$. Then, for all $1/2+\eps<\sigma\ll1$ we have $\CC_\sigma[\1] \ll_\eps \1$.
\end{lemma}
\begin{proof}
  This follows from Corollary~\ref{cor:parabolic} since there are only finitely possibilities for $w_0$ and $w_1$.
\end{proof}

We deduce the following contraction result for $(\CC^\sharp)^2$ on the left of $\Re(s)=1$ line. 

\begin{lemma}\label{lem:convergence-H-CF}
  There exists~$\sigma_0<1$ and~$\rho<1$ such that for $\sigma_0\leq \sigma\leq 1$ 
  $$ (\CC^\sharp_{ \sigma})^2[\varphi] < \rho \varphi. $$
\end{lemma}
\begin{proof}
  The result follows by continuity combining Lemma~\ref{lem:CCsharp-contracts} and~\ref{lem:concn}.\end{proof}

\begin{proposition}\label{prop:Lst-bounded}
  Let~$f:\dD_\Theta \to \C$ be bounded. There exists~$\sigma_0 \in (0, 1)$ such that the sum in~\eqref{eq:def-twistedop-s} converges absolutely and uniformly for all~$s \in \C$ with~$\Re(s)\geq \sigma_0$.
\end{proposition}

\begin{proof}
  Take~$\sigma_0$ as in Lemma~\ref{lem:convergence-H-CF}.
  Then for all~$m\geq 1$ we have~$(\CC^\sharp_{\sigma_0})^m[f] \ll (\CC^\sharp_{\sigma_0})^m[\varphi]\ll \rho^{m/2} \varphi$, with an implicit constant independent of~$m$. The result then follows by~\eqref{eq:rel-LLs-CCs}.
\end{proof}

\subsection{Spectral properties at~\texorpdfstring{$s=1$}{s = 1}}

Finally we record the properties of $\LL$.
By Proposition~\ref{prop:Lst-bounded}, we know that~$\LL$ is a bounded operator from~$L^\infty(\dD_\Theta)$ to~$L^\infty(K_\Theta)$.
Recall the definition of~$\xi$ in~\eqref{eq:def-inv-meas}.
We extend~$\xi$ to a map on~$\dD_\Theta$ supported inside~$K_\Theta$ by letting it be constant on the second coordinate:
$$ \xi(x, \vth_0) := \begin{cases} \xi(x), & (x, \vth_0) \in K_\Theta, \\ 0, & \text{otherwise}. \end{cases} $$
By~\eqref{eq:def-inv-meas} we have that~$\xi$ is an eigenfunction of eigenvalue~$1$ for~$\LL$,
\begin{equation}
  \LL[\xi] = \xi. \label{eq:invar-LL-xi}
\end{equation}

By Proposition~\ref{prop:xi-invar}, the absolutely continuous measure~$\xi(x) \df s(x)$ is~$\cF_K$-invariant, and analogously to \eqref{eq:transferop-invarc}, we also have
\begin{equation}
  \int_{K_\Theta} \LL[f](y, \vth_1) \df s(y)\df\vth_1 = \int_{\dD_\Theta} f(y,\vth_1) \df s(y) \df\vth_1 \label{eq:transferop-invar}
\end{equation}
for all~$f$ integrable on~$\dD_\Theta$, where~$\df\vth_1$ denotes the uniform counting probability measure on~$\Theta$.

\subsection{Cost functions}\label{cfn}

We let~$\vphi = (\vphi_h)_{h\in \cH}$ be a family of maps~$\vphi_h:\Dom(h)\to \R^d$ which belong to $H^{\kappa_0}$ for some $\kappa_0\in(0,1]$, where we recall $H^\kappa$ is introduced in Section~\ref{fsp0}.
Since we have the disjoint union~$\DcF\setminus W=\cF_K^{-1}(K)=\bigcup_{h\in\cH}\im(h)$, the family~$(\vphi_h)_{h\in \cH}$ can be equivalently reinterpreted as a function~$\Phi_K$ on $(\DcF\setminus W)_\Theta$ defined as
\begin{align}
  \Phi_K(x,\vth_0) &{}:=\vth_0 \vth_{h}(y)^{-1} \vphi_{h}(y)&& x=h(y) \text{ for some~$h\in\cH$}\label{eq:def-PhiK}, \\
  \Phi_K(x) &{}:= \Phi_K(x, 1).\notag
\end{align}
By construction we then also have
\begin{equation}\label{eq:PhiK-vphi}
  \vth_1 \vphi_{h}(y)=\Phi_K(h(y),\vth_1 \vth_h(y)) \quad \forall h\in \cH,\ y\in \Dom(h).
\end{equation}

It will be useful to define $\vphi_h$ also for inverse branches of any depth, by letting~$\vphi_h = 0$ if~$h = \id$ and otherwise
\begin{align}\label{cascs}
  \vphi_h(y)=\sum_{i=1}^{n} \vth_{h_{i+1} \circ \dotsb \circ h_n}(y) \vphi_{h_i}(h_{i+1} \circ \dotsb \circ h_n(y))
\end{align} 
for $h=h_1 \circ \dotsb \circ h_n\in \cH^n$, $h_i\in\cH$.

\begin{remark}
  In the introduction, we discussed for simplicity a situation where the family of maps~$(\vphi_g)$ is defined for all~$g\in G$. In such a situation, the relation~\ref{cascs} holds automatically by the cocycle property~\eqref{xax}.
\end{remark}

\subsection{The transfer operator \texorpdfstring{$\LL_{s,\vt}$}{L_st}}

We can now define the main generating function we will work with. For all~$s\in \C$ with~$\Re(s)\geq\sigma_0$ the quantity defined in Proposition~\ref{prop:Lst-bounded}, all~$\vt\in\R^d$ and all bounded map~$f:\dD_\Theta \to \C$, we let
\begin{equation}\label{eq:def-twistedop-general}
  \begin{array}{rccl}
    \LL_{s, \vt}[f]: & K_\Theta &\to& \C,\\ & (y, \vth_1) & \mapsto & \displaystyle \ssum{h\in \cH \\ y\in \Dom(h)} \abs{h'(y)}^s e^{i \inpd{\vt, \vth_1  \vphi_h(y)}} f(h(y), \vth_1 \vth_h(y)).
  \end{array}
\end{equation}

\begin{remark}\label{rmkex}
  Any bounded function $f:K_\Theta \to \C$ extends trivially to a function on $\dD_\Theta$ by setting it to zero on $\dD_\Theta\setminus K_\Theta$. With a slight abuse of notation, we also denote $\LL_{s, \vt}$ the operator acting on bounded functions on $K_\Theta$, obtained by applying the original operator~\eqref{eq:def-twistedop-general} to this extension. Clearly, in this case~\eqref{eq:transferop-invar} holds also with $\dD_\Theta$ in the right-hand side replaced by $K_\Theta$.
\end{remark}

\section{Moments conditions}\label{sec:moments-conditions}

In this section we shall introduce some hypothesis on the period functions~$(\vphi_h)$ which we will need in order to carry out the functional analysis of the operators~$\LL_{s, \vt}$. 

Given~$\kappa\in(0,1]$, we recall the definition of~$\kappa$-Hölder norm and semi-norm given in Section~\ref{fsp0}.
We also quote from~\cite[Section~4.1]{BettinDrappeau2022} the following properties of $\|\|_{(\kappa)}$:
\begin{align}
  &\|fg\|_{(\kappa)}\leq \|f\|_{(\kappa)}\|g\|_{\infty}+\|g\|_{(\kappa)}\|f\|_{\infty}, && f\in H^\kappa,\label{hld1}\\
  &\|f\circ h\|_{(\kappa)}\ll \|h'\|^\kappa_\infty \|f\|_{(\kappa),\im(h)}, && f\in H^\kappa, h\in\cH^*,\label{hld2}\\
  &\|e^{if}\|_{(\kappa)}\ll\|f\|_{(\kappa')}^{\kappa/\kappa'} &&  f\in H^{\kappa'}, \kappa'\in[\kappa,1],\label{hld3}\\
  & \|{f}\|_{(\kappa)} \ll \|{f}\|_\infty^{1-\kappa/\kappa'} \|{f}\|_{(\kappa')}^{\kappa/\kappa'}&&  f\in H^{\kappa'}, \kappa'\geq\kappa. \label{hld4}
\end{align}
We note that here and below, norms are always implicitly restricted to the domains of the function under consideration, so that in particular in~\eqref{hld2} we mean~$\|f\|_{(\kappa), \im(h)} = \|f\|_{(\kappa), \Dom(f) \cap \im(h)}$.

{\renewcommand{\thehypothesis}{\ensuremath{(M_K)}}
  \begin{hypothesis}[Moments conditions]\label{hyp:moments}
    For given positive numbers~$\alpha, \kappa_0, \lambda$, we have
    \begin{align}
      &\sum_{h\in\cH}  \| h'\|_\infty  \|\vphi_{h}\|_{\infty} ^{\alpha} < \infty, \label{eq:moment-condition} \\
      &\sum_{h\in\cH}  \| h'\|_\infty \| \vphi_{h}  \|^\lambda_{(\kappa_0)}  < \infty. \label{eq:holder-condition}
    \end{align}  
  \end{hypothesis}
}

Note that the sums are over~$\cH$, the inverse branches of~$\cF_K$, see~\eqref{eq:def-H1}.
For our applications in~\cite{BettinEtAla,BettinEtAl3}, it will be crucial to reduce the sum over~$h$ to a smaller set. This is the point of the following definition.

{\renewcommand{\thehypothesis}{\ensuremath{(M_c)}}
  \begin{hypothesis}[Cuspidal moments conditions]\label{hyp:moments-cuspidal}
    For given positive numbers~$\alpha, \kappa_0, \lambda$ there exist a family~$(\vphi_h)_{h\in \cH_c}$ of maps from~$\Dom(h)\setminus V$ to $\R^d$ such that the following conditions are satisfied.
    \begin{enumerate}
      \item We have
      \begin{align}
        &\ssum{h_1, h_2 \in \cH_c} \|h_1'\|_{\infty} \|h_2'\|_{\infty} \| \vphi_{h_1} \|_{\infty,\im(h_2)}^{\alpha} < \infty, \label{eq:moment-condition-cuspidal} \\
        &\ssum{h_1, h_2 \in \cH_c}  \|h_1'\|_{\infty} \|h_2'\|_\infty \| \vphi_{h_1}\|_{(\kappa_0), \im(h_2)}^\lambda < \infty. \label{eq:holder-condition-cuspidal}
      \end{align}

      \item For all $(x,\vth_0)\in (\DcF\setminus W)_\Theta$, we have
      \begin{align}\label{eq:ref-PhiK-Phic}
        \Phi_K(x,\vth_0) = \sum_{j=0}^{r_{x}-1}\Phi_c(\cF_{c,\Theta}^{j}(x,\vth_0)),
      \end{align}
      with~$r_x$ as in~\eqref{eq:def-return-Fc} and $\Phi_c:(\dD'\setminus W)_\Theta\to \R^d$ and $\cF_{c,\Theta}:(\dD'\setminus W)_ \Theta\to \dD_\Theta$ defined by
      \begin{align}
        &\Phi_c(x,\vth):= \vth \vth_{h}(y)^{-1} \vphi_{h}(y),&& x=h(y) \text{ for some~$h\in\cH_c$, $y\in\Dom(h)\setminus V$} \label{eq:def-Phictwo}  \\
        &\cF_{c,\Theta}(x,\vth):=(\cF_c(x),\vth \vth_{h^{-1}}(x))
      \end{align}
      which are well-posed since~$\dD'\setminus W=\bigcup_{h\in\cH_c}h(\Dom(h)\setminus V)$, where the union is disjoint. We also denote\footnote{Equivalently, $ \vth_{h}(x)^{-1}\vphi_{h}(x)=\Phi_c(h(x))$, $\forall h\in \cH_c,\ x\in \Dom(h)\setminus V.$} 
      \begin{align}
      \Phi_c(x):=\Phi_c(x,1).\label{eq:def-Phic}
      \end{align}
    \end{enumerate}
  \end{hypothesis}
}
We remark that condition~\eqref{eq:ref-PhiK-Phic} is natural when, as in our applications, one has functions $(\vphi_{g})_g$ for each $g\in G$ which satisfy the cocycle condition
$$ \vphi_{g_1 g_2}(x) = \vphi_{g_2}(x) + \vth_{g_2}(x) \vphi_{g_1}(g_2x)\qquad \forall g_{1},g_2\in G.$$
Indeed, in this case,~\eqref{eq:ref-PhiK-Phic} amounts to say that the cocycle relation holds for the decompositions~\eqref{eq:decomp-Hc-HK}.
We extend the notation~$\vphi_h$ to~$h\in \cH_c^n$ for any~$n\geq 0$ in the same way as~\eqref{cascs}.

Assuming Hypothesis~\ref{hyp:moments-cuspidal}, for all~$s\in\C$ with~$\Re(s)>1/2$ and all~$\vt\in \R^d$, we define the~$(s, \vt)$-twisted operators~$\CC_{s,\vt}, \CC_{s,\vt}^\sharp$ acting on maps~$f:\dD_\Theta \to \C$ by
\begin{align*}
  \CC_{s, \vt}[f](y, \vth_1) &= \ssum{h\in \cH_c \\ y\in \Dom(h)}  |h'(y)|^se^{i \inpd{\vt, \vth_1 \vphi_{h}(y)}} f(h(y), \vth_1 \vth_{h}(y)),\qquad (y,\vth)\in\dD'_\Theta
\end{align*}
and~$\CC_{s,\vt}^\sharp$ defined in the same way with~$\cH_c^\sharp$ instead of~$\cH_c$.
Notice that~\eqref{eq:ref-PhiK-Phic} implies that the following analogue of~\eqref{eq:rel-LLs-CCs} holds:
\begin{align}
  \LL_{s, \vt}[f](y, \vth_1)&=\CC_{s, \vt} (\Id-\CC^\sharp_{s, \vt})^{-1}[f](y, \vth_1),\qquad (y,\vth_1)\in K_\Theta\label{eq:ref-LLst-CCst}
\end{align}
for~$\Re(s)\geq\sigma_0$.

The advantage of Hypothesis~\ref{hyp:moments-cuspidal} over Hypothesis~\ref{hyp:moments} lies in the fact that the sums over~$h_i$ in~\eqref{eq:moment-condition-cuspidal}--\eqref{eq:holder-condition-cuspidal} are restricted to the set~$\cH_c$, which is easier to parametrize since it consists essentially of powers of finitely many parabolic elements. On the other hand, the domains of~$\vphi_{h_1}$ in conditions~\eqref{eq:moment-condition-cuspidal} and~\eqref{eq:holder-condition-cuspidal} are now larger and in particular can approach the vertices of the fundamental domain, which poses additional challenges in applications.

\begin{lemma}[\ref{hyp:moments-cuspidal} implies~\ref{hyp:moments}]\label{lem:moments-c-implies-K}
Let~$\kappa_0>0$. Suppose that~\ref{hyp:moments-cuspidal} holds for~$\alpha>0$ and~$\lambda>0$. Then~\ref{hyp:moments} holds with the same parameters.
\end{lemma}

\begin{proof}
  We recall the definition~\eqref{eq:def-return-Fc} of the return time~$r_x$ of~$\cF_c$ to~$K$ and its relation~\eqref{eq:return-h-r} with the decomposition of~$h$. Then, given~$h\in\cH$ and $a\in\Sigma$ with $K_a\subseteq \Dom(h)$, from Lemma~\ref{lem:decomp-Hc-HK} we have that $h$ decomposes on $K_a$ as~$h=h_1\circ \dots \circ  h_{r_h}$ as in~\eqref{eq:decomp-Hc-HK}. By the cocycle property~\eqref{eq:ref-PhiK-Phic}, we then have for all~$y\in K_a$ 
  \begin{align*}
    \abs{\vphi_h(y)}^\alpha&\ll\bigg(\sum_{j=1}^{r_h} \abs{\vphi_{h_{j}}(h_{j+1}\circ \cdots \circ h_{r_h}(y))}\bigg)^\alpha\leq  r_{h} ^\alpha \sum_{j=1}^{r_h} \abs{\vphi_{h_{j}}(h_{j+1}\circ \cdots \circ h_{r_h}(y))}^\alpha.
  \end{align*}
  and so letting $h_{r_h+1}:=h_{a,1}^c$, for which $\Dom(h)\cap\im(h_{r_h+1})=K_a$,  we have
    \begin{align}\label{stpacs}
    \|\vphi_h\|_{\infty,K_a}^\alpha&\ll  r_{h}^\alpha \sum_{j=1}^{r_h} \|\vphi_{h_{j}}\|^\alpha_{\infty,\im(h_{j+1})}.
  \end{align}
Now, let
  $$ \tilde\CC[f](y) := \ssum{h\in \cH_c \\ y \in \Dom(h)} \abs{h'(y)} \|\vphi_h\|_{\infty,\im(k) }^{\alpha}f(h(y)),\qquad (y\in\im(k),\ k\in\cH_c)$$
and similarly let~$\tilde\CC^\sharp$ be defined by the same expression restricted to~$h\in \cH_c^\sharp$. We claim that
  \begin{align}
    \tilde\CC[1](y) \ll 1, && (y\in K), \label{eq:bound-Cw-1} \\
    \CC \tilde\CC[1](y) \ll 1, && (y\in \dD). \label{eq:bound-Cw-2}
  \end{align}
  Indeed, whenever~$y\in K$, we have that $k$ belongs to the finite set $\{h_{b,1}^c,b\in\Sigma\}$. In particular,~$|k'| \asymp 1$, and so
  $$ \tilde\CC[1](y) = \ssum{h_1 \in \cH_c,\\y\in\Dom(h_1)} \abs{h_1'(y)} \|\vphi_{h_1}\|_{\infty, \im(k) }^{\alpha} \ll 1 $$
  by~\eqref{eq:moment-condition-cuspidal} and positivity. This shows the first item. To prove the second item, we write
  $$ \CC \tilde\CC[1](y) = \ssum{h_1, h_2 \in \cH_c \\ y \in \Dom(h_2),\\h_2(y)\in\Dom(h_1)} \abs{h_1'(h_2(y)) h_2'(y)}  \|\vphi_{h_1}\|_{\infty,\im(h_2)}^{\alpha}  \ll 1 $$
  by~\eqref{eq:moment-condition-cuspidal} again.

  By~\eqref{stpacs} and~\eqref{eq:h'-x-0} and unfolding the decomposition~\eqref{eq:decomp-Hc-HK}, for any $y_a\in K_a$ we have
  \begin{align*}
      \sum_{h\in\cH,\atop K_a\subset \Dom(h)}  \| h'\|_\infty  \|\vphi_{h}\|_{\infty,K_a} ^{\alpha}&\ll\ssum{h\in \cH\\ y_a\in \Dom(h)}\!  \abs{h'(y_a)}r_h^\alpha\sum_{j=1}^{r_{h}} \|\vphi_{h_{j}}\|_{\infty,\im(h_{j+1})}^{\alpha}\\
      &\leq\sum_{r\geq0}(r+1)^\alpha \tilde\CC(\CC^\sharp)^{r}      [\1](y_a)\\
      &\quad+\sum_{r_1,r_2\geq0}(r_1+r_2+2)^{\alpha}\CC(\CC^\sharp)^{r_1}\tilde\CC^\sharp (\CC^\sharp)^{r_2}
      [\1](y_a)\\
      &\ll \Big( \tilde\CC+2^\alpha \CC\tilde\CC^\sharp+ \sum_{r_1\geq1}(r_1+2)^{\alpha}\CC(\CC^\sharp)^{r_1}\tilde\CC^\sharp \Big) [\1](y_a)\numberthis \label{eq:expansion-sumh-operators}
  \end{align*}
where we bounded the sums over $r$ and $r_2$ using the spectral radius inequality~$\spr(\CC^\sharp)<1$ implied by Lemma~\ref{lem:CCsharp-contracts}. We have~$\tilde\CC[\1](y) \ll 1$ for all~$y\in K$ by~\eqref{eq:bound-Cw-1}, and~$\CC^\sharp\tilde\CC^\sharp[\1](y) \leq \CC\tilde\CC^\sharp[\1](y) \leq \CC\tilde\CC[\1](y) \ll 1$ for all~$y\in \dD$ by~\eqref{eq:bound-Cw-2}, and so the above is uniformly bounded, and~\eqref{eq:moment-condition} follows by summing over $a\in\Sigma$.
  
  We now turn to~\eqref{eq:holder-condition}. With the same notation and the same argument as above, using that the maps $\vth_g$ are constant on $K_a$, we have
  \begin{align*} 
    \| \vphi_{h}  \|^\lambda_{(\kappa_0), K_a}     &\ll  r_{h}^\lambda\sum_{j=1}^{r_h}\| \vphi_{h_{j}}\circ h_{j+1}\circ \cdots \circ h_{r_h}\|^{\lambda}_{(\kappa_0),K_a} .
  \end{align*}
  By~\eqref{hld2} and~\eqref{eq:Hc-contraction} we have
  \begin{align*} 
    &\| \vphi_{h_{j}}\circ h_{j+1}\circ \cdots \circ h_{r_h}\|_{(\kappa_0), K_a}\\
    &\hspace{1em}\leq \|(h_{j+1}\circ \cdots \circ h_{r_h})'\|^{\kappa_0}_{\infty,K_a}\cdot
    \| \vphi_{h_{j}}\|_{(\kappa_0),\im(h_{j+1})}
    \ll  \| \vphi_{h_{j}}\|_{(\kappa_0),\im(h_{j+1})}
  \end{align*}
  Thus,
  \begin{align*}
    \sum_{h\in\cH,\atop K_a\subset \Dom(h)} \|h'\| \| \vphi_{h}  \|^\lambda_{(\kappa_0), K_a}  \ll \sum_{h\in\cH,\atop K_a\subset \Dom(h)} \|h'\|r_h^\lambda \sum_{j=1}^{r_h} \| \vphi_{h_{j}}\|_{(\kappa_0),\im(h_{{j+1}})}^{\lambda}
  \end{align*}
  and we conclude by the same arguments as in~\eqref{eq:expansion-sumh-operators}.
\end{proof}

For convenience we record the following variant of~\eqref{eq:moment-condition}, which comes at almost no cost.
\begin{lemma}\label{lem:moments-log}
  Let~$\alpha_1>0$, and assume that the bound~\eqref{eq:moment-condition} holds for all~$\alpha<\alpha_1$. Then for any~$m\geq 1$ and any~$\alpha<\alpha_1$, we have
  $$     \sum_{h\in\cH}\|h' \cdot (\log \abs{h'})^m \cdot |\vphi_{h}|^{\alpha}\|_{\infty}  < \infty. $$
\end{lemma}
\begin{proof}
  The sum~$\sum_{h\in \cH, y \in \Dom(h)} \abs{h'(y)}\log(|h'(y)|)^M$ is finite for all~$M\geq 1$ by Proposition~\ref{prop:Lst-bounded}. The claimed bound follows from Hölder's inequality along with~\eqref{eq:moment-condition}.
\end{proof}

\section{Spectral analysis of transfer operators} \label{sec:transfer:pert}

In this section we set up the functional spaces on which~$\LL_{s, \vt}$ will act and show uniform spectral gap and Dolgopyat-type estimate, generalizing the arguments in~\cite{BaladiVallee2005,BettinDrappeau2022,LeeSun2019}. These will be the main ingredients to establish Theorem~\ref{th:limitlaw:quasi}. 

From this section onwards, we shall assume Hypothesis~\ref{hyp:moments} holds. For certain results, we shall also require the stronger Hypothesis~\ref{hyp:moments-cuspidal}, which will be assumed explicitly. 

\subsection{Functional space}

Given $0< \kappa\leq \kappa_0$ and $(s,\vt)\in\C\times\R^d$ with $\Re(s)\geq\sigma_0$, we will at first consider $\LL_{s, \vt}$ as an operator acting on $H^\kappa$ (cf.~Remark~\ref{rmkex}). Notice that $\LL_{s, \vt}[f]$ is indeed in $H^\kappa$ whenever $f\in H^\kappa$ since 
the indicator function for the condition $y\in \Dom(h)$ in~\eqref{eq:def-twistedop-general} is locally constant on $K$.

Throughout the rest of the paper, we fix an arbitrary $\eps\in(0,\alpha/2)$, on which all subsequent implicit bounds might depend, even if not explicitly indicated. We also fix a value $\kappa>0$ satisfying
\begin{align}\label{dfnkappa}
  \kappa \leq\frac{\kappa_0 \min(\lambda,\alpha/2,1) (\alpha - \min( 1,\alpha-\eps))}{2(\alpha - \min(\lambda,\alpha/2,1))}.
\end{align}
Notice that this in particular implies 
\begin{align}\label{dfnkappa2}
  \kappa \leq \frac{\min(1,\lambda)}2 \kappa_0,
\end{align}
which will be used in later estimates (see e.g. the proof of Lemma~\ref{lem:perturbation-LL1t}),  and that~\eqref{dfnkappa} simplifies to $\kappa\leq\kappa_0/2$ if $\alpha\geq2$ and $\lambda\geq1$. 

\begin{remark}
  We take the opportunity to mention that the choice of $\kappa$ in~\cite[eq.~(5.3)]{BettinDrappeau2022} should be corrected similarly to~\eqref{dfnkappa}. This is needed to prove the analogue of~\eqref{sbsa} which is implicitly used in Section~7 therein.
\end{remark}

\begin{lemma} \label{cond:hennion}
  The space~$(H^\kappa, \norm{\cdot}_{\kappa})$ is a Banach space, and the norm~$\norm{\cdot}_{\kappa}$ is pre-compact with respect to~$\norm{\cdot}_\infty$.
\end{lemma}
\begin{proof}
  The first claim follows from the bijection~$H^\kappa \rightarrow \bigoplus_{a \in \Sigma} H^\kappa([a])$ by finiteness of~$\Sigma$, and the second claim is implied by Arzela--Ascoli's theorem. 
\end{proof}

\subsection{Dominant spectral properties}

The goal of this subsection is to prove the following Ruelle--Perron--Frobenius Theorem (see~\cite[Theorem~1.5]{Baladi2000}) for the family~$\LL_{\sigma, 0}$ with~$\sigma$ real. Recall that~$\sigma_0$ is defined in Lemma~\ref{lem:convergence-H-CF} and that~$\spr$ denotes the spectral radius.

\begin{proposition}\label{prop:spectrum-Ls0}
  For all~$\sigma \in \R$ with~$\sigma\geq\sigma_0$, the following hold.
  \begin{enumerate}[(i)]
    \item The operator~$\LL_{\sigma, 0}:H^\kappa \rightarrow H^\kappa$ is quasi-compact. It has a maximal eigenvalue~$\lambda_{\sigma,0}>0$ with a corresponding positive eigenfunction~$\xi_{\sigma,0}$.
    \item  The dominant eigenvalue~$\lambda_{\sigma,0}$ is simple and there is a centered disk of radius strictly less than~$\lambda_{\sigma,0}$ which contains all other eigenvalues.
    \item In particular, $\LL$ has leading eigenvalue~$1$ with eigenfunction~$\xi$ and its spectral decomposition is given by
    $$ \LL = \PP + \NN $$
    where~$\PP$ is a projection,~$\spr(\NN) < 1$, and~$\PP \NN = \NN\PP = 0$.
  \end{enumerate}
\end{proposition}

This shows alternatively the existence of an invariant measure. Note that~$\LL  \xi =\xi$, together with
Proposition~\ref{prop:spectrum-Ls0}, yields the existence of a unique finite measure~$\nu$ on $K_\Theta$ such that~$\LL^* \nu=\nu$ and~$\int \xi d\nu=1$ by means of the dual operator. By~\eqref{eq:transferop-invar} we have
\begin{align}\label{fdn}
 \df \nu(x, v) = {\fm}^{-1}\df s(x) \df v,
 \end{align}
where 
\begin{equation}\label{dfnm}
\fm:=\int_{K} \xi \df s(x)
\end{equation}
is the total mass of $\xi$ on $K$. The projector in the spectral decomposition has the explicit formula:
\begin{equation}
  \PP f = \Big(\int_{K_\Theta} f \df\nu \Big) \xi .\label{eq:def-eigenproj}
\end{equation}

Proposition~\ref{prop:spectrum-Ls0} is proven by following the arguments in~\cite[Theorem 1.5]{Baladi2000}, and~\cite[Theorem~5.9]{LeeSun2019} for skew-product extensions.  We highlight the two most important ingredients in the proof.

\begin{lemma} \label{lem:quasicomp}
  Let~$\rho$ be as in Lemma~\ref{lem:contraction}. There exists a function $(\sigma_0,\infty)\ni \sigma\mapsto A_\sigma\in\R_{+}$ which is continuous, piecewise smooth, decreasing, and with $A_1=1$ such that for all~$r \geq 1$ and~$\sigma_0 \leq \sigma \ll 1$, 
  \begin{align}
    \|\LL_{\sigma,0}^r f\|_{\infty} &\ll A_\sigma^r  \|f\|_{\infty}, \label{asx1}\\
    \|\LL_{\sigma, 0}^r f\|_{(\kappa)} &\ll A_\sigma^r  (\|f\|_\infty+\rho^{\kappa r}\|f\|_{(\kappa)}),\label{asx2}
  \end{align}
  where the implicit constant depends at most the choice of the fundamental domain~$R$ and the boundedness of~$\sigma$.
 Moreover, for any~$\sigma_1>\sigma_0$ and~$\eps>0$, there is a choice of map~$\sigma\mapsto A_\sigma$ such that~$A_{\sigma_1} < \spr(\LL_{\sigma_1,0}) + \eps$.
\end{lemma}
\begin{proof}
  By Lemma~\ref{lem:contraction}, there exists $m\geq1$ such that $\|h'(x)\|_\infty<1$ for all $h\in \cH^m$. For $\sigma\geq\sigma_0$, let $A_\sigma :=  \norm{\LL^m_{\sigma, 0}[\xi]/\xi}_\infty^{1/m}$. In particular, $A_1=1$ by~\eqref{eq:invar-LL-xi}, and for any given~$\sigma_1>\sigma_0$ and~$\eps>0$, by choosing~$m$ large we can ensure~$A_{\sigma_1} < \spr(\LL_{\sigma_1,0})+\eps$. It is also clearly continuous, piecewise smooth and decreasing.
  Moreover, for~$f\in H^\kappa$, we have
  \begin{align*}
    \|\LL^m_{\sigma,0} [f]/\xi \|_\infty = \big\|\LL^m_{\sigma, 0} \big[\tfrac{f}{\xi} \xi\big]/\xi\big\|_\infty \leq \norm{f/\xi}_\infty\norm{\LL^m_{\sigma, 0}[\xi]/\xi}_\infty = A^m_\sigma \norm{f/\xi}_\infty,
  \end{align*}
  by positivity. By induction, $\|\LL^{rm}_{\sigma,0} [f]/\xi \|_\infty\leq A^{rm}_\sigma \norm{f/\xi}_\infty$ for $r
  \geq1$ and~\eqref{asx1} follows since~$\xi \asymp 1$ and $\|\LL_{\sigma,0}\|_\infty\ll 1$ for $\sigma_0<\sigma\ll1$.
  Regarding the second estimate, for all~$r\geq 1$ we have
  \begin{align*}
    \|  \LL_{\sigma, 0}^r f\|_{(\kappa)} \leq {}&  \ssum{h\in \cH^r} \left\|  | h'|^\sigma \cdot f\circ h\right\|_{(\kappa)}\\
    \ll{}&  \ssum{h\in \cH^r} \big(\|  | h'|^\sigma\|_{(\kappa)}\|f\|_{\infty}+ \|h'\|_{\infty}^{\sigma+\kappa} \|f\|_{(\kappa)}\big) 
  \end{align*}
  by~\eqref{hld1}-\eqref{hld2}. By bounded distortion~\eqref{eq:bounded-dist} we have~$\||h'|^{\sigma}\|_{(\kappa)} \ll |\sigma|^\kappa \|h'\|_{\infty}^{\sigma-\kappa} \|h''\|_\infty^{\kappa} \ll |\sigma|^\kappa  \|h'\|_\infty^\sigma$. Moreover, by~\eqref{eq:contraction} we have $\|h'\|_{\infty}^{\kappa}\ll \rho^{\kappa r}$.
  Thus~\eqref{asx2} follows by~\eqref{eq:h'-x-0} and~\eqref{asx1}.
\end{proof}

\begin{lemma}\label{lem:cF-topmix}
  The map~$\cF_\Theta$ is topologically mixing.
\end{lemma}
\begin{proof}
  Let~$r_0$ be given by Hypothesis~\eqref{eq:vth-mixing}.
  Let~$\mathcal V\subset K$ be a non-empty open arc and~$\vth_1 \in \Theta$. We wish to show that for some~$m \geq 1$ depending on~$\mathcal V$, and all~$n\geq m$, we have~$\cF_\Theta^n(\mathcal V\times\{\vth_1\}) = K_\Theta$. Since the diameter of a fundamental arc~$[w_0]$, $w_0 = a_1\dotsb a_n \in \Sigma^*$, tends uniformly to~$0$ as~$n\to \infty$, we deduce that~$\mathcal V$ contains such an interval with~$n\geq 1$.
  
  Let~$x\in K$ and $\vth_2\in\Theta$. We apply the hypothesis~\eqref{eq:vth-mixing} with~$b_1 = a_n$, $b_2\in \Sigma$ the letter such that~$x\in [b_2]$, and
  $$ \vth_0 = \vth_2^{-1} \vth_1 \vth_{a_1 \dotsb a_{n-1}}(a_n)^{-1}. $$
  We deduce the existence of~$w$ with~$a_n w b_2 \in \Sigma^*$ such that
  $$ \vth_{a_n w}(b_2) = \vth_2^{-1} \vth_1 \vth_{a_1 \dotsb a_{n-1}}(a_n)^{-1} \quad \text{and} \quad \ell_K(a_n w b_2) = r_0. $$
  Let~$h = g_{w_0 w}$. By the definition~\eqref{eq:def-ellK} of $\ell_K$ we have
  $$ m := \ell_K(a_2 \dotsb a_n w b_2) = \ell_K(w_0) + r_0 - 2, $$
  which depends neither on~$x$ nor on~$\vth_2$. Then, we have
  $$ \cF_\Theta^m( h(x), \vth_1) = (x, \vth_1 \vth_{w_0 w}(b_2)). $$
  However by the cocycle property we have
  $$ \vth_1 \vth_{w_0 w}(b_2) = \vth_1 \vth_{a_n w}(b_2) \vth_{a_1 \dotsb a_{n-1}}(a_n)  = \vth_2 $$
  by construction, and so
  $$ \cF_\Theta^m( h(x), \vth_1) = (x, \vth_2). $$
  Since we have~$h (x) \subset [w_0] \subset \mathcal V$, $(x, \vth_2)$ is arbitrary and~$m$ is independent of~$(x, \vth_2)$, we deduce
  $$ \cF_\Theta^m(\mathcal V \times \{\vth_1\}) = K_\Theta $$
  as required.
\end{proof}

The proof of Proposition~\ref{prop:spectrum-Ls0} follows now by standard arguments, see~\cite[Theorem~1.5]{Baladi2000}.
Lemmas~\ref{cond:hennion} and~\ref{lem:quasicomp} (Lasota-Yorke inequalities) show that Hennion's criterion~\cite{Hennion1993} is statisfied. This implies quasi-compactness of~$\LL_{\sigma,0}$ acting on~$H^\kappa$. The positivity of some eigenvalue~$\lambda_{\sigma,0}$ of largest modulus and its eigenfunction~$\xi_{\sigma,0}$ follows from the fact that~$\LL_{\sigma,0}$ is weighted by the positive real-valued function~$g := |h'|^\sigma$. 

For the second item, we use Lemma~\ref{lem:cF-topmix} (topological mixing) to conclude that~$\lambda_{\sigma,0}$ is a unique and simple eigenvalue of largest modulus by density and convergence arguments. In particular, by change of variable, the operator~$\LL$ admits~$1$ as dominant eigenvalue with the associated eigenfunction~$\xi$. This spectral gap yields a spectral decomposition
\[ \LL_{\sigma,0}=\lambda_{\sigma,0} \PP_{\sigma,0}+\NN_{\sigma,0} \]
where the eigenprojection~$\PP_{\sigma,0}$ has an image spanned by~$\xi_{\sigma,0}$ and~$\spr(\NN_{\sigma,0})<\lambda_{\sigma,0}$ by the quasi-compactness. The spectral projection~$\PP_{\sigma, 0}$ has an integral form similar to~\eqref{eq:def-eigenproj}.

\subsection{Perturbation estimates}

\subsubsection{Estimates relative to~$\cF_K$}
We recall the definition~\eqref{eq:def-PhiK} and its equivalent counterpart~\eqref{eq:PhiK-vphi}.
For all~$m\geq0$, the operator on~$L^\infty(K_\Theta)$
\begin{align*}
  f\mapsto  \LL[ (\log|\cF_K'|)^m f](y, \vth_1) &= \ssum{h\in \cH\\ y\in \Dom(h)} (-\log |h'(y)|)^m\abs{h'(y)} f(h(y), \vth_1 \vth_h(y))
\end{align*}
is well-defined by Lemma~\ref{lem:moments-log}, as well as bounded on~$H^\kappa$ by~\eqref{eq:bounded-dist} and Proposition~\ref{prop:Lst-bounded}. 
In the same way, the operators
\begin{align*}
  f\mapsto  \LL[ \Phi_K(\log|\cF_K'|)^m  f](y, \vth_1) &= \ssum{h\in \cH\\ y\in \Dom(h)}  \vth_1\vphi_h(y)(-\log |h'(y)|)^m\abs{h'(y)} f(h(y), \vth_1 \vth_h(y)),\\
  f\mapsto  \LL[ \Phi_K\cdot  \Phi_K^T f](y, \vth_1) &= \ssum{h\in \cH\\ y\in \Dom(h)} (\vth_1\vphi_h(y))(\vth_1\vphi_h(y))^T \abs{h'(y)} f(h(y), \vth_1 \vth_h(y))
\end{align*}
are well-defined provided that Hypothesis~\ref{hyp:moments} holds with~$\alpha>1$ and~$\alpha>2$ respectively. Here~$\norm{\cdot}_2$ is the Euclidean norm.

Finally, we state some perturbation estimates and expansions.

\begin{lemma}\label{lem:perturbation-LL1t}
  Assume Hypothesis~\ref{hyp:moments}. For all~$s$ and~$\vt$ with~$\abs{s-1}$ and~$\nvt$ sufficiently small, we have
  \begin{align}
    \|\LL_{s, \vt}-\LL\|_\infty &\ll  \abs{s-1} + \nvt+\nvt^{\alpha-\eps}\label{fbsa},\\
    \|\LL_{s, \vt}-\LL\|_{\kappa} &\ll \abs{s-1} + \nvt+\nvt^{\alpha-\eps}.\label{sbsa}
  \end{align}
\end{lemma}
\begin{proof}
  Let $\vt$ and $s-1$ be sufficiently small. We first prove~\eqref{fbsa}. We have $e^{y}=1+O(|y|^{\alpha-\eps}+|y|)$ for~$y\in i\R$, so that from the definition~\eqref{eq:def-twistedop-general}, Proposition~\ref{prop:Lst-bounded} and Lemma~\ref{lem:moments-log} we obtain~\eqref{fbsa} by H\"older's inequality.

  We now move to~\eqref{sbsa}.
  Let $f\in H^\kappa$. By~\eqref{fbsa} it suffices to bound $ \|\LL_{s, \vt}[f]-\LL[f]\|_{(\kappa)}$. By~\eqref{hld1}-\eqref{hld2} and splitting the differences, we have
  \begin{align*}
    &\|   ( \abs{h'(y)}^s e^{i \langle \vt, \vth_1\vphi_h \rangle}-    \abs{h'(y)}) f(h(y), \vth_1 \vth_h(y))\|_{(\kappa)}
    \\
    &\hspace{2.5em}\ll \|   ( \abs{h'(y)}^se^{i \langle \vt, \vth_1\vphi_h \rangle}-    \abs{h'(y)})\|_{\infty}\| f\|_{(\kappa)}+ \|    \abs{h'(y)}^s\|_{(\kappa)}\| (\e^{i \langle \vt, \vth_1\vphi_h \rangle}-1)\|_{\infty}\|f\|_\infty\\
    &\hspace{2.5em}\quad+
    \| \abs{h'(y)}^s-    \abs{h'(y)}\|_{(\kappa)}\|f\|_\infty+\|\abs{h'(y)}^\sigma\|_{\infty}\| (e^{i \langle \vt, \vth_1\vphi_h \rangle}-1)\|_{(\kappa)}\|f\|_\infty.
  \end{align*}
  The first three summands on the right can be dealt as for~\eqref{fbsa}, where for the third one we use that $\||h'|^{s}-|h'| \|_{(\kappa)}\ll |s-1|(\|h'\|_{\infty}+\|h'\|_{\infty}^{\sigma})|\log|h'||$ which follows, as in the proof of Lemma~\ref{lem:quasicomp}, from $|h'|^{s}-|h'| =\int_1^s|h'|^w\log|h'|\,dw$. It remains to deal with the last summand. 
  We let 
  \begin{align*}
    \kappa_1 &= \kappa/\min( 1,\alpha-\eps).
  \end{align*}
  Notice that clearly $\kappa_1\geq\kappa$ and, by~\eqref{dfnkappa}, we also have $\kappa_1\leq \kappa_0/2\leq1/2$. By~\eqref{hld3} and~\eqref{hld4} we then have
  \begin{equation*}\begin{split} 
      \| e^{i \langle \vt, \vth_1\vphi_h \rangle}-1 \|_{(\kappa)}=\| e^{i \langle \vt, \vth_1\vphi_h \rangle} \|_{(\kappa)}& \ll  \nvt^{\kappa/\kappa_1} \|\vphi_h\|_{(\kappa_1)}^{\kappa/\kappa_1}\\
      &\ll \nvt^{\kappa/\kappa_1}
      \|\vphi_h\|_\infty^{\kappa/\kappa_1-\kappa/\kappa_0}  \|\vphi_h\|_{(\kappa_0)}^{\kappa/\kappa_0}
      \\
    \end{split}  
  \end{equation*}
  Next, we let
  \begin{align*}
    (\kappa/\kappa_1-\kappa/\kappa_0)p= \alpha,\qquad \kappa/\kappa_0 q=\min(\lambda,\alpha/2,1),
  \end{align*}
  so that by~\eqref{dfnkappa}, we have
  $$ 1 - \frac1p - \frac1q \geq \frac{\alpha - \min(1, \alpha - \eps)}{2\alpha} > 0, $$
  and let~$r\in (1, \infty)$ be such that~$\frac1p + \frac1q + \frac1r=1$. Then, by H\"older inequality we have
  \begin{align*}
    &\ssum{h\in \cH} \|\abs{h'(y)}^\sigma\|_{\infty}\| (e^{i \langle \vt, \vth_1\vphi_h \rangle}-1)\|_{(\kappa)}\|f\|_\infty\\
    &\hspace{5em}\ll \nvt^{\min(1,\alpha-\eps)} \ssum{h\in \cH} 
    \|h'(y)\|_{\infty}^\sigma  \|\vphi_h\|_\infty^{\kappa/\kappa_1-\kappa/\kappa_0}  \|\vphi_h\|_{(\kappa_0)}^{\kappa/\kappa_0}\\
    &\hspace{5em}\leq \nvt^{\min(1,\alpha-\eps)} \bigg(\ssum{h\in \cH} 
    \|h'(y)\|_{\infty}^{1+r(\sigma-1)}\bigg)^{1/r}\bigg(\ssum{h\in \cH}\|h'(y)\|_{\infty} \|\vphi_h\|_\infty^{\alpha} \bigg)^{1/p}\\
    &\hspace{5em}\quad\times\bigg( \ssum{h\in \cH}\|h'(y)\|_{\infty} \|\vphi_h\|_{(\kappa_0)}^{\min(\lambda,\alpha/2,1)}\bigg)^{1/q}\\
    &\hspace{5em}\ll \nvt^{\min(1,\alpha-\eps)},
  \end{align*}
  by Proposition~\ref{prop:Lst-bounded} and~\eqref{eq:moment-condition}-\eqref{eq:holder-condition}, provided that $\sigma-1$ is small enough in terms of~$\eps$ and~$\alpha$.
\end{proof}

If~$\alpha>2$ we can refine~\eqref{fbsa} further. 

\begin{lemma}\label{lem:expansion-Lst-clt}
  Assume~$\alpha>2$. There exists~$t_0>0$ such that for~$|s-1|\ll\nvt<t_0$, we have
  \begin{align*}
    \LL_{s, \vt}[f] &= \LL[f]+i\inpd{\vt ,  \LL[ \Phi_K f]}-(s-1)\LL[\log |\cF_K'| f]+{}\\
    &\quad-\tfrac12\vt^T\LL[ \Phi_K\cdot  \Phi_K^T f]\vt-i(s-1)\inpd{\vt , \LL[ \Phi_K \log |\cF_K'| f]}\\
    &\quad+\tfrac12(s-1)^2\LL[\log |\cF_K'|^2 f] +O(\nvt^{3}+\nvt^{\alpha-\eps})
  \end{align*}
  in~$L^\infty(K_\Theta)$.
\end{lemma}
\begin{proof}
  Write~$\LL_{s, t}[f] = \LL[\e^{i\inpd{\vt, \Phi_K}} \abs{h'}^{s-1} f]$. First we expand~$\abs{h'}^{s-1}$ to third order, which, by Lemma~\ref{prop:Lst-bounded},  gives an error term~$O(|s-1|)^3 = O(\nvt^3)$ for~$t_0$ small enough. Then we expand the exponential by~$e^{it}=1+it-t^2/2+O(|t|^{\beta})$ for real~$t$, with $\beta=\min(3,{\alpha-\eps})$ if $\alpha-\eps\geq 2$ and $\beta=2$ otherwise. The ensuing error terms are bounded by
  $$ \sum_{h\in\cH} \abs{h'}^\sigma(1+\abs{\log\abs{h'}}^2)(\nvt \norm{\vphi_h}_{\infty})^{\beta}. $$
  where~$\sigma = \Re(s)$. By the Hölder inequality, and applying Lemmas~\ref{prop:Lst-bounded} and~\ref{lem:moments-log}, we have that this is~$O(\nvt^\beta)$ provided that~$t_0$ is small enough.
\end{proof}
\begin{lemma}\label{lem:expansion-Lst-clt2}
  Assume~$\alpha>2$. There exists~$t_0>0$ such that for~$|s-1|\ll\nvt<t_0$ and $f\in H^\kappa$ we have
\begin{align*}
\|  \LL_{s,\vt}[f]-\LL[f]-i\inpd{\vt,\LL[\Phi_K f]}
       +(s-1)\LL[\log|\cF_K'|f]\|_{\kappa}
\ll \nvt^{\min(2,\alpha-1-\eps)}\|f\|_\kappa.
\end{align*}
\end{lemma}
\begin{proof}
First, we note that it suffices to prove the bound for $\|\|_{(\kappa)}$ since the required bound for the infinity norm follows immediately from Lemma~\ref{lem:expansion-Lst-clt}.

Let $q(t)=e^{it}-1-it$ and $\beta=\min(2,\alpha-1-\eps)$ if $\alpha-\eps\geq2$ and $\beta=1$ otherwise. Since $q(t)\ll |t|^\beta$ and $q'(t)\ll |t|^{\beta-1}$, by~\eqref{hld4} we have
\begin{align*}
\|q(\inpd{\vt,\vth_1\vphi_h(y)})\|_{(\kappa)}&\ll \|q(\inpd{\vt,\vth_1\vphi_h(y)})\|_{\infty}^{1-\kappa/\kappa_0}\|q(\inpd{\vt,\vth_1\vphi_h(y)})\|_{(\kappa_0)}^{\kappa/\kappa_0}\\
&\ll\nvt^{\beta}\cdot\|\vphi_h\|_\infty^{\beta-\kappa/\kappa_0}\|\vphi_h\|_{(\kappa_0)}^{\kappa/\kappa_0}.
\end{align*}
By~\eqref{dfnkappa} we have
\begin{align*}
\frac{\beta-\kappa/\kappa_0}{\alpha}+\frac{\kappa/\kappa_0}{\lambda}\leq \frac{\beta}{\alpha}-\frac{\kappa}{\kappa_0}\frac{\min(1,\lambda)-\alpha}{\alpha\min(1,\lambda)}\leq \frac{\beta}{\alpha}+\frac{\alpha-1}{2\alpha}\leq1-\frac{\eps}{2\alpha}<1,
\end{align*}
since $2\beta\leq\alpha+1-\eps$ for $\eps<\alpha/2$. Thus,
we can apply H\"older's inequality as in Lemma~\ref{lem:perturbation-LL1t} and obtain
\begin{align*}
\|  \LL_{s,\vt}[f]-\LL_{s,0}[f]-i\inpd{\vt,\LL_{s,0}[\Phi_K f]}\|_{(\kappa)}
\ll \nvt^\beta\|f\|_\kappa
\end{align*}
for $|s-1|$ small enough. Expanding $|h'|^s$, with the additional powers of $\log|h'|$ treated as in
Lemma~\ref{lem:moments-log}, we then obtain the claimed bound.  
\end{proof}

\subsubsection{Estimates relative to~$\cF_c$}

In this subsection, we assume the hypotheses~\ref{hyp:moments-cuspidal}. Our goal is to provide technical estimates needed in order to rephrase the parameters of the limit law in Theorem~\ref{th:limitlaw:quasi} in terms of~$\cF_c$ and its inverse branches, whenever possible. We recall the definition~\eqref{eq:def-Phic} of~$\Phi_c$ and its relation~\eqref{eq:ref-PhiK-Phic} with $\Phi_K$.

The following lemma relates~$\LL_{1, \vt}$ to~$\CC_{1, \vt}$, using the decomposition given in Lemma~\ref{lem:decomp-Hc-HK}.
\begin{lemma}\label{lem:expansion-L1t-subclt}
  Assume Hypothesis~\ref{hyp:moments-cuspidal}.
  Then, for $f\in L^\infty(K_\Theta)$ we have
  \begin{align*}
    &\int_{K_\Theta} (\LL_{1, \vt}-\LL)[f] \df\nu=\int_{K_\Theta} \big( (\CC_{1, \vt}-\CC) + \CC (\Id-\CC^\sharp)^{-1} (\CC_{1, \vt}^\sharp-\CC^\sharp) \big) (\Id-\CC^\sharp)^{-1}[ f] \df\nu\\
    &\hspace{3em}+\int_{K_\Theta} \big( (\CC_{1, \vt}-\CC) + \CC (\Id-\CC^\sharp)^{-1} (\CC_{1, \vt}^\sharp-\CC^\sharp) \big)(\CC_{1, \vt}^\sharp-\CC^\sharp) (\Id-\CC^\sharp)^{-1}[ f] \df\nu\\
    &\hspace{3em}+O\big( \nvt^2+\nvt^{2\alpha-2\eps}\big)
  \end{align*}
\end{lemma}

\begin{proof}
  First, we observe that proceeding as in~\eqref{fbsa} one has that~\eqref{eq:moment-condition-cuspidal}  implies
  \begin{align} 
    &\|\CC_{1,\vt}(\CC^\sharp_{1, \vt}-\CC^\sharp)\|_{\infty}+    \|\CC(\CC^\sharp_{1, \vt}-\CC^\sharp)\|_{\infty}+ \|\CC^\sharp(\CC^\sharp_{1, \vt}-\CC^\sharp)\|_{\infty} \ll \nvt + \nvt^{\alpha-\eps},\label{perturbation-CC1ta}
    \\
    &  \int_{\dD_\Theta}|(\CC_{1, \vt}-\CC)[f]|\df\nu\ll  \nvt + \nvt^{\alpha-\eps} \qquad \forall f\in L^\infty(\dD_\Theta).\label{perturbation-CC1tb}
  \end{align}
  Recalling~\eqref{eq:ref-LLst-CCst}, a quick computation shows that on bounded functions supported on $K_\Theta$ we have
  \begin{align*}
    \LL_{1, \vt}-  \LL&=\CC_{1,\vt}(\Id-\CC_{1, \vt}^\sharp)^{-1}-\CC(\Id-\CC^\sharp)^{-1}\\
    &= \big((\CC_{1, \vt}-\CC)+\CC(\Id-\CC^\sharp)^{-1}(\CC_{1, \vt}^\sharp-\CC^\sharp)\big)(\Id-\CC^\sharp)^{-1}\\
    &\quad +\big((\CC_{1, \vt}-\CC)+\CC(\Id-\CC^\sharp)^{-1}(\CC_{1, \vt}^\sharp-\CC^\sharp)\big)
    (\CC_{1, \vt}^\sharp-\CC^\sharp)(\Id-\CC^\sharp)^{-1}\\
    &\quad +\big((\CC_{1, \vt}-\CC)+\CC(\Id-\CC^\sharp)^{-1}(\CC_{1, \vt}^\sharp-\CC^\sharp)\big)\CC^\sharp_{1,\vt}(\Id-\CC_{1,\vt}^\sharp)^{-1}(\CC_{1, \vt}^\sharp-\CC^\sharp)(\Id-\CC^\sharp)^{-1}.
    &\quad\end{align*}
  We apply this to $f$, and integrate over $K_\Theta$ and the result follows by estimating the contribution of the last line by~\eqref{perturbation-CC1ta}-\eqref{perturbation-CC1tb}.
\end{proof}

\subsection{Dolgopyat-type estimate for the twists}

The goal of this section is to show the existence of a strip~$\{s\in \C, \Re(s)>\sigma\}$ for some~$\sigma<1$, in which~$\LL_{s, \vt}$ is contracting in a quantitative sense. 
To this end we use a method originating from work of Dolgopyat~\cite{Dolgopyat1998}. This method relies on a property of the map~$\cF_K$ known as uniform non-integrability (UNI).
Dolgopyat's paper contains two arguments. One method (Section~5 \emph{ibid}) relies on a stronger hypothesis, and is simpler to implement. The other method (Sections~(6--8) \emph{ibid.}) relies on a weaker condition, but is much more involved. The conditions were called respectively strong UNI and weak UNI in~\cite{BaladiVallee2005,BaladiVallee2005a}.

The generalization of Dolgopyat's argument to a case with countably many branches was obtained by Baladi--Vallée~\cite{BaladiVallee2005a}, under the weak UNI condition. The specific case of the Gauss map was included in the earlier work of Baladi-Vallée~\cite{BaladiVallee2005}, and there the strong UNI condition was used.

In the following, we show that the strong UNI condition also holds in our case (this implies automatically the weak UNI condition), using the explicit construction of the natural extension from Section~\ref{sec:natur-extens-invar}.
We then justify that the strong UNI argument due to Dolgopyat~\cite{Dolgopyat1998} and Baladi--Vallée~\cite{BaladiVallee2005} hold in our situation.

\subsubsection{The strong UNI condition}

Our aim in this section is to show the ``strong UNI'' property of~$\cF_K$ acting on the circle; see~\cite[p.~357]{BaladiVallee2005} for the case of the Gauss map.
Given~$n \geq 1$ and~$h, k \in \cH^n$, we denote
\begin{equation} \label{def:uni:dist}
  \Delta(h, k) := \inf_{x\in \Dom(h)\cap\Dom(k)} \abs{\frac{h''}{h'}(x) - \frac{k''}{k'}(x)}. 
\end{equation}
Since for a disk automorphism $g$ we have 
$$e^{i\theta}\frac{g''(e^{i\theta})}{g'(e^{i\theta})}=\big(|g'(e^{i\theta})|-1\big)-i\frac{d}{d\theta}\log |g'(\e^{i\theta})|,\qquad \theta\in\R,$$
by~\eqref{eq:contraction} we have that $\Delta$ can be expressed in terms of the derivative of $\log\frac{|h'(z)|}{|k'(z)|}$ as
\begin{align}\label{aldelt}
\Delta(h,k)=\inf_{e^{i\theta}\in \Dom(h)\cap\Dom(k)} \Big|\tfrac{d}{d\theta}\log\tfrac{|h'(e^{i\theta})|}{|k'(e^{i\theta})|}\Big|+O(\rho^n).
\end{align} 

\begin{proposition} \label{uni:strong} 
  There exists a constant $\rho<1$ such that for all~$\eta>0$, $n\geq 2$ and $k \in \cH^n$, we have
  $$ \ssum{h\in \cH^n \\ \im(h) \subset K \\ \Delta(h, k)\leq \eta} \|h'\|_\infty \ll \eta + \rho^n. $$
\end{proposition}

The arguments we propose is an extension of those in~\cite[Section~3.4]{BaladiVallee2005}. The question is translated entirely in terms of a dual dynamical system, which is also Markov and hyperbolic, whose inverse branches are~$\{h^{-1}, h\in \cH^*\}$: then, by means of~\eqref{eq:comparison-distortion}, this allows to translate the condition~$\Delta(h, k)\leq \eta$ into a condition of type~$h(x)\in I$ for some interval~$I$ of length~$O(\eta)$. This is then implemented as an input test function for the dual transfer operator. That such a system exists is related to the existence of the natural extension going backwards in the coding, see the comment in~\cite[p.~363]{BaladiVallee2005}. This is the main reason we needed to explicit the natural extensions for~$\cF_K$ in Section~\ref{sec:natur-extens-invar}.

\begin{proof}
  We let $\rho\in(\spr(\NN), 1)$, which is allowed by Proposition~\ref{prop:spectrum-Ls0}, and suppose~$\rho$ is close enough to~$1$ that it satisfies also Lemma~\ref{lem:contraction}.

  For all~$a\in\Sigma$, we let~$x_{a}$ be an arbitrary point of~$K\cap(\dD\setminus[\hat a])$. By~\eqref{eq:h'-x-0}, for all~$h\in\cH^*$ we have
  $$ \|h'\|_{\infty} \asymp \ssum{(b_1, b_2)\text{ non-cusp.}\\ [b_1,b_2]\subset \Dom(h)} \|h'\|_{\infty, [b_1, b_2]} \asymp \ssum{(b_1, b_2)\text{ non-cusp.}\\ [b_1,b_2]\subset \Dom(h)} \abs{h'(g_{b_1}g_{b_2} x_{b_2})}. $$
  Next we note that
  $$ \abs{h'(g_{b_1b_2} x_{b_2})} = \frac{\abs{(h\circ g_{b_1b_2})'(x_{b_2})}}{\abs{g_{b_1b_2}'(x_{b_2})}} \asymp \abs{(h\circ g_{b_1b_2})'(x_{b_2})}. $$
  For all~$(a_2, b_1)\in \Sigma$, we define
  \begin{align*}
    W:= {}&  \{(a_1, a_2, w, b_1, b_2)\in\Sigma^2\times\Sigma^*\times\Sigma^2\mid (a_1, a_2) \text{ and } (b_1, b_2) \text{ non-cuspidal}, \ell_K(a_2 w b_1)=n\} \\
    W':={}& \{(h, b_1, b_2)\in\cH^n\times\Sigma^2\mid \im(h)\subset K, (b_1, b_2)\text{ non-cuspidal}, [b_1, b_2]\subset \Dom(h)\}.
  \end{align*}
  Then the map
  \begin{align}
    W\ni (a_1, a_2, w, b_1, b_2) \mapsto (h_{a_1a_2w,n}, b_1, b_2)\in W'
    \label{eq:bij-Wstar-Hn}
  \end{align}
  is a bijection  and so, from the above discussion, we have
  $$ \ssum{h\in \cH^n  \\ \Delta(h, k)\leq \eta} \|h'\|_\infty \asymp \ssum{(a_1, a_2) \\ (b_1, b_2) \\ \text{non-cusp.}} \ssum{w\\ \ell_K(a_2 w b_1)=n \\ \Delta(g_{a_1a_2w}, k)\leq \eta} \abs{g_{a_1a_2 w b_1b_2}'(x_{b_2})}. $$
  We first transform the condition involving~$\Delta(h, k)$. By~\eqref{eq:comparison-distortion}, we have
  $$ \Delta(h, k) \gg \abs{g_{a_1a_2w}^{-1}(0) - k^{-1}(0)}. $$
  It follows that there exists a set~$I\subset\C$ of diameter~$O(\eta)$, depending only on~$k$, such that
  $$ \Delta(h, k)\leq \eta \implies g_{a_1a_2w}^{-1}(0) \in I. $$
  Define
  $$ J_1 := \bigcup_{\substack{(b_1, b_2) \\ \text{ non-cusp.}}} g_{b_1b_2}^{-1}(I), $$
  so that $J_1$ a finite union of sets of diameter $O(\eta)$. For all non-cuspidal pairs $(b_1,b_2)$ we then have
  $$ \Delta(h, k)\leq \eta \implies g_{a_1a_2wb_1b_2}^{-1}(0) \in J_1. $$
  Finally, using~\eqref{eq:h'-x-0} and the fact that $|g'(0)|=|(g^{-1})'(0)|$ for $g\in \SU(1, 1)$, we have
  $$ \abs{g_{a_1a_2 w b_1b_2}'(x_{b_2})} \asymp \abs{g_{a_1a_2 w b_1b_2}'(0)} = \abs{(g_{a_1a_2 w b_1b_2}^{-1})'(0)}. $$
  From the above, we conclude that
  $$ \ssum{h\in \cH^n \\ \Delta(h, k)\leq \eta} \|h'\|_\infty \ll \ssum{(a_1, a_2) \\ (b_1, b_2) \\ \text{non-cusp.}} \ssum{w\\ \ell_K(a_2 w b_1)=n\\ g_{a_1a_2wb_1b_2}^{-1}(0) \in J_1} \abs{(g_{a_1a_2 w b_1b_2}^{-1})'(0)}. $$
  We may now use the fact that the set~$W$ is invariant under the involution
  $$ (a_1, a_2, w, b_1, b_2) \mapsto (\hat b_2, \hat b_1, \hat w, \hat a_2, \hat a_1), $$
  to the effect that
  $$ \ssum{(a_1, a_2) \\ (b_1, b_2) \\ \text{non-cusp.}} \ssum{w\\ \ell_K(a_2 w b_1)=n \\ g_{a_1a_2wb_1b_2}^{-1}(0) \in J_1} \abs{(g_{a_1a_2 w b_1b_2}^{-1})'(0)}
  = \ssum{(a_1, a_2) \\ (b_1, b_2) \\ \text{non-cusp.}} \ssum{w\\ \ell_K(a_2 w b_1)=n \\ g_{a_1a_2wb_1b_2}(0) \in J_1} \abs{g_{a_1a_2 w b_1b_2}'(0)}. $$
  We now use again~\eqref{eq:h'-x-0} and~\eqref{eq:diffh-0-h'}, obtaining
  \begin{align*}
    \ssum{(a_1, a_2) \\ (b_1, b_2) \\ \text{non-cusp.}} \ssum{w\\ \ell_K(a_2 w b_1)=n \\ g_{a_1a_2wb_1b_2}(0) \in J_1} \abs{g_{a_1a_2 w b_1b_2}'(0)} 
    \ll{} \ssum{(a_1, a_2) \\ (b_1, b_2) \\ \text{non-cusp.}} \ssum{w\\ \ell_K(a_2 w b_1)=n \\ g_{a_1a_2wb_1b_2}(x_{b_2}) \in J_2} \abs{g_{a_1a_2 w b_1b_2}'(x_{b_2})}
  \end{align*}
  for some set~$J_2$ which, by~\eqref{eq:contraction}, is a finite union of sets of diameter $O(\eta+\rho^n)$.
  In the last expression, we have
  $$ \abs{g_{a_1a_2 w b_1b_2}'(x_{b_2})} = \abs{g_{a_1a_2 w}'(g_{b_1b_2}x_{b_2}) g_{b_1b_2}'(x_{b_2})} \ll \abs{g_{a_1a_2 w}'(g_{b_1b_2}x_{b_2})}, $$
  and therefore, letting~$y_{b_1b_2} := g_{b_1b_2} x_{b_2}$ and using the bijection~\eqref{eq:bij-Wstar-Hn},
  \begin{align*}
    \ssum{h\in \cH^n \\ \Delta(h, k)\leq \eta} \|h'\|_\infty     \ll{}& \ssum{(a_1, a_2) \\ (b_1, b_2) \\ \text{non-cusp.}} \ssum{w\\ \ell_K(a_2 w b_1)=n \\ g_{a_1a_2w}(y_{b_1b_2}) \in J_2} \abs{g_{a_1a_2 w}'(y_{b_1b_2})} \\
    ={}& \ssum{(b_1, b_2) \\ \text{non-cusp.}} \ssum{h\in \cH^n \\ y_{b_1b_2} \in \Dom(h)}  \abs{h'(y_{b_1b_2})}\cdot \1_{ J_2}(h(y_{b_1b_2})) \\
    ={}& \ssum{(b_1, b_2) \\ \text{non-cusp.}} \LL^n[\1^*_{J_2}](y_{b_1b_2}), \numberthis\label{eq:strongUNI-laststep}
  \end{align*}
  where $\1^*_{J_2}(x,\vth_0):=\1_{J_2}(x)$ for all $(x,\vth_0)\in K_\Theta$.
  To estimate~$\LL^n[\1^*_{J_2}]$, we first upper-bound the map~$\1_{J_2}$ by a smooth, bounded map~$\psi$ on~$\dD$ such that
  $$ \int_K \psi = O(\eta + \rho^{n/2}), \qquad \|\psi'\|_\infty = O(\rho^{-n/2}). $$
  Letting $\psi^*(x,\vth_0):=\psi(x)$, we deduce~$\|\psi^*\|_{\kappa} = \norm{\psi}_\infty + \norm{\psi}_{(\kappa)} \ll 1 + \norm{\psi'}_\infty^\kappa \ll  \rho^{-\kappa n/2}$. By Proposition~\ref{prop:spectrum-Ls0}.(iii), it follows that
  $$ \LL^n[\psi^*](y) = \frac1\fm\Big(\int_K \psi\Big) \xi(y) + O(\rho^n \|\psi^*\|_{\kappa}) \ll \eta + \rho^{n/2} $$
  uniformly in~$y\in K$. Since~$\1_{J_2}\leq \psi$ and the sum in~\eqref{eq:strongUNI-laststep} over~$(b_1, b_2)$ is finite, we deduce the claimed bound, up to the value of~$\rho$.
\end{proof}

\subsubsection{Estimates on the quasi-inverse for large twists}

We start with an uniform version of~\eqref{eq:holder-condition} for inverse branches of $\cH^n$.
\begin{lemma}\label{lem:ub}
  Assume~\eqref{eq:holder-condition} holds. Then uniformly in $n\geq1$ we have
  \begin{align}
    &\sum_{h\in\cH^n}  \| h'\|_\infty \cdot \|\vphi_{h}\|_{(\kappa_0)} ^{\lambda} \ll1.
  \end{align}
\end{lemma}
\begin{proof}
  Writing each $h\in \cH^n$ as $h=h_1 \circ \dotsb \circ h_n\in \cH^n$ with $h_i\in\cH$, we have
  \begin{align*}
    \|\vphi_h\|_{(\kappa_0)}\ll \sum_{i=1}^{n}  \|\vphi_{h_i}(h_{i+1} \circ \dotsb \circ h_n)\|_{(\kappa_0),\Dom(h)} \ll \sum_{i=1}^{n}  \|\vphi_{h_i}\|_{(\kappa_0)}\rho^{n-i}
  \end{align*} 
  by~\eqref{cascs}, \eqref{hld2} and~\eqref{eq:contraction}. If $\lambda\geq 1$ this implies by H\"older's inequality
  \begin{align*}
    \|\vphi_h\|_{(\kappa_0)}^\lambda \ll \Big(\sum_{i=1}^{n}  \|\vphi_{h_i}\|_{(\kappa_0)}^\lambda\rho^{(n-i)}\Big) \Big( \sum_{i=1}^{n} \rho^{(n-i)}\Big)^{\lambda-1}\ll \sum_{i=1}^{n}  \|\vphi_{h_i}\|_{(\kappa_0)}^\lambda\rho^{(n-i)}
  \end{align*} 
  and the same bounds hold trivially for $\lambda<1$, save for the fact that $\rho$ is replaced by $\rho^\lambda$. Thus, enlarging the value of $\rho\in(0,1)$ if needed, we have
  \begin{align*}
    \sum_{h\in\cH^n}  \| h'\|_\infty \cdot \|\vphi_{h}\|_{(\kappa_0)} ^\lambda & \ll \ssum{n_1+n_2=n-1,\\ n_1,n_2\geq0} \rho^{n_2} \ssum{h_1\in\cH^{n_1},h\in\cH,h_2\in\cH^{n_2},\atop\Dom(h_1\circ h\circ h_2)\neq\emptyset}  \| (h_1\circ h\circ h_2)'\|_\infty \|\vphi_{h}\|_{(\kappa_0)}^\lambda \\
    & \ll \sum_{n_1+n_2=n-1,\atop n_1,n_2\geq0}\rho^{n_2}\ll1,
  \end{align*}
  by~\eqref{eq:holder-condition} and~\eqref{asx1}.
\end{proof}

Next, we prove a Lasota--Yorke estimate for~$\LL_{s,\vt}$.
\begin{lemma} \label{est:ly:t-part}  
  Let~$\rho$ be as in Lemma~\ref{lem:contraction}. Then, there exists $\delta_0>0$ such that for all~$r \geq 1$,~$s=\sigma+i\tau$ with~$1-\delta_0\leq\sigma\ll 1$, $\vt\in\R^d$, and $f\in H^\kappa$, we have
  \[  \| \LL_{s,\vt}^r [f] \|_{(\kappa)}  \ll  A_{2\sigma-1}^{r/2}\big((\nvt^{\min(1,\lambda)/2}+|s|^\kappa) \| f \|_\infty + \rho^{\kappa r} \| f\|_{(\kappa)}\big). \]
\end{lemma}
\begin{proof}
  Let $\delta_0:=(1-\sigma_0)/2$.  As in the proof of~\eqref{asx2}, we have~$\||h'|^{s}\|_{(\kappa)} \ll |s|^\kappa  \|h'\|_\infty^\sigma$. Also, since $\kappa/\kappa_0\leq \min(1,\lambda)/2$, by~\eqref{hld3} we have
  \begin{equation*}\begin{split} 
      \| e^{i \langle \vt, \vth_1\vphi_h \rangle} \|_{(\kappa)}& \ll \nvt^{\min(1,\lambda)/2} \big(1+\|\vphi_h\|_{(\kappa_0)}^{\lambda/2}\big)
    \end{split}
  \end{equation*}
  Thus, by~\eqref{hld1}-\eqref{hld2}, and~\eqref{eq:contraction} we have
  \begin{align*}
    \|  \LL_{s,\vt}^r f\|_{(\kappa)} \ll {}&   \ssum{h\in \cH^r} \|h'\|_{\infty}^\sigma \Big(  \big(|s|^\kappa +\nvt^{\kappa/\kappa_0} \|\vphi_h\|_{(\kappa_0)}^{\kappa/\kappa_0}\big) \|f\|_{\infty}+  \rho^{\kappa r}\|f\|_{(\kappa)}\Big).
  \end{align*}
  The claimed bound then follows since by the Cauchy-Schwarz inequality and~\eqref{asx1} we have
  \begin{align*}
    \|  \LL_{s,\vt}^r f\|_{(\kappa)}^2 \ll {}&   A_{2\sigma-1}^r \ssum{h\in \cH^r} \|h'\|_{\infty} \Big( \big(|s|^{2\kappa} +\nvt^{2\kappa/\kappa_0} \|\vphi_h\|_{(\kappa_0)}^{2\kappa/\kappa_0}\big) \|f\|^2_{\infty}+ \rho^{2\kappa r} \|f\|_{(\kappa)}^2\Big)\\
    \ll{}&A_{2\sigma-1}^{r}\big((|s|^{2\kappa} +\nvt^{\min(1,\lambda)} ) \|f\|^2_{\infty}+ \rho^{2\kappa r} \|f\|_{(\kappa)}^2\big)
  \end{align*}
  where in the last step we used Lemma~\ref{lem:ub} and~\eqref{dfnkappa2}.
\end{proof}

Using Lemma~\ref{est:ly:t-part} and Proposition~\ref{uni:strong}, we can establish the following~$L^2$-type estimate. 

\begin{proposition} \label{dolg:L2:est}
  Let~$\rho$ be as in Lemma~\ref{lem:contraction}.  There exists~$\delta_0>0$ such that for all~$\vt \in\R^d$ and~$s=\sigma+i \tau$ satisfying~$ 1-\delta_0\leq \sigma\ll1$ and~$|\tau| \geq \rho^{-n}$, $n\geq 2$, we have  
  \begin{align}\label{bsmm}
    \int_{K_\Theta} |\LL_{s,\vt}^n [f]|^2 \df\nu \ll A_{2\sigma-1}^n\rho^{\kappa n/2}\|f\|_{\infty
    }\big((1+\nvt^{\min(1,\lambda)})\|f\|_{\infty}+|\tau|^{-\kappa}\|f\|_{(\kappa)} \big) 
  \end{align}
  for all $f\in H^\kappa$.
\end{proposition}

\begin{proof}
  Expanding the square, we have 
  \begin{equation} \label{L2:exp}
    \int_{K_\Theta} | \LL_{s,\vt}^n[f]|^2 d\nu= \sum_{h,k \in \cH^n } I(h,k)
  \end{equation}
  where 
  \begin{align*}
    I(h,k) &= \int_{ \Dom(h)\cap\Dom(k)\times\Theta} e^{i \tau \log \frac{|h'(y)|}{|k'(y)|} }  p_{h}(y,\vth_1)\overline{p_{k}(y,\vth_1) } d\nu, \\
    p_{h}(y,\vth_1) &= e^{i \inpd{\vt, \vth_1  \vphi_h(y)}} |h'(y)|^\sigma \cdot f(h(y), \vth_1 \vth_h(y) )  .
  \end{align*}
  Here we note that since~$f$ is supported on~$K_\Theta$, the sum over~$h, k$ is void unless there exists~$y\in \Dom(h) \cap \Dom(k)$ with~$h(y), k(y) \in K$; since~$n\geq 2$, this implies that~$\im(h), \im(k) \subset K$.
  We decompose~\eqref{L2:exp} into two parts according to the distance~\eqref{def:uni:dist}. Let~$I^-$ denote the sum over the branches satisfying~$\Delta(h,k)\leq \eta$ and~$I^+$ denote the sum over the branches satisfying~$\Delta(h,k)> \eta >0$. 
  
  Regarding~$I^-$, we observe that 
  \[ | I(h,k) | \leq \| p_{h}p_k \|_\infty \ll  \| f \|_\infty^2 \| h' \|_\infty^\sigma \| k' \|_\infty^\sigma  \]
  hence
  by the Cauchy--Schwarz inequality for $\delta_0$ sufficiently small we have
  \begin{align}
    I^-&\leq\| f \|_\infty^2    \Big(\ssum{h,k \in \cH^n \\ \im(h), \im(k) \subset K \\ \Delta(h,k)\leq \eta} \| h' \|_\infty\| k' \|_\infty\Big)^{1/2} \Big(\sum_{h,k \in \cH^n } \|h'\|_\infty^{2\sigma-1}\| k' \|_\infty^{2\sigma-1}\Big)^{1/2}\notag\\
    &\ll \| f \|_\infty^2  A_{2\sigma-1}^{n}(\eta+\rho^n)^{1/2}\label{1csd}
  \end{align}
  by Lemma~\ref{lem:quasicomp} and Proposition~\ref{uni:strong}. 
  
  For~$I^+$, writing $\Psi_{h,k}(\theta)=\log\tfrac{|h'(e^{i\theta})|}{|k'(e^{i\theta})|}$ and assuming $\rho^n<c\eta$ with $c>0$ sufficiently small, by~\eqref{aldelt} we have $|\Psi_{h,k}'(\theta)|>\eta/2$. Moreover, $\Psi_{h,k}''$ is a rational function of  $e^{i\theta}$ of bounded degree, and thus the domain can be divided into $O(1)$ subintervals on each of which \(\Psi_{h,k}'\) is monotone. Thus, recalling that~$\nu$ is proportional to the Lebesgue measure on its first component, using the version of Van der Corput Lemma in~\cite[Lemma~4.1]{BettinDrappeau2022} and~\eqref{hld1} we have
  \begin{align} 
    I^+ &\leq  \sum_{h, k \in \cH^n \atop \Delta(h,k) > \eta} \left| I(h,k) \right| \ll \sum_{h, k \in \cH^n} \left( \frac{\|p_hp_k\|_\infty}{|\tau| \eta}+ \frac{\|p_h\|_{(\kappa)}\|p_k\|_\infty+\|p_h\|_\infty\|p_k\|_{(\kappa)}}{(|\tau| \eta)^\kappa} \right)\notag\\
    &\ll A_{2\sigma-1}^{n}\Big(  \frac{\| f \|_\infty^2}{|\tau|\eta}+\| f \|_\infty  \frac{(\nvt^{\min(1,\lambda)}+|\sigma|^\kappa) \| f \|_\infty + \rho^{\kappa n} \| f\|_{(\kappa)} }{(|\tau|\eta)^\kappa}\Big)
    \label{2csd}  \end{align}
  where the last line follows as in the proof of Lemma~\ref{est:ly:t-part}, with~$s$ replaced by~$\sigma$
in the factor~$|h'|^s$.

  Finally, we pick $\eta=\rho^{n/2}/c$ so that for $|\tau|\geq\rho^{-n}$ the bounds~\eqref{1csd}-\eqref{2csd} give
  \begin{align*}
    I^-&\ll  A_{2\sigma-1}^{n}\| f \|_\infty^2  \rho^{n/4}\\
    I^+    &\ll   A_{2\sigma-1}^{n} \Big( \rho^{\kappa n/2} (1+\nvt^{\min(1,\lambda)}+|\sigma|^\kappa) \| f \|_\infty^2+\frac{ \rho^{\kappa n/2} \| f\|_{(\kappa)}\| f \|_\infty }{|\tau|^\kappa} \Big).
  \end{align*}
  Inserting these results in~\eqref{L2:exp} and recalling $\kappa\leq1/2$ we get the claimed bound.
\end{proof}

Finally we transfer the~$L^2$-bound in Proposition~\ref{dolg:L2:est} to the infinity norm bound.

\begin{proposition} \label{dolg:Linf:est2}
  There exist~$\delta_0,c,C,\tau_1>0$ and $\rho\in(0,1)$ such that for all~$\vt \in\R^d$ and~$s=\sigma+i \tau$ satisfying~$ 1-\delta_0\leq \sigma\ll1$,~$|\tau|\geq\tau_1$, and~$\nvt<|\tau|^{c}$, we have  
  \begin{align*}
    \| \LL_{s,\vt}^{n} [f] \|_\infty  &\ll  \rho^{n}(\|f\|_{\infty}+|\tau|^{-\kappa}\|f\|_{(\kappa)} )
  \end{align*}
  for $n=C\log|\tau|+O(1)$ and all $f\in H^\kappa$.
\end{proposition}

\begin{proof}

  We fix $\rho_1$ such that  $\max(\rho^{\kappa/4},\spr(\cN))<\rho_1<1$ with~$\rho$ as in Lemma~\ref{lem:contraction}. 

  We let $n=\ell_1+\ell_2$.
  By using the Cauchy--Schwarz inequality and~\eqref{asx1} we have
  \begin{align*}
    \| \LL_{s,\vt}^{n} [f] \|^2_\infty=\| \LL_{s,\vt}^{\ell_1} [\LL_{s,\vt}^{\ell_2}[f]] \|^2_\infty  &\ll A_{2\sigma-1}^{\ell_1} \| \LL_{1,0}^{\ell_1}[ | \LL_{s,\vt}^{\ell_2} [f] |^2 ] \|_\infty.
  \end{align*}
  By Proposition~\ref{prop:spectrum-Ls0}.(iii) and \eqref{eq:def-eigenproj} we have $\LL^{\ell_1}_{1,0}[f]\ll \int f \,d\nu+\rho_1^{\ell_1}\|f\|_{\kappa}$ for $f\in H^\kappa$ and thus we deduce
  \begin{align}\label{cacfd}
    \| \LL_{s,\vt}^{n} [f] \|^2_\infty  &\ll A_{2\sigma-1}^{\ell_1} \left( \int_{K_\Theta} | \LL^{\ell_2}_{s,\vt}[f] |^2 d\nu +  \rho_1^{\ell_1} \| | \LL_{s,\vt}^{\ell_2}[f]|^2  \|_{(\kappa)}  \right).
  \end{align}
  By Lemma~\ref{est:ly:t-part} and~\eqref{asx1}, and using $A_{\sigma}\ll A^{1/2}_{2\sigma-1}$ by Cauchy-Schwarz, we have 
  \begin{align*}
    \| | \LL_{s,\vt}^{\ell_2}[f] |^2 \|_{(\kappa)}  &\leq  \| \LL_{s,\vt}^{\ell_2}[f] \|_{(\kappa)} \| \overline{\LL_{s,\vt}^{\ell_2}[f] }\|_{\infty} + \| \LL_{s,\vt}^{\ell_2}[f] \|_\infty \| \overline{\LL_{s,\vt}^{\ell_2}[f] }\|_{(\kappa)}\\
    &\ll A_{2\sigma-1}^{\ell_2} \|f \|_\infty \big((1+\nvt^{\min(1,\lambda)}+|\tau|^\kappa) \| f \|_\infty + \rho_1^{\ell_2} \| f\|_{(\kappa)}\big).
  \end{align*}
  Inserting this bound and~\eqref{bsmm} into~\eqref{cacfd} we deduce that for $|\tau|>\rho^{-\ell_2}$
  \begin{align*}
    \| \LL_{s,\vt}^{n} [f] \|^2_\infty  &\ll A_{2\sigma-1}^{n} \rho_1^{\kappa \ell_2/2}\|f\|_{\infty
    }((1+\nvt^{\min(1,\lambda)})\|f\|_{\infty}+|\tau|^{-\kappa}\|f\|_{(\kappa)} )\\
    &\quad+A_{2\sigma-1}^{n} \rho_1^{\ell_1}\|f \|_\infty \big((1+\nvt^{\min(1,\lambda)}+|\tau|^\kappa) \| f \|_\infty + \rho_1^{\ell_2} \| f\|_{(\kappa)}\big)\\
    &\ll A_{2\sigma-1}^{n} (\rho_1^{\kappa \ell_2/2}+\rho_1^{\ell_1/2})\|f\|_{\infty
    }((1+\nvt^{\min(1,\lambda)})\|f\|_{\infty}+|\tau|^{-\kappa}\|f\|_{(\kappa)} ),
  \end{align*}
  where the last bound holds with the additional assumption $\rho_1^{-\ell_1/(2\kappa)}>|\tau|$. Finally, we pick $\ell_1=c_1\log |\tau|+O(1)$, $\ell_2=c_2\log|\tau|+O(1)$ for some fixed $c_1,c_2>0$ chosen so that $\rho_1^{-\ell_1/(2\kappa)}>|\tau|>\rho^{-\ell_2}$ is satisfied. Notice $n=(1+c_1/c_2)\ell_2+O(1)=(1+c_2/c_1)\ell_1+O(1)$ and thus $\rho_1^{\kappa \ell_2/2}+\rho_1^{\ell_1/2}\ll\rho_2^n$ with $\rho_2=\max(\rho_1^{\kappa/(2+2c_1/c_2)},\rho_1^{1/(2+2c_2/c_1)})$. 
  By continuity of $A_\sigma$ we can pick $\delta>0$ small enough so that  $A_{2\sigma-1}<\rho_2^{-1/2}$. With these assumptions we then have
  \begin{align*}
    \| \LL_{s,\vt}^{n} [f] \|^2_\infty  &\ll  \rho_2^{n/2}\|f\|_{\infty
    }((1+\nvt^{\min(1, \lambda)})\|f\|_{\infty}+|\tau|^{-\kappa}\|f\|_{(\kappa)} ).
  \end{align*}
  Finally, we notice that $\rho_2^{n/4}\ll |\tau|^{-c_3}$ for some $c_3>0$. Thus, assuming $\nvt\ll |\tau|^c $ for some small enough $c>0$ we deduce the claimed bound.
\end{proof}

\begin{proposition}\label{prop:dolgopyat-estimates}
  There exists~$c, \delta_0, \tau_1>0$ and $\rho\in(0,1)$ such that for~$s=\sigma+i\tau$ and $\vt\in\R^d$ satisfying
  $$ 1 - \delta_0\leq \sigma \ll 1, \qquad \abs{\tau} \geq \tau_1, \qquad \nvt\leq |\tau|^c$$
  the operator~$\LL_{s,\vt}$ acting on~$H^\kappa$ has spectral radius at most~$\rho$.
  More precisely, there exists $C>0$ such that for
  \begin{align}\label{def1t}
    \|\cdot\|_{1,\tau}:=\|\cdot\|_\infty+(1+|\tau|)^{-\kappa}\|\cdot\|_{(\kappa)}
  \end{align}
  and under the above conditions
  \begin{align}\label{bdasc} \sum_{n\geq 0}\| \LL_{s, \vt}^n[f]  \|_{1, \tau} \ll \abs{\tau}^{C} \|f\|_{1,\tau}\qquad \forall f\in H^\kappa.
  \end{align}
\end{proposition}

\begin{proof}
  Let $\delta_0,c,C,\tau_1,n$, $\vt$ and $s$ as in Proposition~\ref{dolg:Linf:est2}.
  For $f\in H_\kappa$ and some $\rho\in(0,1)$ we have
  \begin{align*}
    \| \LL_{s,\vt}^{2n} [f] \|_{1,\tau}
    &\ll  A_{2\sigma-1}^{n/2}\big(1+\nvt^{\min(1,\lambda)/2}/|\tau|^\kappa \big) \| \LL_{s,\vt}^{n}[f] \|_\infty + \rho^{n} |\tau|^{-\kappa}\| \LL_{s,\vt}^{n}[f]\|_{(\kappa)}\\
    &\ll  \rho^n A_{2\sigma-1}^{n/2}\big(1+\nvt^{\min(1,\lambda)/2}/|\tau|^\kappa\big) \| f \|_{1,\tau} \end{align*}
  where in both lines we applied Proposition~\ref{dolg:Linf:est2} and Lemma~\ref{est:ly:t-part}. Reducing $\delta_0$ and $c$ and enlarging $\tau_1$, if needed, we then deduce
  \begin{align*}
    \| \LL_{s,\vt}^{2n} [f] \|_{1,\tau}  &\leq  \rho^{n/2} \| f \|_{1,\tau} 
  \end{align*}
  This then implies that~$\LL_{s, \vt}$ acting on~$(H^\kappa, \|\cdot\|_{1,\tau})$ satisfies~$\spr(\LL_{s,\vt})\leq \rho^{1/4}$, and $\|( \id -\LL_{s,\vt}^{2n})^{-1} \|_{1,\tau}\ll1$. Since we also have the crude bound~$\|\LL_{s,\vt}^r\|_{1,\tau}\ll e^{O(r|1-\sigma|)}$ by~\eqref{asx1} and Lemma~\ref{est:ly:t-part}, we deduce 
  \begin{align*}
    \sum_{r\geq 0}\|\LL_{s,\vt}^r \|_{1,\tau}  &\leq   \big(1 - \| \LL_{s,\vt}^{2n}\|_{1,\tau}\big)^{-1} \sum_{r=0}^{2n-1}\|\LL_{s,\vt}^r\|_{1,\tau}\ll  ne^{O(n|1-\sigma|)}\ll |\tau|^{O(|\sigma-1|)}\log|\tau|,
  \end{align*}
  Since $\|\|_\kappa$ and $\|\|_{1,\tau}$ are equivalent we obtain that the operator~$\LL_{s,\vt}$ acting on~$H^\kappa$ has spectral radius strictly less than~$1$, uniformly in $|\tau|\geq\tau_1$. This concludes the proof.
\end{proof}

\subsection{Spectral properties of the perturbed operator}
We are now ready to generalize Proposition~\ref{prop:spectrum-Ls0} to~$\LL_{s, \vt}$ with arbitrary~$s, \vt$. We recall that we are assuming Hypothesis~\ref{hyp:moments} and~\eqref{dfnkappa}.

\begin{proposition}\label{prop:spectrum-Ls2}
  For all~$s \in \C$ with~$\Re(s)\geq\sigma_0$ and~$\vt \in \R^d$, the operator~$\LL_{s,\vt}$ acting on $H^\kappa$ satisfies the following.
  \begin{enumerate}[(i)]
    \item For some neighborhood $U$ of~$(1,0)$ and all~$(s,\vt)\in U$, the operator~$\LL_{s,\vt}:H^\kappa \rightarrow H^\kappa$ has a simple maximal eigenvalue~$\lambda_{s,\vt}$ and the rest of the spectrum is contained in a disk of radius strictly smaller than~$\inf_{(s,\vt)\in U}|\lambda_{s,\vt}|$.
    \item For all $\tau_0>0$, there exists $t_0,\delta>0$, $\rho\in(0,1)$ such that $\spr(\LL_{s, \vt})<\rho$ for $1-\lambda<\Re(s)\ll 1$, $\abs{\Im(s)}\geq \tau_0$, and $\nvt<t_0$.
  \end{enumerate}
\end{proposition}

\begin{proof}
  The first item is a consequence of Proposition~\ref{prop:spectrum-Ls0}, equation~\eqref{sbsa} and perturbation theory~\cite[Theorem~IV.3.16]{Kato1995}.

  For the second item, first assume that~$\Re(s) = 1$ and~$\vt = 0$. We already know that~$\spr(\LL_{s, 0})\leq 1$ by Proposition~\ref{prop:spectrum-Ls0} and positivity.
  Assume by contradiction that~$\spr(\LL_{s, \vt})=1$, which by Lemma~\ref{est:ly:t-part} and Hennion's criterion~\cite{Hennion1993}, means that there is an eigenvalue~$\lambda$ of~$\LL_{s, \vt}$ of modulus~$1$. Let~$f$ be an eigenfunction for~$\lambda$, and~$\tilde f := f/\xi$. Multiplying~$f$ by a constant if necessary, and by the compactness of~$K_\Theta$, we may find~$(x_0, \vth_0) \in K_\Theta$ such that~$\|\tilde f\|_\infty = |\tilde f(x_0, \vth_0)| = 1$.
  We first assume $\vt=0$, $\Re(s)=1$ and write~$s = 1+i\tau$, with $\tau\in \R_{\neq0}$. Following the arguments in~\cite[p.~365]{BaladiVallee2005}, we deduce that~$|\tilde f| = 1$ identically on~$K_\Theta$, and that for all~$h\in\cH$ and all~$(y, \vth_1)\in K_\Theta$ with~$y\in \Dom(h)$ and~$h(y)\in K$, we have
  \begin{equation}
    \frac{\lambda \tilde f(y, \vth_1)}{\tilde f(h(y), \vth_1 \vth_h(y))} =  \abs{h'(y)}^{i\tau}
    \label{eq:expr-eigenval-argument}
  \end{equation}
  Let~$b_1\dotsb b_r\in\Sigma^*$ be any cuspidal cycle such that~$y\in [b_1]$. Let~$c\in\Sigma$ be such that~$[c, b_1]\subset K$. Finally, let~$N\geq 1$ be the order of~$\Theta$, and~$w := (b_1\dotsb b_r)^N$. For all~$\ell\geq 1$, the word~$cw^{\ell}$ belongs to~$\Lambda_K(1)$, and defines a branch~$h = h_{c w^\ell, 1}\in \cH$ for which~$y\in \Dom(h)$.
  Using the cocycle property~\eqref{eq:vth-cocycle-letter}, we have
  $$ \vth_h(y) = \vth_{b_1 \dotsb b_r}(b_1)^{\ell N} \vth_c(b_1) = \vth_c(b_1)  $$
  which is independent of~$\ell$. On the other hand, as~$\ell\to\infty$, we have
  $$ h(y) = g_c g_w^\ell y \to [c, w, w, \dotsc] \in K. $$
  By continuity, we deduce that the left-hand side of~\eqref{eq:expr-eigenval-argument} converges as~$\ell\to\infty$.
  On the other hand, by Lemma~\ref{lem:parabolic}.(3), we have~$h'(y) = g_c'(g_w^\ell y) g_{w^\ell}'(y) \sim C \ell^{-2}$ for some constant~$C$ as~$\ell\to \infty$.
  Therefore the right-hand side of~\eqref{eq:expr-eigenval-argument} cannot converge to a constant, yielding the desired contradiction. Thus,~$\spr(\LL_{1+i\tau, 0})<1$ for $\tau\neq0$.

  By perturbation theory~\cite[Theorems~IV.3.1 and IV.2.14]{Kato1995} and compactness, for arbitrary fixed~$\tau_1>0$, we deduce $\spr(\LL_{s, \vt})<\rho$ for $1+\delta>\Re(s)>1-\delta$, $\tau_0\leq \abs{\Im(s)}\leq \tau_1$,  $\nvt<t_0$ for some $\rho\in(0,1)$ and $t_0,\delta>0$ small enough. We pick~$\tau_1$ as in Proposition~\ref{prop:dolgopyat-estimates}; this Proposition then ensures that the inequality~$\spr(\LL_{s, \vt})<\rho$ persists for all~$\Im(s)\geq \tau_1$. This concludes the proof for~$\Re(s)<1+\delta$; in the remaining case~$1+\delta\leq \Re(s)\ll 1$ we have~$\spr(\LL_{s,\vt}) \leq A_{2\sigma-1}^{1/2} < \rho$ by \eqref{asx1}, Lemma~\ref{est:ly:t-part} and picking~$\rho<1$ larger if necessary.
\end{proof}

\section{Asymptotic behaviour of the leading eigenvalue}\label{sec:asympt-behav-lead}

In this section, we study the behaviour of the leading eigenvalue~$\lambda(s,\vt):=\lambda_{s,\vt}$ of~$\LL_{s,\vt}$ for~$(s,\vt)\approx (1,0)$.
Some of our estimates will be given also in a different form under the assumption~\ref{hyp:moments-cuspidal}.
Most of our results are based on the perturbation estimates from~\cite{Kloeckner2019} (as in~\cite{BettinDrappeau2022}), which take the bounds~\ref{lem:perturbation-LL1t} as input.
We mention however that some of our foregoing statements could be proved under weaker conditions~\cite{KellerLiverani1999} using from Lasota-Yorke inequalities.

The following variant of Kac's formula~\cite[Theorem~9.3]{PollicottYuri1998} (see also~\cite[eq.~(3.1)]{PeneThomine2021}) relates integrals relative to~$\cF_c$ to integrals relative to~$\cF_K$. We recall the definition~\eqref{eq:def-return-Fc} of the return time of~$\cF_c$ to~$K$.

\begin{lemma}\label{lem:wrap}
  For~$f\in L^1(\dD)$ we have
  \begin{align*}
    \int_K \sum_{m=0}^{r_x-1}f(\cF_c^m(x))\xi(x)\df s(x)=\int_{\dD} f(x)\varphi(x)\df s(x)
  \end{align*}
\end{lemma}
\begin{proof}
Let~$n\in\N$ and let $A$ be the full-measure subset of $\dD$ consisting of the points whose $\cB$-iterates  visit $K$ infinitely often, as in the proof of Proposition~\ref{prop:varphi-invar}. Recall the return time~$r_x$ defined in~\eqref{eq:def-return-Fc}.  We write~$K_n=\{x\in K\cap A\mid r_x=n\}$ and~$K_n^{>}=\{x\in K\cap A \mid r_x>n\}=\{x\in K\cap A\mid \cF_c^i(x)\notin K, i=1,\dots,n\}$ and $K_0^>=K\cap A$. We have
  \begin{align*}
    \int_K \sum_{m=0}^{r_x-1}f(\cF_c^m(x))\xi(x)\df s(x)&=\sum_{n=1}^{\infty} \sum_{m=0}^{n-1}  \int_{K_n} f(\cF_c^m(x))\xi(x)\df s(x)
    =\sum_{m=0}^{\infty}  \int_{K_m^>} f(\cF_c^m(x))\xi(x)\df s(x).
  \end{align*}
  The contribution of~$m=0$ is~$\int_{K} f(x)\xi(x)\df s(x)=\int_{K} f(x)\varphi(x)\df s(x)$ since~$\xi(x)=\varphi(x)$ on~$K$.  Now for~$m\geq1$ we have
  $$
  K^>_m=(K\cap A)\cap \sideset{}{^*}\bigcup_{h=h_1\circ \dotsb\circ h_m} h(\Dom(h))
  $$
  where~$*$ indicates the condition that~$h_1,\dots,h_{m}\in\cH_c^\sharp$ are such that~$\Dom(h_1\circ \dotsb\circ h_m)\neq\emptyset$. In particular, $\Dom(h)\subset \Dom(h_m)\subset \dD\setminus K$.   Thus, 
  \begin{align*}
    \sum_{m\geq1}\int_{K_m^>} f(\cF_c^m(x))\xi(x)\df s(x)&= \sideset{}{^*}\sum_{h=h_1\circ \dotsb\circ h_m,\, m\geq1 }\int_{\Dom(h)\cap h^{-1}(K)}   f(y) \xi(h(y))|{h}'(y)|\df s(y)\\
    &= \int_{\dD\setminus K} f(y)\sideset{}{^*}\sum_{\substack{h=h_1\circ \dotsb\circ h_m,\, m\geq1\\ y\in\Dom(h),\, h(y)\in K }} \xi(h(y))|{h}'(y)|\df s(y).
  \end{align*}
  The result then follows provided we prove
  \begin{equation}\label{eq:rel-sumxi-phi}
    \sideset{}{^*}\sum_{\substack{h=h_1\circ \dotsb\circ h_m,\, m\geq1\\ y\in\Dom(h),\, h(y)\in K }} \xi(h(y))|{h}'(y)|=\varphi(y) \qquad (y \in \dD' \setminus K).
  \end{equation}
  Write~$y\in[a]\setminus K$ for some~$a\in \Sigma$. By~\eqref{chva} the contribution of~$m=1$ to the left hand side is
  $$ \sideset{}{^*}\sum_{\substack{b\in \Sigma,\, | a-\hat b|>1}} \xi(h_b(y))|{h_b}'(y)|=  \int_{t\in \dD\setminus ([a]\cup[a+1]\cup[a-1])}  \frac{\zeta(\df y, \df t)}{\df s(y)}. $$
  Now assume~$m>1$ and thus~$y\in[a,\hat a+\delta]$ with $\delta \in \{\pm 1\}$. If we write~$h$ as~$h=h_{w}$ for~$w=a_1\dotsb a_k$ with~$w\in\Sigma^*$, then the condition~$h(y)\in K$ is equivalent to~$[a_1,a_2]\subset K$ (i.e.~$[\hat a_2,\hat a_1]\seq K$), the condition~$y\in\Dom(h)$ is equivalent to~$\hat a_k=a+\delta$ and, given the other conditions, the condition~$h_i\in \cH_c^\sharp$ for all~$i$ and~$m\geq2$ is equivalent to~$\ell_{K}(a_2\dotsb a_k)=1$ and~$k\geq2$.
  Thus, by~\eqref{chva}
  \begin{align*}
    \sideset{}{^*}\sum_{\substack{h=h_1\circ \dotsb\circ h_m,\, m\geq2\\ y\in\Dom(h),\, h(y)\in K }} \xi(h(y))|{h}'(y)|&=\sum_{\substack{w=a_1\dotsb a_k\in\Sigma^*,\, k\geq2\\ \hat a_k=a+\delta,\,[\hat a_2,\hat a_1]\seq K,\\\ell_K(a_2\dots a_k)=1}}\xi(h_w(y))|h_w'(y)|\\
    &=\sum_{\substack{w=a_1\dotsb a_k\in\Sigma^*,\, k\geq2\\ \hat a_k=a+\delta,\,[\hat a_2,\hat a_1]\seq K,\\\ell_K(a_2\dots a_k)=1}}   \int_{t\in [\hat a_k,\dots,\hat a_1]}   \frac{\zeta(\df y, \df t)}{\df s(y)}\\
    &=   \int_{t\in [a+\delta]}   \frac{\zeta(\df y, \df t)}{\df s(y)}
  \end{align*}
  and~\eqref{eq:rel-sumxi-phi} follows by the definition~\eqref{eq:def-invar-Fc} of~$\varphi$.
\end{proof}

\subsection{Perturbation theory and existence}

\begin{lemma}\label{lem:approx-partialder}
  There exist~$\delta,t_0>0$ such that whenever~$\abs{s-1}\leq\delta$ and~$\nvt\leq t_0$, we have
  \begin{equation}
    \label{eq:partialder-lambda}
    \partial_{10}\lambda(s,\vt) = -\fd + O(\abs{s-1} + \nvt +\nvt^{\alpha-\eps}),
  \end{equation}
  where 
  \begin{equation}
    \fd  :=\int_K  \log \abs{\cF_K'(y)} \xi(y)\frac{\df s(y)}{\fm}  = \int_{\dD} \log \abs{\cF_c'(y)} \varphi(y)\frac{\df s(y)}{\fm} >0\label{eq:def-fd}
  \end{equation}
  and $\fm$ as defined in~\eqref{dfnm}.
\end{lemma}

\begin{proof}
  Proceeding as in~\cite[Lemma~7.1]{BettinDrappeau2022} and keeping track of the error terms using~\eqref{sbsa}, we obtain the estimate~\eqref{eq:partialder-lambda} with  
  $$\fm\fd=-\int_K \partial_{1,0} \LL_{1, 0}[\xi]\df s(y)=-\int_K \sum_{h\in\cH} \abs{h'(y)} \log \abs{h'(y)}\xi(h(y)) \df s(y) = \int_K  \log \abs{\cF_K'(y)} \xi(y)\df s(y)  $$
  where the last equality follows by~$h'(y)=1/\cF_K'(h(y))$ and~\eqref{eq:transferop-invar}.

  Notice that for all~$m\geq1$ by the chain rule and~\eqref{eq:transferop-invar} we have
  \begin{align*}
    \frac1m\int_K  \log \abs{(\cF_K^m)' (y)} \xi(y)\df s(y)&=\frac1m\sum_{j=0}^{m-1}\int_K  \log \abs{\cF_K' (\cF_K^{j} (y))} \xi(y)\df s(y)\\
    &=\int_K  \log \abs{\cF_K'(y)} \xi(y)\df s(y)=\fm\fd.
  \end{align*}
  The fact that~$\fd>0$ follows from the left-hand side expression, since by~\eqref{eq:contraction} we have~$\abs{(\cF_K^m)'} > 1$ for~$m$ large enough. Finally, the second expression for~$\fd$ in~\eqref{eq:def-fd} follows by Lemma~\ref{lem:wrap} since~$\log \abs{\cF_K'(y)}=\sum_{m=0}^{r_y-1}\log \abs{\cF_c'(\cF_c^m(y))}$.
\end{proof}

By the implicit function theorem (cf.~\cite[Lemma~7.2]{BettinDrappeau2022}) we obtain the following.
\begin{lemma}\label{lem:existence-s0}
  For all~$\eta>0$, there exists~$t_0>0$ and a unique function~$s_0:[-t_0,t_0]^d\to\C$ such that~$s_0(0) = 1$ and, for~$\nvt\leq t_0$,
  $$ \abs{s_0(\vt) - 1} \leq \eta, \qquad \lambda(s_0(\vt), \vt) = 1. $$
\end{lemma}
\begin{proof}
  This follows from a general form of the implicit functions theorem, \textit{e.g.}~\cite[Theorem~1.1]{Kumagai1980}, whose hypotheses are satisfied by virtue of Lemma~\ref{lem:approx-partialder}.
\end{proof}

\subsection{The sub-CLT case}\label{sec:sub-clt-case}
Let
\begin{align*}
  \cJ_K &:= \int_{K_\Theta} \big( \e^{i\langle t, v  \Phi_K(y)\rangle}-1\big) \xi(y) \frac{\df s(y) \df v}{\fm}, \\
  \cJ_c &:= \int_{\dD_\Theta} \big( \e^{i\langle t, v  \Phi_c(y)\rangle}-1\big) \varphi(y) \frac{\df s(y) \df v}{\fm}, & (\text{assuming~\ref{hyp:moments-cuspidal}})
\end{align*}
where~$\Phi_K, \Phi_c$ are defined in~\eqref{eq:def-PhiK} and~\eqref{eq:def-Phic} respectively.

\begin{lemma}\label{lem:estim-s0-subclt}
  \begin{enumerate}
    \item For~$\nvt \leq t_0$, we have
    \begin{equation}
      s_0(\vt) - 1 = \frac{1}{\fd}\cJ_K + O(\nvt^2 + \nvt^{2\alpha-2\eps}).\label{eq:expansion-s0t}
    \end{equation}
    \item If we assume~\ref{hyp:moments-cuspidal} and
    \begin{align}
      &\ssum{h_1, h_2,h_3 \in \cH_c} \|h_1'\|_{\infty} \|h_2'\|_{\infty}\|h_3'\|_{\infty} \| \vphi_{h_1}|_{\im(h_2)} \|_\infty\| \vphi_{h_2}|_{\im(h_3)} \|_\infty < \infty, \label{eq:moment-condition-cuspidal2}
    \end{align}
    then the same asymptotic formula holds with~$\cJ_c$ instead of~$\cJ_K$.
  \end{enumerate}
\end{lemma}
\begin{proof}
  We proceed as in~\cite[Lemma~7.3]{BettinDrappeau2022} and obtain by~\eqref{eq:moment-condition} and~\eqref{sbsa}
  \begin{align}\label{eq:expansion-lambdast}
    \lambda(s, \vt)     = {}& {- \fd} (s-1)+ \int_{K_\Theta} \LL_{1, \vt}[\xi] \df \nu + O(\abs{s-1}^2 + \nvt^2 + \nvt^{2\alpha-2\eps}).
  \end{align}
  Noting that~$\LL_{1,\vt}[\xi] = \LL[\e^{i\inpd{\vt, \Phi_K}} \xi]$ and using~\eqref{eq:transferop-invar}, we obtain by Hypothesis~\ref{hyp:moments} that
  \begin{equation}
    \lambda(s, \vt) =1 -\fd(s-1) + \cJ_K + O(\abs{s-1}^2 + \nvt^2 + \nvt^{2\alpha-2\eps}).\label{eq:expansion-lambdast-2}
  \end{equation}

  Assume next Hypothesis~\ref{hyp:moments-cuspidal}.
  We start from~\eqref{eq:expansion-lambdast}. By Lemma~\ref{lem:expansion-L1t-subclt} applied with~$f=\xi$ we obtain
  \begin{align*}
    \int_{K_\Theta}    (\LL_{1, \vt}-\LL)[\xi]\df\nu&=\tilde\cJ_c+\tilde\cI_c+O(\nvt^2 + \nvt^{2\alpha-2\eps}).
  \end{align*}
  where 
  \begin{align*}
    \tilde\cJ_c&=
    \int_{K_\Theta}\big((\CC_{1, \vt}-\CC)+\CC (\Id-\CC^\sharp)^{-1}(\CC_{1, \vt}^\sharp-\CC^\sharp)\big)(\Id-\CC^\sharp)^{-1}[\xi]\df\nu\\
    \tilde\cI_c&=
    \int_{K_\Theta}\big((\CC_{1, \vt}-\CC)+\CC (\Id-\CC^\sharp)^{-1}(\CC_{1, \vt}^\sharp-\CC^\sharp)\big)(\CC^\sharp_{1, \vt}-\CC^\sharp)(\Id-\CC^\sharp)^{-1}[\xi]\df\nu\\
  \end{align*}
  We start by analyzing $\tilde\cJ_c$. Expanding the definitions of the transfer operators, the integrand in the right-hand side becomes
  \begin{align*}
    & \ssum{h=h_1\circ\dotsb \circ h_m, m\geq1 \\ y\in \Dom(h)} |h'(y)| \sum_{i=1}^m \Big( \e^{i \inpd{\vt, \vth_1 \vth_{h_{i+1}\circ \dotsb \circ h_m}(y)\vphi_{h_i}(h_{i+1}\circ \dotsb \circ h_m(y))}}-1\Big) \xi(h(y)) \\
    ={}& \ssum{h\in \cH \\ y\in \Dom(h)}|h'(y)|\sum_{j=0}^{r_{h(y)}-1} \Big(e^{i \inpd{\vt,  \Phi_c(\cF_{c,\Theta}^{j}(h(y),\vth_1\vth_h(y)))}}-1\Big)  \xi(h(y))
  \end{align*}
  where the sum on the first line is over~$h_1,\dots,h_m$ with~$h_1,\dots,h_{m-1}\in\cH_c^\sharp$ and~$h_m\in\cH_c$, so that by Lemma~\ref{lem:decomp-Hc-HK}~$h$ ranges over~$\cH$ with~$m=r_{h(y)}$. We integrate over $(y,\vth_1)\in K_\Theta$ and, after expressing $\Phi_c\circ \cF_{c,\Theta}^j$ in terms of $\Phi_c\circ \cF_c^j$, we make a  suitable change of variable in $\vth_1$ to absorb the multiplier. Lemma~\ref{lem:wrap} then yields that $\tilde\cJ_c=\cJ_c$.
  In the same way, the integrand in $\tilde\cI_c$ is 
  \begin{align*}
    \ssum{h\in \cH \\ y\in \Dom(h)} |h'(y)| \sum_{j=0}^{r_{h(y)}-2} \Big(e^{i \inpd{\vt,  \Phi_c(\cF_{c,\Theta}^{j}(h(y),\vth_1\vth_h(y)))}}-1\Big) \Big(e^{i \inpd{\vt,  \Phi_c(\cF_{c,\Theta}^{j+1}(h(y),\vth_1\vth_h(y))}}-1\Big)  \xi(h(y)).
  \end{align*}
  Expressing back~$h = h_1\circ \dotsb \circ h_m$ and passing to absolute values, this is bounded by
  \begin{align*}
    \ll{}& \nvt^2 \ssum{h = h_1 \circ\dotsb\circ h_m \\ m\geq 2 \\ y \in \Dom(h)} \abs{h'(y)} \sum_{r=1}^{m-1} \|\vphi_{h_r}\|_{\infty, \im h_{r+1}} \|\vphi_{h_{r+1}}\|_{\infty, \im h_{r+2}} \\
    \ll{}& \nvt^2 \sum_{h_1, h_2, h_3 \in \cH_c} \norm{h'_1}_\infty \norm{h'_2}_\infty \norm{h'_3}_\infty \|\vphi_{h_1}\|_{\infty, \im h_2} \|\vphi_{h_2}\|_{\infty, \im h_3}
  \end{align*}
  where~$\im h_{r+2}$ in the first line and~$\im h_3$ in the second line are replaced by~$K$ if~$r=m-1$; here we used the boundedness of
  $$ \sum_{j\geq 0} \ssum{h = h_1 \circ \dotsb \circ h_j \\ h_1,\dots,h_j \in \cH_c^\sharp} \norm{h'}_\infty $$
  which follows from~\ref{lem:convergence-H-CF}. Thus the hypothesis~\eqref{eq:moment-condition-cuspidal2} readily implies that~$\tilde \cI_c \ll \nvt^2$ as claimed.

  Since~$\cJ_K, \cJ_c, \tilde\cI_c\ll \nvt^{\alpha-\eps}+\nvt$ we deduce~$s_0(\vt) - 1\ll \nvt^{\alpha-\eps}+\nvt$ in either case~\ref{hyp:moments} or case~\ref{hyp:moments-cuspidal} and~\eqref{eq:moment-condition-cuspidal2}. Setting~$s=s_0(\vt)$ in~\eqref{eq:expansion-lambdast} we then deduce the claimed result.
\end{proof}
\begin{corollary}\label{corest}
  Let~$\alpha\geq 1$, then for all~$\eps>0$ and all~$\nvt \leq t_0$, we have
  $$ s_0(\vt) - 1 =  i \langle\vt, \mu_0\rangle  + O(\nvt^2 + \nvt^{\alpha}). $$
  where
  \begin{equation}
    \mu_0 := \frac{\frac1{|\Theta|}\sum_{\vth\in\Theta}\vth}{\fd} \times \left\{\begin{aligned} &\int_K \Phi_K(x) \xi(x) \frac{\df s(x)}{\fm}, &&  \\ &\int_{\dD} \Phi_c(x) \varphi(x) \frac{\df s(x)}{\fm}, && (\text{assuming~\ref{hyp:moments-cuspidal} and~\eqref{eq:moment-condition-cuspidal2}}). \end{aligned}\right.\label{eq:def-muphi}
  \end{equation}
\end{corollary}
\begin{proof}
  This follows immediately from the Taylor expansion~$\e^{iu} = 1 + iu + O(u^2+ u^{\alpha})$ and the bound~\eqref{eq:moment-condition}. 
  Notice that the two expressions for~$\mu_0$ coincide by~\eqref{eq:ref-PhiK-Phic} and Lemma~\ref{lem:wrap}.
\end{proof}
\subsection{The CLT case}

In this section we assume
\begin{equation}
  \label{eq:assumptions-CLT}
  \text{Hypothesis~\ref{hyp:moments} holds with the additional condition } \alpha>2.
\end{equation}
Our goal is to extract the term of order~$\nvt^2$ in the analysis above, following~\cite[Section~7.2]{BettinDrappeau2022}.
In order to do so, we introduce the following notation:
\begin{align*}
  \KK[f] := {}& \frac1\xi (\Id - \NN)^{-1} \NN[f\xi], \\
  \vpsi := {}& \Phi_K - \mu_0 \log\abs{\cF_K'}, \\
  \vchi := {}& \KK [\vpsi],
\end{align*}
where $\NN$ is as in Proposition~\ref{prop:spectrum-Ls0}. To see that~$\vchi$ is well-defined, note first that~$\PP[\vpsi\xi] = 0$ by construction, so that
\begin{align*}
  \NN[\vpsi \xi](y, \vth) ={}& \LL[\vpsi \xi](y, \vth) \\
  ={}& \sum_{h\in\cH} \abs{h'(y)} \vpsi(h(y), \vth\vth_h(y)) \xi(h(y)) \\
  ={}& \sum_{h\in\cH} \abs{h'(y)} \xi(h(y)) \big( \Phi_K(h(y), \vth \vth_h(y)) - \mu_0 \log\abs{\cF_K'(h(y))}\big) \\
  ={}& \sum_{h\in\cH} \abs{h'(y)} \xi(h(y)) \big( \vth \vphi_h(y) + \mu_0 \log\abs{h'(y)}\big)
\end{align*}
by~\eqref{eq:PhiK-vphi}. The conditions~\ref{hyp:moments} along with Lemma~\ref{prop:Lst-bounded} ensures that this defines a function in~$H^\kappa$. Since~$\spr(\NN)<1$ on~$H^\kappa$ by Proposition~\ref{prop:spectrum-Ls0}, this show that~$\vchi$ is well-defined.

\begin{lemma}\label{lem:estim-s0-clt}
  Assuming~\eqref{eq:assumptions-CLT}, we have
  \begin{equation}
    s_0(\vt) - 1 =  i \langle\vt, \mu_0\rangle - \tfrac12 \vt^T \Sigma_0 \vt + O(\nvt^3 + \nvt^{\alpha-\eps}),\label{eq:approxs0-alpha-2}
  \end{equation}
  with
  \begin{equation}
    \Sigma_0 := \frac{1}{\fd}\int_{K_\Theta} (\vpsi + \vchi - \vchi\circ \cF_\Theta\big)\cdot(\vpsi + \vchi - \vchi\circ \cF_\Theta\big)^T \xi\df\nu.\label{eq:def-sigmagam}
  \end{equation}
\end{lemma}

\begin{remark}
  The remark in~\cite[p.~1408]{BettinDrappeau2022} applies to~$\Sigma_0$: this matrix is symmetric, positive semi-definite, and definite if and only if the image set~$(\vpsi + \vchi - \vchi\circ \cF_\Theta)(K_\Theta)$ span~$\R^d$; the analogue of the expansion~(7.9) \emph{ibid.} holds as soon as~$\alpha>1$, and the matrix~$\Sigma_0$ is well-defined whenever~$\alpha>2$.
\end{remark}

\begin{proof}
  We let~$s = s_0(\vt) = 1+i \langle\vt, \mu_0\rangle  + O(\nvt^2 )$ and apply Lemma~\ref{lem:expansion-Lst-clt}
  \begin{align*}
    \int_{K_\Theta} \LL_{s, \vt}[\xi] \df\nu &= {} 1 +i\fd \langle\vt, \mu_0\rangle- \fd(s-1)-\frac12\int_{K_\Theta} \vt^T\LL[ \Phi_K\cdot  \Phi_K^T \xi] \vt \df\nu \\
    {}&
    \quad  + \langle\vt, \mu_0\rangle \int_{K_\Theta} \inpd{\vt , \LL[ \Phi_K \log |\cF_K'|\xi]} \df\nu\\
    &\quad- \frac12\langle\vt, \mu_0\rangle ^2\int_{K_\Theta} \LL[\log |\cF_K'|^2 \xi] \df\nu +  O(\nvt^3 + \nvt^{\alpha-\eps})\\
    &= {} 1 +i\fd \langle\vt, \mu_0\rangle- \fd(s-1)-\frac12\vt^T \int_{K_\Theta} \vpsi\cdot \vpsi^T \xi \df\nu \cdot \vt \\
    {}& \quad +  O(\nvt^3 + \nvt^{\alpha-\eps}).
  \end{align*}  
  Moreover, by Lemma~\ref{lem:expansion-Lst-clt2}, for~$f\in H^\kappa$ we also have
  \begin{align*}
    (\LL_{s, \vt}-\LL)[f] &=i\inpd{\vt , \LL[\vpsi f]} + O((\nvt^{2}+\nvt^{\alpha-1-\eps})\|f\|_\kappa),
  \end{align*}
  in $H^\kappa$.  Thus, since $(\Id-\LL)^{-1}(\Id-\PP)$ is bounded on $H^\kappa$, and using~\eqref{sbsa} we have
  \begin{align*}
    \int_{K_\Theta} (\LL_{s, \vt}-\LL) & (\Id-\LL)^{-1}(\Id - \PP) (\LL_{s, \vt}-\LL)\xi\df\nu  \\
    = {}& -\int_{K_\Theta} \langle \vt, \vpsi \rangle (\Id-\LL)^{-1}(\Id-\PP)\LL[\langle \vt, \vpsi \rangle \xi]\df\nu + O(\nvt^{3} + \nvt^{\alpha-\eps})\\   
    = {}& -\vt^T \int_{K_\Theta} \vpsi \cdot \vchi^T \xi\df\nu\cdot \vt + O(\nvt^{3} + \nvt^{\alpha-\eps})    
  \end{align*}
  since~$(\Id-\LL)^{-1}(\Id-\PP)\LL[f\xi] = \xi\K[f]$ by the identities~$\PP\LL = \PP$ and~$\LL = \PP + \NN$.

  By the third expansion in~\cite[Theorem~2.6]{Kloeckner2019} with~$s=s_0(t)$, we have
  $$ 0 =  \int_{K_\Theta} (\LL_{s, \vt}-\LL) \xi \df\nu + \int_{K_\Theta} (\LL_{s, \vt}-\LL) (\Id-\LL)^{-1}(\Id - \PP) (\LL_{s, \vt}-\LL)\xi\df\nu + O(\nvt^3) $$
  and thus one obtains the approximation~\eqref{eq:approxs0-alpha-2} for~$s_0(\vt)$ with
  \begin{equation}
    \Sigma_0 = \frac1\fd \int_{K_\Theta} (\vpsi \vpsi^T + \vchi \vpsi^T + \vpsi \vchi^T) \xi \df \nu. \label{eq:expr-sigma-variant}
  \end{equation}
  Using the identity~$\LL[\xi(f\circ\cF_\Theta)]=\xi f$, one then proceeds as in~\cite[Lemma~7.4]{BettinDrappeau2022} to show that~$\Sigma_0$ can be written as in~\eqref{eq:def-sigmagam}.
\end{proof}

\begin{remark}
  We believe an expression similar to~\eqref{eq:def-sigmagam}, but involving~$\cF_c$ and its inverse branches, should hold under the assumptions~\ref{hyp:moments-cuspidal}. One way to prove this is to use the fact that~$(K, \cF_K)$ is obtained from~$(\dD, \cF_c)$ by induction on~$K$, and follow the arguments worked out in~\cite{PeneThomine2021}. However, following this proof would require quite more about the spectral properties of~$\CC$ (the issue is to show boundedness of the series~(3.4) in~\cite{PeneThomine2021}). An easier alternative way would be to argue that the leading eigenvalues of~$\LL_{s,\vt}$ and~$\CC_{s,\vt}$ equal~$1$ at the same point~$s=s_0(\vt)$, and check that the arguments from this section carry over (compare also the distributional point of view described in~\cite[Section~5]{PeneThomine2021}). In our applications, neither expression gives an advantage over the other for the computation of~$\Sigma_0$, so we content with~\eqref{eq:def-sigmagam}.
\end{remark}

\section{Meromorphic continuation of the quasi-inverse}
\label{sec:continuation-quasiinv}

We let $\dD'':=\dD\setminus (V\cup \partial K)$.
Given $f:\dD''_\Theta\to\C$, let
\begin{align}
  \|f\|_{(\kappa),\dD''_\Theta}&:=\sup_{a,b\in\Sigma,\atop b\neq \hat a}\sup_{\vth\in\Theta,\, x,x'\in\dD''\cap[a,b] \atop x\neq x'}\frac{|f(x,\vth)-f(x',\vth)|}{|x-x'|^\kappa},\label{adaca}\\
  \|f\|_{\kappa,\dD''_\Theta}&:=\|f\|_{\infty,\dD''_\Theta}+\|f\|_{(\kappa),\dD''_\Theta}\notag
\end{align}
and let $H^\kappa_{\dD_{\Theta}''}$ be the space of those functions satisfying $\|f\|_{\kappa,\dD''_\Theta} < \infty$. For $f\in H^\kappa_{\dD''_\Theta}$, we define its values on $\partial K\times\Theta$ by taking limits from within $K$.
Similarly to~\eqref{def1t} we also define 
\begin{align*}
  \|f\|_{1,\tau,\dD''_\Theta}&:=\|f\|_{\infty,\dD''_\Theta}+(1+|\tau|)^{-\kappa}\|f\|_{(\kappa),\dD''_\Theta}.
\end{align*}

For  $s\in\C$ with $\Re(s)>1$ and $\vt\in\R^d$, we define the following operator from $H^\kappa_{\dD''_\Theta}$ to $L^\infty(K_\Theta)$
\begin{align*}
  S_{\LL;s,\vt}[f](y,\vth_1)&:= \sum_{h\in\cH^*,\atop y\in\Dom(h)}e^{i \inpd{\vt, \vth_1  \vphi_h(y)}} |h'(y)|^s \cdot f(h(y), \vth_1 \vth_h(y) ),\qquad \forall (y,\vth_1)\in K_\Theta
\end{align*}
and define similarly $S_{\CC;s,\vt}:L^\infty(\dD''_\Theta)\to L^\infty(\dD''_\Theta)$ in the same way but with the sum ranging over $h\in\cH_c^*$. 

We need perturbation estimates for $\LL_{s,\vt}$ acting on the space of functions with the larger domain $\dD_{\Theta}''$.

\begin{lemma}\label{gvam}
  There exists $\delta>0$, such that for all $(s,\vt)\in\C\times \R^d$ with $\Re(s)>1-\delta$, we have 
  \begin{enumerate}
    \item For all $f\in H^\kappa_{\dD_{\Theta}''}$ we have that $\LL_{s,\vt}[f]\in H^\kappa$ and is holomorphic in $s$.
    \item For all $f\in H^\kappa_{\dD_{\Theta}''}$ and all $\vt\ll1$, $1-\delta<\Re(s)\ll1$, we have
    \begin{align}
      &\|\LL_{s,\vt}[f]\|_{1,\abs{\Im(s)}}\ll \|f\|_{1,\abs{\Im(s)},\dD''_\Theta},\notag\\
      &\|\LL_{s,\vt}[f]-\LL[f]\|_{\infty}\ll (\abs{s-1} + \nvt+\nvt^{\alpha-\eps})\|f\|_{\infty,\dD''_\Theta},\label{acdsaaa}\\
      &\|\LL_{s, \vt}-\LL\|_{\kappa,\dD''_\Theta} \ll \abs{s-1} + \nvt+\nvt^{\alpha-\eps}.\label{acdsaaa2}
    \end{align}
  \end{enumerate}
\end{lemma}
\begin{proof}
  The proof is obtained proceeding as in the proofs of~\eqref{asx1} and~Lemma~\ref{est:ly:t-part} (for $r=1$) and of~\eqref{fbsa}-\eqref{sbsa}. Notice that these results were proven for $\LL_{s,\vt}$ applied to functions $f\in H^\kappa$, but the same arguments carry over for $f\in H^\kappa_{\dD''_{\Theta}}$.
\end{proof}

In the prospect of future applications, notably that in~\cite{BettinEtAl3}, it will be useful to prove an analogous result for $\CC_{s,\vt}$ and $\CC^\sharp_{s,\vt}$, where we recall that $\CC^\sharp_{s,\vt}$ is introduced in Section~\ref{sec:moments-conditions}.
Due to the double-average over inverse branches in hypothesis~\ref{hyp:moments-cuspidal}, it is natural to work with norms that incorporate these averages. Given $f:\dD''_\Theta\to\C$, we let
\begin{align*}
  \|f\|^*_{\infty} &:= \sum_{h\in\cH_c}\|h'\|_{\infty}  \|f\|_{\infty, (\dD''\cap\im(h))_\Theta},\\
  \|f\|^*_{(\kappa)} &:= \sum_{h\in\cH_c}\|h'\|_{\infty}  \max_{a,b\in\Sigma,\atop b\neq \hat a}\|f\|_{(\kappa), (\dD''\cap\im(h)\cap[a,b])_\Theta}
\end{align*}
and define the corresponding norms  $\|f\|^*_{\kappa}$ and $\|f\|^*_{1,\tau}$ as above. Finally, we define a variation with the unaveraged infinity-norm
\begin{align*}
  \|f\|^\dagger_{\kappa} &:=\|f\|_{\infty}+\|f\|^*_{(\kappa)},\\ 
  \|f\|^\dagger_{1,\tau} &:= \|f\|_{\infty}+(1+|\tau|)^{-\kappa}\|f\|^*_{(\kappa)},
\end{align*}
and let $H^\dagger$ be the space of functions with $\|f\|^\dagger_{\kappa} < \infty$.
\begin{lemma}\label{gvam2}
  Assume~\ref{hyp:moments-cuspidal}.
  There exists $\delta,\mu>0$, such that for all $(s,\vt)\in\C\times \R^d$ with $\Re(s)>1-\delta$, we have 
  \begin{enumerate}
    \item For all $f\in H^{\dagger}$ we have that $\CC_{s,\vt}[f],\CC^\sharp_{s,\vt}[f]\in H^{\dagger}$ and are holomorphic in $s$.
    \item There exists $\rho\in(0,1)$, $t_0>0$ such that for all $\nvt\ll t_0$, $1-\delta<\Re(s)\ll1$, we have
    \begin{align}
\|\CC_{s,\vt}\|^\dagger_{1,\abs{\Im(s)}}\ll1,\qquad
\|(\CC^\sharp)^r_{s,\vt}\|^\dagger_{1,\abs{\Im(s)}}\ll \rho^r,\quad r\geq0.\label{eqcs}
    \end{align}
    Moreover, for all $f\in H^{\dagger}$
    \begin{align}\label{petu}
      &\|\CC_{s,\vt}[f]-\CC_{1,0}[f]\|^*_{\kappa},\|\CC^\sharp_{s,\vt}[f]-\CC^\sharp_{1,0}[f]\|^*_{\kappa}\ll (\abs{s-1} + \nvt^{\mu})\|f\|^\dagger_{\kappa}.
    \end{align}
  \end{enumerate}
\end{lemma}
\begin{proof}
  As in the previous Lemma, the same arguments used for $\LL_{s,\vt}$ apply to $\CC_{s,\vt}$ and $\CC^\sharp_{s,\vt}$ as well, with the only slightly delicate estimates being the second bound in~\eqref{eqcs} and~\eqref{petu}. The latter follows similarly to the proof of Lemma~\ref{lem:perturbation-LL1t} with the difference that when we bound the terms involving the H\"older seminorm of $f\circ h$, we use $\| e^{i \langle \vt, \vth_1\vphi_h \rangle}-1 \|_{\infty,\im(k)}  \ll \nvt^\mu \| \vphi_{h}\|_{\infty,\im(k)}^\mu $ for some small $\mu>0$, and bound the average of this as $k\in\cH_c$ by 
\begin{align*}
\sum_{k\in\cH_c}\|k'\|_\infty \|\vphi_h\|_{\infty,\im(k)}^\mu \ll \|h'\|_\infty^{-\mu/\alpha}
\end{align*}
which follows  from~\eqref{eq:moment-condition-cuspidal} and H\"older's inequality.
In particular, the ensuing sum contains the factor $\|h'\|_\infty^{\Re(s)+\kappa-\mu/\alpha} \ll\|h'\|_\infty$ for $\delta,\mu$ sufficiently small, as required for the seminorm $\|\cdot\|^*_{(\kappa)}$.

Finally, proceeding as in the proof of~\eqref{asx1}, using
Lemma~\ref{lem:convergence-H-CF}, we obtain
\begin{align}\label{dadca}
  \|(\CC^\sharp_{s,\vt})^r[f]\|_\infty \ll \rho^r\|f\|_\infty,\qquad r\geq0.
\end{align}
Moreover, grouping the compositions by their first branch and using~\eqref{eq:Hc-contraction}, the same argument of Lemma~\ref{est:ly:t-part} gives
\begin{align*}
  \|(\CC^\sharp_{s,\vt})^m[f]\|^*_{(\kappa)}\ll\rho^{m\kappa/2}\|f\|^*_{(\kappa)}
  +C_m(1+|\Im(s)|)^\kappa\|f\|_\infty,
\end{align*}
for some $C_m>0$. Taking $m$ sufficiently large and fixed, and combining this estimate with~\eqref{dadca} as in the proof of Proposition~\ref{prop:dolgopyat-estimates}, we obtain~\eqref{eqcs}.
\end{proof}

\begin{proposition}\label{prop:spectrum-Lsf}
  There exist $t_0,C,\mu,\delta>0$ such that for $\vt\in\R^d$ with $|\vt|\leq t_0$ we have
  \begin{enumerate}
    \item For all $f\in H^{\kappa}_{\dD_{\Theta}''}$ the map~$s\mapsto S_{\LL;s,\vt}[f]\in H^\kappa$ extends meromorphically to $\Re(s)>1-\delta$, with (at most) a unique simple pole at $s=s_0(\vt)$, with $s_0(\vt)$  as described in Section~\ref{lem:approx-partialder}. Moreover, for $|s-s_0(\vt)|\gg1$, $1-\delta<\Re(s)\ll1$ we have 
    \begin{align}\label{bso} 
      \| S_{\LL;s}[f]\|_{1,\abs{\Im(s)}} \ll (1+\abs{\Im(s)}^{C})\|f\| _{1,\abs{\Im(s)},\dD''_\Theta}.
    \end{align}
    The residue $R_{\LL;\vt}[f](y,\vth)$ at $s_0(\vt)$ is
    \begin{align*}
      R_{\LL;0}[f](y,\vth)&=\frac1{\fd}\Big(\int_{\dD_\Theta} f \df\nu \Big) \xi(y),&& (y,\vth)\in K_\Theta,
    \end{align*}
if $\vt=0$, with $\nu$ extended to $\dD_{\Theta}$ by the same formula~\eqref{fdn}, and otherwise it satisfies
    \begin{align}\label{fcacsaa}
      \|R_{\LL;\vt}[f]-R_{\LL;0}[f]\|_{\kappa,K_\Theta}\ll (\nvt+\nvt^{\alpha-\eps})\|f\|_{\kappa,\dD''_\Theta}.
    \end{align}
    \item Assume~\ref{hyp:moments-cuspidal}. For all $f\in H^\dagger$ the map~$s\mapsto S_{\CC;s,\vt}[f]\in H^\dagger$ extends meromorphically to $\Re(s)>1-\delta$, with (at most) a unique simple pole at $s=s_0(\vt)$. Moreover, for $|s-s_0(\vt)|\gg1$, $1-\delta<\Re(s)\ll1$ we have 
    \begin{align}\label{bsoc} 
      \| S_{\CC;s}[f]\|^\dagger_{1,\abs{\Im(s)}} \ll (1+\abs{\Im(s)}^{C})\|f\|^\dagger_{1,\abs{\Im(s)}}.
    \end{align}
    The residue $R_{\CC;\vt}[f](y,\vth)$ at $s_0(\vt)$ is
    \begin{align}\label{expadca}
      R_{\CC;0}[f](y,\vth)&=\frac1{\fd}\Big(\int_{\dD_\Theta} f \df\nu \Big) \varphi(y),&& (y,\vth)\in\dD''_\Theta
    \end{align}
    if $\vt=0$ and otherwise it satisfies
    \begin{align}\label{fcacsaa2}
      \|R_{\CC;\vt}[f]-R_{\CC;0}[f]\|^*_{\infty}\ll \nvt^\mu \|f\|^\dagger_{\kappa}.
    \end{align}

  \end{enumerate}

\end{proposition}
\begin{proof}
  We start with (1), for which the argument is a straightforward generalization of \emph{e.g.}~\cite[proof of Proposition~6.9]{BettinDrappeau2022}.
  By Lemma~\ref{gvam}, we have $\LL_{s,t}[f]\in H^\kappa$. Then, for $\Re(s)>1$ we have
  \begin{align}
    S_{\LL;s,\vt}[f]&=\sum_{n\geq 0}\LL_{s,t}^n[f]
    =\Big(f|_{K_\Theta}+ (\id-\LL_{s,t})^{-1}[\LL_{s,t}[f]]\Big).
  \end{align}
  For~$(s,\vt)$ in a neighborhood of~$(1,0)$ with~$\Re(s)>1$, proceeding as~\cite[eq.~(3.35)]{BaladiVallee2005} or~\cite[Proposition~6.9]{BettinDrappeau2022} using Proposition~\ref{prop:spectrum-Ls2} (i),
  we have that $\LL_{s,\vt}$ (acting on $H^\kappa$) can be decomposed as
  \begin{align}
    \LL_{s,t}=\lambda(s,\vt)\PP_{s, \vt}+\NN_{s,\vt},
  \end{align}
  where $\PP_{s, \vt},\NN_{s, \vt}$ are operators on $H^\kappa$ which are analytic in $s$ for $(s,\vt)$ in a neighborhood of $(1,0)$. Also, $\PP_{s, \vt}$ is of rank 1 and one has $\PP_{s, \vt}\NN_{s,\vt}=\NN_{s,\vt}\PP_{s, \vt}=0$, $\PP_{s, \vt}^2=\PP_{s, \vt}$ and $\spr(\NN_{s, \vt})<\rho$. It follows that 
  $$ (\id - \LL_{s, \vt})^{-1} = \frac{\PP_{s, \vt}}{1 - \lambda(s, \vt)}-\PP_{s, \vt} + (\id - \NN_{s, \vt})^{-1}. $$
  Thus 
  \begin{align}
    S_{\LL;s,\vt}[f]&=f|_{K_\Theta}+\Big( \frac{\PP_{s, \vt}\LL_{s,t}}{1 - \lambda(s, \vt)}-\PP_{s, \vt}\LL_{s,t}+(\id - \NN_{s, \vt})^{-1}\LL_{s,t}\Big)[f].
  \end{align}
  This provides the claimed analytic continuation in $s$ for $\abs{\Im(s)}\leq \tau_0$, $\Re(s)>1-\delta_0$, $\nvt<t_0$ for some $\tau_0,\delta_0,t_0>0$, with the possible simple pole as $s=s_0(\vt)$. Moreover, by the expression~\eqref{eq:def-eigenproj} for $\PP_{1,0}$ and~\eqref{eq:transferop-invar} we have
  \begin{align}\label{fspm}
    \PP_{1, 0}\LL_{1,0}[f]= \xi \cdot \int_{K_\Theta} \LL[f]\df \nu= \xi \cdot \int_{\dD_\Theta} f\df\nu.
  \end{align}
  and thus, by~\eqref{eq:partialder-lambda}, we get the claimed expression for the residue of $S_{\LL;s,0}[f]$ as $s_0(0)=1$.

  The analytic continuation for $\abs{\Im(s)}>\tau_0$ and the bound~\eqref{bso} follow by~\eqref{bdasc}, Proposition~\ref{prop:spectrum-Ls2} (ii), and Lemma~\ref{gvam}. Finally~\eqref{fcacsaa} follows from~\eqref{acdsaaa2},~\eqref{sbsa} and~\cite{Kloeckner2019}. This concludes (1).
    \medskip

  We now move to (2).  By Lemma~\ref{lem:decomp-Hc} for $\Re(s)>1$ we have
  \begin{align}
    \sum_{m\geq0}\CC_{s,\vt}^m= \sum_{m,\ell\geq0}(\CC^\sharp_{s,\vt})^m\LL_{s,\vt}^\ell ,
  \end{align}
  where $\LL_{s,\vt}^\ell$ is trivially extended to an operator on functions with domain $\dD_{\Theta}''$ by setting its image to be zero outside $K_\Theta$ if $\ell\geq1$.
  
For $f\in H^\dagger$, since $K_\Theta$ is covered by finitely many branches the same estimates as in
Lemma~\ref{gvam2} give $\|(\CC_{s,\vt}[f])|_{K_\Theta}\|_{1,\abs{\Im(s)}}  \ll \|f\|^\dagger_{1,\abs{\Im(s)}}.$
Thus, by~\eqref{eq:rel-LLs-CCs} and Lemma~\ref{gvam2} we have $\|\LL_{s,\vt}[f]\|_{1,\abs{\Im(s)}}  \ll \|f\|^\dagger_{1,\abs{\Im(s)}}$ and the same bound holds trivially for $\|\LL_{s,\vt}[f]\|_{1,\abs{\Im(s)}}^\dagger$ when $\LL_{s,\vt}[f]$ is extended to $\dD_{\Theta}''$ as above.
The result then follows by the spectral estimates used in (1) and Lemma~\ref{gvam2}, which implies $\sum_{m\geq0}(\CC^\sharp_{s,\vt})^m$ extends analytically to $\Re(s)>1-\delta$, save for the claims~\eqref{expadca}-\eqref{fcacsaa2} on $R_{\CC;\vt}$.

By~\eqref{eq:transferop-invar} and~\eqref{eq:partialder-lambda}, we have
\begin{align*}
  R_{\CC;0}[f]=\frac1{\fd}\Big(\int_{\dD_\Theta}f\df\nu\Big) \sum_{m\ge0}(\CC^\sharp)^m[\xi\chi_K].
\end{align*}
which yields the desired expression for $R_{\CC;0}[f]$ since by~\eqref{eq:rel-sumxi-phi}
  \begin{align*}
    \sum_{m\geq0}(\CC^\sharp_{1,0})^m[\xi\chi_K] = \xi\chi_K + \sum_{m\geq1}(\CC^\sharp_{1,0})^m[\xi\chi_K] = \xi\chi_K + \varphi(1-\chi_K) = \varphi,
  \end{align*}
  as $\varphi=\xi$ on $K$.
  
   Finally, we prove~\eqref{fcacsaa2}. Letting $A_{s,\vt}=(\Id-\CC^\sharp_{s,\vt})^{-1}$, the decomposition above gives
\begin{align*}
  R_{\CC;\vt}[f]  =-\frac{
    A_{s_0(\vt),\vt}[\chi_K\PP_{s_0(\vt),\vt}\LL_{s_0(\vt),\vt}[f]]}{\partial_s\lambda(s_0(\vt),\vt)},
\end{align*}
for $\vt$ sufficiently close to zero. By the perturbation estimates for $\LL_{s,\vt}$, it suffices to estimate $\|(A_{s,\vt}-A_{1,0})[f]\|_\infty^*$ and $\|(\LL_{s,\vt}-\LL)[f]\|_{\kappa}$ for $f\in H^\dagger$ for $s,\vt$ in a neighbourhood of $(1,0)$.
For the latter estimate, the same argument as in the proof of~\eqref{petu}, grouping
the inverse branches by their first cuspidal branch, gives $\|\LL_{s,\vt}[f]-\LL[f]\|_\kappa \ll (|s-1|+\nvt^\mu)\|f\|^\dagger_\kappa$, decreasing $\mu>0$ if needed. For the former, we have
\begin{align*}
  (A_{s,\vt}-A_{1,0})[f]  =g+A_{s,\vt}\CC^\sharp_{s,\vt}[g].
\end{align*}
with $g:=(\CC^\sharp_{s,\vt}-\CC^\sharp)A_{1,0}[f].$ By~\eqref{petu} we have $\|g\|_\infty^*
  \ll (|s-1|+\nvt^\mu)\|f\|^\dagger_\kappa$ and so it suffices to bound the second summand. Applying Cauchy-Schwarz we have
\begin{align*}
  \|\CC^\sharp_{s,\vt}[g]\|_\infty  \ll \sum_{h\in\cH_c}\|h'\|_\infty^\sigma \|g\|_{\infty,\im(h)}\ll (\|g\|_\infty^*)^{1/2}\|g\|_\infty^{1/2}\ll (|s-1|+\nvt^\mu)^{1/2}\|f\|^\dagger_\kappa,
\end{align*}
by~\eqref{eqcs}-\eqref{petu}. By~\eqref{dadca} we have $\|A_{s,\vt}\|_\infty\ll1$ and thus we deduce that the same bound holds for $\|A_{s,\vt}\CC^\sharp_{s,\vt}[g]\|_\infty$, completing the proof.
\end{proof}

\begin{remark}\label{unifo}
  The implicit constants and the parameters $t_0,\delta$ in Proposition~\ref{prop:spectrum-Lsf} can depend on the group $G$, on the choice of the fundamental domains $R$, and on~$\Theta$; for item (1), they may also depend on the values of the series~\eqref{eq:moment-condition}--\eqref{eq:holder-condition}, while for item (2) they may depend on the values of~\eqref{eq:moment-condition-cuspidal}--\eqref{eq:holder-condition-cuspidal} (the expression for $s_0(\vt)$ in Lemma~\ref{lem:estim-s0-subclt} may depend on~\eqref{eq:moment-condition-cuspidal2}). There is no further dependence on~$(\vphi_h)_h$ besides for these bounds. 
\end{remark}

\begin{remark}\label{comple}
  Proposition~\ref{prop:spectrum-Lsf} holds also in the case when~$\vt$ and the values of~$\vphi_h$ are in~$\C^d$ rather than in~$\R^d$.
  This follows by identical arguments, requiring the hypotheses to be adapted as follows:
  for (1) it suffices to replace the conditions~\eqref{eq:moment-condition}--\eqref{eq:holder-condition} of~\ref{hyp:moments} by the stronger condition (finite exponential moments)
  \begin{equation}\label{com1}
    \sum_{h\in\cH} \| h'\|_\infty (1+\|\varpi \vphi_{h} \|^\lambda_{(\kappa_0)} )e^{\varpi \|\vphi_{h}\|_{\infty} } \leq C
  \end{equation}
  for some $\lambda, \varpi, C>0$ (see also~\cite[eq.~(2.5)]{BaladiVallee2005}).
  The conclusion of Proposition~\ref{prop:spectrum-Lsf}, item (1), then holds, uniformly in~$\vphi$ and~$\varpi$ for $\nvt<t_0\varpi$, for some $t_0>0$, so long as~$\lambda$ and~$C$ are fixed, with $\nvt+\nvt^{\alpha-\eps}$ in~\eqref{fcacsaa} replaced by $\nvt/\varpi+(\nvt/\varpi)^{\alpha-\eps}$. For (2) it suffices instead to replace the conditions~\eqref{eq:moment-condition-cuspidal}-\eqref{eq:holder-condition-cuspidal} of~\ref{hyp:moments-cuspidal} with
  \begin{align}\label{com2}
    \ssum{h_1, h_2 \in \cH_c}  \|h_1'\|_{\infty}\|h_2'\|_\infty (1+\| \varpi\vphi_{h_1}|_{\im(h_2)} \|_{(\kappa_0)}^\lambda)e^{\varpi \|\vphi_{h_1}\|_{\infty} } \leq C.
  \end{align}
The same uniformity holds for item (2), with $\nvt^\mu$ in~\eqref{fcacsaa2} replaced by $(\nvt/\varpi)^\mu$.
Notice that these conditions only guarantee the absolute convergence of the series $  S_{\LL}$ and $S_{\CC}$  in a half-plane sufficiently far to the right, cf.~\cite{BaladiVallee2005}.  Because the proof is essentially identical to the real case, we omit the details.
\end{remark}

\section{Meromorphic continuation of the generating series at cusps} \label{sec:est:mgf}

In this section, we fix~$\cb\in C(G)$. We let
$$ \psi : \dD \to \C $$
be a fixed smooth map whose compact support avoids~$1$, and recall the definitions~\eqref{eq:def-Y} and~\eqref{eq:def-qx}. In particular, by Lemma~\ref{lem:conv-sum-q}, we have
$$ \sum_{x\in G\cb\setminus\{1\}} \abs{\psi(x)} q(x)^{-2-\eps} < \infty, \qquad (\eps > 0). $$
We denote also
\begin{equation}
  \psi_s(x, \vth) := \begin{cases} \abs{\frac{x-1}2}^{-2s}\psi(x), & (\vth = 1,\ x\neq1), \\ 0 & (\text{otherwise}). \end{cases}\label{eq:def-psis}
\end{equation}
Finally, we recall that the finite set $W\subset K$ is defined in~\eqref{defW}.

We are interested in the complex-analytic behaviour of the series
\begin{equation}
  \fS(s, \vt) := \sum_{x\in G\cb\setminus \{1\}} \psi(x) q(x)^{-2s} e^{i\inpd{\vt, \qf(x)}},\label{eq:def-fS}
\end{equation}
where $\qf:C(G) \to \R^d$ satisfies the relation
\begin{align} \label{qmfr}
  \qf(h(y)) = \vth_h(y)^{-1}(\vphi_h(y) + \qf(y)),\qquad y\in C(G)\cap\Dom(h),
\end{align}
for all $h\in\cH$ (this is special case of~\eqref{eq:qmf}).

The following lemma expresses~$\fS(s, \vt)$ in terms of the quasi-inverse~$S_{\LL;s,\vt}$ which we analyzed in the previous section.
\begin{lemma}\label{acdw|}
  For~$\Re(s)>1$ and any $\vt\in\R^d$, we have
  \begin{equation}
    \begin{aligned}
      \fS(s, \vt) ={}& \ssum{x\in G\cb \cap V} \Upsilon(x)^{-2s} e^{i\inpd{\vt, \qf(x)}} \psi_s(x, 1) \\
      {}& + \ssum{x\in G\cb \cap W \\ \vth\in \Theta} \Upsilon(x)^{-2s} e^{i \inpd{\vt, \vth \qf(x)}} S_{\LL;s,\vt}[\psi_s](x, \vth).
    \end{aligned}\label{eq:fS-expr-LL}
  \end{equation}
\end{lemma}
\begin{proof} 
  Define the sets
  $$ X_V := G\cb \cap V, \qquad X_W := G\cb \cap W, \qquad X_K := G\cb \cap K, \qquad X_O := G\cb \setminus (X_V \cup X_W). $$
  For any~$f\in L^{\infty}(\dD_\Theta)$ and $*\in\{V,W,K,O\}$,  we define the functional
  $$ T_\ast(f) := \ssum{(x, \vth) \in X_\ast\times\Theta} f(x, \vth) \Upsilon(x)^{-2s} e^{i\inpd{\vt, \vth\qf(x)}}. $$
  We first show that
  \begin{equation}
    T_K = T_W + T_K\LL_{s, \vt} \qquad \text{on } L^\infty(K_\Theta).\label{eq:SK-SWSKL}
  \end{equation}
  Consider~$x\in X_K$. If~$\cB(x) \in V$, then~$x\in W$, and otherwise we have~$\cF_K(x) \in K$. Thus we have a splitting
  $$ T_K(f) = T_W(f) + T(f), $$
  where
  $$ T(f) = \ssum{(x,\vth)\in X_K\times \Theta\\ \cF_K(x) \in K}  f(x, \vth) \Upsilon(x)^{-2s} e^{i\inpd{\vt, \vth\qf(x)}}. $$
  We wish to show that~$T(f) = T_K(\LL_{s,\vt}[f])$ for~$f \in L^\infty(K_\Theta)$. We sort the sum according to the value of~$y = \cF_K(x) \in K$. The condition~$x\in G\cb$ is equivalent to~$y\in G\cb$; and given~$y$, the associated values of~$x$ are given by~$h(y)$ for~$h\in \cH$ and~$y\in \Dom(h)$. We thus get
  $$ T(f) = \sum_{y\in X_K} \sum_{\vth\in \Theta} \ssum{h\in \cH \\ y \in \Dom(h)} f(h(y), \vth) \Upsilon(h(y))^{-2s} e^{i\inpd{\vt, \vth\qf(h(y))}}. $$
  Let~$w\in \Lambda_K(1)$ be the word which determines~$h = h_{w, 1}$ in the notation~\eqref{eq:def-hw}, and let~$g := g_w$, so that~$h(y) = g_w y$. Then~$\abs{h'(y)} = \abs{j(g, y)}^{-2}$, so that by~\eqref{eq:denom-cocycle}, we get
  $$ \Upsilon(h(y)) = \abs{h'(y)}^{-1/2} \Upsilon(y). $$
  Recalling~\eqref{qmfr} and changing variables in the~$\vth$-sum, we then obtain
  $$ T(f) = \sum_{y\in X_K}  \Upsilon(y)^{-2s}  \sum_{\vth\in \Theta} e^{i\inpd{\vt, \vth\qf(y)}} \ssum{h\in \cH \\ y \in \Dom(h)} \abs{h'(y)}^s e^{i\inpd{\vt, \vth \vphi_h(y)}}f(h(y), \vth\vth_h(y)) . $$
  We recognize the inner sum as~$\LL_{s,\vt}[f](y, \vth)$, which gives~$T(f) = T_K(\LL_{s,\vt}[f])$ as required.

  We now consider~$f\in L^\infty(\dD_\Theta)$; note that~$\LL_{s, \vt}[f] \in L^\infty(K_\Theta)$.
  We observe that, since~$x\in X_O$ is equivalent to~$\cF_K(x) \in K$ for $x\in G\cb$, arguing as above we have
  $$ T_O(f) = T_K(\LL_{s, \vt}[f]). $$
  Recall the definition~\eqref{eq:def-fS}. By~\eqref{eq:def-qx} and~\eqref{eq:def-psis}, we have
  $$ \fS(s, \vt) = \sum_{(x,\vth)\in G\cb\times\Theta} \psi_s(x, \vth) \Upsilon(x)^{-2s} e^{i\inpd{\vt, \vth\qf(x)}} $$
  with~$\psi_s$ defined in~\eqref{eq:def-psis}. We split the sum according to the partition~$G\cb = X_V \cup X_W \cup X_O$. This yields
  $$ \fS(s, \vt) = [T_V + T_W + T_O](\psi_s) = T_V(\psi_s) + T_W(\psi_s) + T_K(\LL_{s,\vt}[\psi_s]). $$
  Since~$1\not\in\supp\psi$, we have~$\psi_s\in L^\infty(\dD_\Theta)$, and therefore by Proposition~\ref{prop:Lst-bounded} we have
  $$ \LL_{s,\vt}[\psi_s] \in L^\infty(K_\Theta). $$
  Iterating the relation~\eqref{eq:SK-SWSKL}, we find for all~$J\geq 2$
  $$ T_K(\LL_{s,\vt}[\psi_s]) = T_W\Big(\sum_{j=1}^{J-1} \LL_{s,\vt}^j[\psi_s]\Big) + T_K(\LL_{s,\vt}^J[\psi_s]). $$
  Since $\Re(s)>1$, by~\eqref{asx1} we deduce that~$\|\LL_{s,\vt}^J[\psi_s]\|_{\infty, K_\Theta} \ll \rho^J$ as~$J\to \infty$. By Lemma~\ref{lem:conv-sum-q}, we deduce~$T_K(\LL_{s,\vt}^J[\psi_s])\ll \rho^J$ as~$J\to \infty$, and therefore
  $$ \fS(s, \vt) = T_V(\psi_s) + T_W\big(\sum_{j\geq0}\LL_{s,\vt}^j[\psi_s]\big), $$
  which is the claimed expression.
\end{proof}

By Proposition~\ref{prop:spectrum-Lsf} and Lemma~\ref{acdw|} we can immediately deduce the meromorphic continuation for $\fS(s, \vt)$ to the left of $\Re(s)=1$.

\begin{proposition}\label{prop:anaprop-genseries}
Assume $\max_{x\in V\cup W}|\qf(x)|\ll 1$.   For some~$t_0, \delta_0>0$, the following holds whenever~$\nvt \leq t_0$.
  \begin{enumerate}
    \item The map~$s\mapsto \fS(s, \vt)$, initially defined and analytic on~$\Re(s)>1$, extends to a meromorphic function of~$s$ in the strip
    \begin{equation}
      \Re(s) \geq 1 - \delta_0.\label{eq:s-half-plane}
    \end{equation}
    \item The map~$s\mapsto \fS(s, \vt)$ is analytic everywhere except, possibly, at a single point~$s_0(\vt)$, with $s_0(\vt)$ as in Lemma~\ref{lem:existence-s0}, where it has (at most) a simple pole.
    \item If $\vt=0$ the possible pole at $s_0(0)=1$ has residue
    \begin{align*}
      \frac{1}{\fd} \int_{\dD} \frac{4 \psi(y)}{|y-1|^2} \frac{\df s(y)}{\fm}\cdot \ssum{x\in G\cb\setminus V\atop \cB(x)\in V} \Upsilon(x)^{-2} \xi(x) =
      \frac{1}{\pi \vol(G \backslash \D)} \int_{\dD} \psi(y) \frac{2\df s(y)}{|y-1|^2}
    \end{align*}
    and in general the residue is equal to this expression up to $O(\nvt+\nvt^{\alpha-\eps})$, where the implicit constant may depend also on $\max_{x\in W}|\qf(x)|$
    \item For some~$C>0$ and all~$s$ in~\eqref{eq:s-half-plane} with~$\abs{s - s_0(\vt)} \gg 1$ and $\Re(s)\ll1$, we have
    $$ \abs{\fS(s, \vt)} \leq (2 + \abs{\Im(s)})^C. $$
  \end{enumerate}
\end{proposition}

\begin{proof}[Proof of Proposition~\ref{prop:anaprop-genseries}]
  We start with the expression~\eqref{eq:fS-expr-LL}.
  Since~$1$ is bounded away from the support of~$\psi$ by assumption, and~$V$ is a finite set, the first term on the right-hand side is an entire function of~$s$ bounded in vertical strips.
  Moreover, we have~$\psi_s \in H^{\kappa}_{\dD_{\Theta}''}$, with~$\norm{\psi_s}_{1,\abs{\Im(s)},\dD''_\Theta} \ll 1$ uniformly for~$s$ in vertical strips with $\Re(s)\ll1$. Thus, since~$W$ is a finite set, the claim follows from Proposition~\ref{prop:spectrum-Lsf}.
  The second expression in (3) is obtained by partial summation from the definition~\eqref{eq:def-fS} using Lemma~\ref{lem:psi-totalmass}.
\end{proof}

\subsection{Estimation of the characteristic function}
\label{sec:main-results}

We are now ready to state and prove our main technical result, from which Theorem~\ref{th:limitlaw:quasi} will follow, and for which we explicitly recall the hypothesis we made throughout the paper. 

\begin{theorem}\label{mtf}
  Let~$G\subset \SU(1, 1)$ satisfy the conditions in Section~\ref{sec:disk-model}, and~$R$ be an ideal fundamental domain as in Section~\ref{sec:fundamental-domain}.
  Let~$\psi:\dD \to \R_+$ be a smooth, not identically zero test function with~$1 \not\in \supp\psi$.
  Let~$\cb$ be a cusp of~$G$, and let~$q(x)$ be the associated denominator function from~\eqref{eq:def-qx}.
  Let~$(\vth_g)_{g\in G}$ be a multiplier system satisfying the conditions in Section~\ref{sec:skew-products}
  and~$(\vphi_h)_{h\in \cH}$ a family of maps as in Section~\ref{cfn} satisfying hypothesis~\ref{hyp:moments}.
  Finally, let~$\qf:C(G) \to \R^d$ satisfy the relations~\eqref{qmfr}.
  Then the conclusions of Theorem~\ref{th:limitlaw:quasi} hold as stated with the following parameters:
  \begin{enumerate}[(i)]
    \item the random variable~$\Phi$ is given almost everywhere by~$\Phi = \Phi_K(x, \vth)$ on the probability space~$K_\Theta$ with measure~$\frac{\xi(x)}{\fm} \df s(x) \df \vth$;
    \item the normalizing quantity~$\fd$ is given by~\eqref{eq:def-fd};
    \item the mean-value vector and covariance matrix are given by~$\mu_\Phi=2\mu_0$ and~$\Sigma_\Phi=2\Sigma_0$, with $\mu_0$ and $\Sigma_0$ as in~\eqref{eq:def-muphi} and~\eqref{eq:def-sigmagam}.
  \end{enumerate}
\end{theorem}

For instance, we have
$$ \E(\e^{i\inpd{\vt, \Phi}}) = \int_{K_\Theta} \e^{i\inpd{\vt, \Phi_K(x, \vth)}} \xi(x)\frac{\df s(x) \df \vth}{\fm}. $$

\begin{proof}
  This is proved by the same arguments as~\cite[p.~1410, Proposition~8.1]{BettinDrappeau2022}, which is a variant of~\cite[pp.~370--372]{BaladiVallee2005} (see also~\cite{Cesaratto2009,LeeSun2019a}), using Proposition~\ref{prop:anaprop-genseries} and where the required bound for the number of denominators~$q(x)$ in short intervals follows from Lemma~\ref{lem:psi-totalmass}.
\end{proof}

\begin{remark}\label{fire}
  The same considerations in Remarks~\ref{unifo} and~\ref{comple} on Proposition~\ref{prop:spectrum-Lsf}, item (1), hold for Proposition~\ref{prop:anaprop-genseries} and Theorem~\ref{mtf}, with the only difference that the implied constants and $t_0$ may also depend on an upper bound for $\max_{x\in V\cup W}|\qf(x)|$, or for
  $\varpi\max_{x\in V\cup W}|\qf(x)|$ in the complex case.
\end{remark}

\begin{remark}
  We refer to~\cite{Goueezel2010} for conditions under which Birkhoff sums of a Gibbs-Markov map converge to a stable law. We expect a statement similar to~\cite[Theorem~1.5]{Goueezel2010} to hold for~$\qf(x)$, with the constant $c$ replaced by the map~$\log|\cF'(x)|$, see~\cite[Corollary~3.2]{BettinDrappeau2022}.
\end{remark}

\bibliographystyle{alpha}

\bibliography{LLcfrac-ext}

\begin{thebibliography}{{Sag}23}

\bibitem[Aar97]{Aaronson1997}
J.~Aaronson.
\newblock {\em An introduction to infinite ergodic theory}, volume~50 of {\em
  Math. Surv. Monogr.}
\newblock Providence, RI: American Mathematical Society, 1997.

\bibitem[AD01]{AaronsonDenker2001}
J.~Aaronson and M.~Denker.
\newblock Local limit theorems for partial sums of stationary sequences
  generated by {G}ibbs-{M}arkov maps.
\newblock {\em Stochastics Dyn.}, 1(2):193--237, 2001.

\bibitem[AF91]{AdlerFlatto1991}
R.~Adler and L.~Flatto.
\newblock Geodesic flows, interval maps, and symbolic dynamics.
\newblock {\em Bull. Amer. Math. Soc.}, 25(2):229--334, 1991.

\bibitem[AKU26]{AbramsEtAl2026}
A.~Abrams, S.~Katok, and I.~Ugarcovici.
\newblock Reduction theory for {Fuchsian} groups with cusps.
\newblock {\em Geom. Dedicata}, 220(4):29, 2026.

\bibitem[AMU16]{ArtigianiEtAl2016}
M.~Artigiani, L.~Marchese, and C.~Ulcigrai.
\newblock The {Lagrange} spectrum of a {Veech} surface has a {Hall} ray.
\newblock {\em Groups Geom. Dyn.}, 10(4):1287--1337, 2016.

\bibitem[AMU20]{ArtigianiEtAl2020}
M.~Artigiani, L.~Marchese, and C.~Ulcigrai.
\newblock Persistent {Hall} rays for {Lagrange} spectra at cusps of {Riemann}
  surfaces.
\newblock {\em Ergodic Theory Dynam. Systems}, 40(8):2017--2072, 2020.

\bibitem[AS19]{ArnouxSchmidt2019}
P.~Arnoux and T.~A. Schmidt.
\newblock Natural extensions and {Gauss} measures for piecewise homographic
  continued fractions.
\newblock {\em Bull. Soc. Math. France}, 147(3):515--544, 2019.

\bibitem[Bal00]{Baladi2000}
V.~Baladi.
\newblock {\em Positive transfer operators and decay of correlations},
  volume~16 of {\em Advanced Series in Nonlinear Dynamics}.
\newblock World Scientific Publishing Co., Inc., River Edge, NJ, 2000.

\bibitem[BD]{BettinEtAl3}
S.~Bettin and S.~Drappeau.
\newblock Distribution of trigonometric sums formed from coefficients of
  {$GL(2)$} {$L$}-series.

\bibitem[BD22]{BettinDrappeau2022}
S.~Bettin and S.~Drappeau.
\newblock Limit laws for rational continued fractions and value distribution of
  quantum modular forms.
\newblock {\em Proc. London Math. Soc.}, 125(6):1377--1425, 2022.

\bibitem[BD25]{BettinDrappeau2025}
S.~Bettin and S.~Drappeau.
\newblock On quantum modular forms of non-zero weights.
\newblock {\em Compos. Math.}, 161(11):3044--3087, 2025.

\bibitem[BDL]{BettinEtAla}
S.~Bettin, S.~Drappeau, and J.~Lee.
\newblock Central values of additive twists of automorphic {$L$}-series.

\bibitem[Bro96]{Broise1996}
A.~Broise.
\newblock Transformations dilatantes de l'intervalle et théorèmes limites.
\newblock {\em Astérisque}, 238:1--109, 1996.

\bibitem[BS79]{BowenSeries1979}
R.~Bowen and C.~Series.
\newblock Markov maps associated with {Fuchsian} groups.
\newblock {\em Inst. Hautes études Sci. Publ. Math.}, 50:153--170, 1979.

\bibitem[BV05a]{BaladiVallee2005a}
V.~Baladi and B.~Vall\'{e}e.
\newblock Exponential decay of correlations for surface semi-flows without
  finite {M}arkov partitions.
\newblock {\em Proc. Amer. Math. Soc.}, 133(3):865--874, 2005.

\bibitem[BV05b]{BaladiVallee2005}
V.~Baladi and B.~Vallée.
\newblock Euclidean algorithms are {Gaussian}.
\newblock {\em J. Number Theory}, 110(2):331--386, 2005.

\bibitem[Ces09]{Cesaratto2009}
E.~Cesaratto.
\newblock A note on ``{Euclidean} algorithms are {Gaussian}'' by {V}. {Baladi}
  and {B}. {Vall{\'e}e}.
\newblock {\em J. Number Theory}, 129(10):2267--2273, 2009.

\bibitem[CN25]{ConstantinescuNordentoft2025}
P.~Constantinescu and A.~C. Nordentoft.
\newblock Non-vanishing of geodesic periods of automorphic forms.
\newblock {\em Geom. Funct. Anal.}, 35(4):1108--1146, 2025.

\bibitem[Con22]{Constantinescu2022}
P.~Constantinescu.
\newblock Distribution of modular symbols in {{\(\mathbb{H}^3\)}}.
\newblock {\em Int. Math. Res. Not.}, 2022(7):5425--5465, 2022.

\bibitem[DN26]{DrappeauNordentoft2022}
S.~Drappeau and A.~C. Nordentoft.
\newblock Central values of additive twists of {Maaß} forms {$L$}-series.
\newblock To appear at Amer. J. Math., 2026.

\bibitem[Dol98]{Dolgopyat1998}
D.~Dolgopyat.
\newblock On {Decay} of {Correlations} in {Anosov} {Flows}.
\newblock {\em Ann. of Math. (2)}, 147(2):357--390, 1998.

\bibitem[Eva10]{Evans2010}
L.~C. Evans.
\newblock {\em Partial differential equations}, volume~19 of {\em Grad. Stud.
  Math.}
\newblock Providence, RI: American Mathematical Society (AMS), 2nd ed. edition,
  2010.

\bibitem[Gou10]{Goueezel2010}
S.~Gou{\"e}zel.
\newblock Characterization of weak convergence of {Birkhoff} sums for
  {Gibbs}-{Markov} maps.
\newblock {\em Isr. J. Math.}, 180:1--41, 2010.

\bibitem[Hen93]{Hennion1993}
H.~Hennion.
\newblock On a spectral theorem and its application to {Lipschitzian} kernels.
\newblock {\em Proc. Amer. Math. Soc.}, 118(2):627--634, 1993.

\bibitem[Hen94]{Hensley1994}
D.~Hensley.
\newblock The {Number} of {Steps} in the {Euclidean} {Algorithm}.
\newblock {\em J. Number Theory}, 49(2):142--182, 1994.

\bibitem[Iwa97]{Iwaniec1997}
H.~Iwaniec.
\newblock {\em Topics in classical automorphic forms}, volume~17 of {\em
  Graduate {Studies} in {Mathematics}}.
\newblock American Mathematical Society, Providence, RI, 1997.

\bibitem[Iwa02]{Iwaniec2002}
H.~Iwaniec.
\newblock {\em Spectral methods of automorphic forms.}, volume~53 of {\em Grad.
  Stud. Math.}
\newblock Providence, RI: American Mathematical Society (AMS); Madrid: Revista
  Matem{\'a}tica Iberoamericana, 2nd ed. edition, 2002.

\bibitem[Kat92]{Katok1992}
S.~Katok.
\newblock {\em Fuchsian groups}.
\newblock Chicago: The University of Chicago Press, 1992.

\bibitem[Kat95]{Kato1995}
T.~Kato.
\newblock {\em Perturbation theory for linear operators}.
\newblock Classics in Mathematics. Springer-Verlag, Berlin, 1995.

\bibitem[Khi63]{Khintchine1963}
A.~Y. Khintchine.
\newblock {\em Continued fractions}.
\newblock Translated by Peter Wynn. P. Noordhoff, Ltd., Groningen, 1963.

\bibitem[KL99]{KellerLiverani1999}
G.~Keller and C.~Liverani.
\newblock Stability of the spectrum for transfer operators.
\newblock {\em Ann. Sc. Norm. Super. Pisa, Cl. Sci., IV. Ser.}, 28(1):141--152,
  1999.

\bibitem[Klo19]{Kloeckner2019}
B.~R. Kloeckner.
\newblock Effective perturbation theory for simple isolated eigenvalues of
  linear operators.
\newblock {\em J. Operator Theory}, 81(1):175--194, 2019.

\bibitem[KU17]{KatokUgarcovici2017}
S.~Katok and I.~Ugarcovici.
\newblock Structure of attractors for boundary maps associated to {Fuchsian}
  groups.
\newblock {\em Geom. Dedicata}, 191:171--198, 2017.

\bibitem[Kum80]{Kumagai1980}
S.~Kumagai.
\newblock An implicit function theorem: {Comment}.
\newblock {\em J. Optim. Theory Appl.}, 31:285--288, 1980.

\bibitem[LS19]{LeeSun2019a}
J.~Lee and H.-S. Sun.
\newblock Another note on ``{Euclidean} algorithms are {Gaussian}'' by {V}.
  {Baladi} and {B}. {Vall{\'e}e}.
\newblock {\em Acta Arith.}, 188(3):241--251, 2019.

\bibitem[LS25]{LeeSun2019}
J.~Lee and H.-S. Sun.
\newblock Dynamics of continued fractions and distribution of modular symbols.
\newblock {\em J. Eur. Math. Soc. (JEMS)}, 27(9):3527--3582, 2025.

\bibitem[Mar18]{Marchese2018}
L.~Marchese.
\newblock Transfer operators and dimension of bad sets for non-uniform
  {Fuchsian} lattices.
\newblock arXiv preprint, 2018.

\bibitem[May91]{Mayer1991}
D.~H. Mayer.
\newblock The thermodynamic formalism approach to {Selberg}'s zeta function for
  {PSL}(2,{{\({\mathbb{Z}})\)}}.
\newblock {\em Bull. Am. Math. Soc., New Ser.}, 25(1):55--60, 1991.

\bibitem[Mor97]{Morita1997}
T.~Morita.
\newblock Markov systems and transfer operators associated with cofinite
  {Fuchsian} groups.
\newblock {\em Ergodic Theory Dyn. Syst.}, 17(5):1147--1181, 1997.

\bibitem[Nor21]{Nordentoft2021}
A.~C. Nordentoft.
\newblock Central values of additive twists of cuspidal {L}-functions.
\newblock {\em J. Reine Angew. Math.}, 2021(776):255--293, 2021.

\bibitem[PR04]{PetridisRisager2004}
Y.~N. Petridis and M.~S. Risager.
\newblock Modular symbols have a normal distribution.
\newblock {\em Geom. Funct. Anal.}, 14(5):1013--1043, 2004.

\bibitem[PR18]{PetridisRisager2018}
Y.~N. Petridis and M.~S. Risager.
\newblock Arithmetic statistics of modular symbols.
\newblock {\em Invent. Math.}, 212:997--1053, 2018.

\bibitem[PT21]{PeneThomine2021}
F.~P{\`e}ne and D.~Thomine.
\newblock Probabilistic potential theory and induction of dynamical systems.
\newblock {\em Ann. Inst. Henri Poincar{\'e}, Probab. Stat.}, 57(3):1736--1767,
  2021.

\bibitem[PY98]{PollicottYuri1998}
M.~Pollicott and M.~Yuri.
\newblock {\em Dynamical systems and ergodic theory}, volume~40 of {\em Lond.
  Math. Soc. Stud. Texts}.
\newblock Cambridge: Cambridge University Press, 1998.

\bibitem[{Sag}23]{sagemath}
T.~{Sage Developers}.
\newblock {\em {S}ageMath, the {S}age {M}athematics {S}oftware {S}ystem
  ({V}ersion 9.5)}, 2023.

\bibitem[Ser81]{Series1981}
C.~Series.
\newblock The infinite word problem and limit sets in {Fuchsian} groups.
\newblock {\em Ergodic Theory Dynam. Systems}, 1:337--360, 1981.

\bibitem[Ser85]{Series1985}
C.~Series.
\newblock The modular surface and continued fractions.
\newblock {\em J. Lond. Math. Soc. (2)}, 31:69--80, 1985.

\bibitem[Sta04]{Stadlbauer2004}
M.~Stadlbauer.
\newblock The return sequence of the {Bowen}-{Series} map for punctured
  surfaces.
\newblock {\em Fund. Math.}, 182(3):221--240, 2004.

\bibitem[Tuk73]{Tukia1973}
P.~Tukia.
\newblock Extension of boundary homeomorphisms of discrete groups of the unit
  disk.
\newblock {\em Ann. Acad. Sci. Fenn. Ser A I}, 548:16, 1973.

\bibitem[Val03]{Vallee2003}
B.~Vallée.
\newblock Dynamical analysis of a class of {Euclidean} algorithms.
\newblock {\em Theoret. Comput. Sci.}, 297(1-3):447--486, 2003.

\bibitem[Zag10]{Zagier2010}
D.~Zagier.
\newblock Quantum modular forms.
\newblock In {\em Quanta of maths}, volume~11 of {\em Clay Math. Proc.}, pages
  659--675. Amer. Math. Soc., Providence, RI, 2010.

\end{thebibliography}

\end{document}